\documentclass[10pt]{amsart}

\usepackage{amsmath,amsfonts,amssymb,amsthm,amscd,epsfig,color,euscript,mathrsfs, texdraw,multicol,rotating,pifont}
\usepackage[vcentermath,enableskew]{youngtab}
\usepackage[all]{xy}

\newtheorem{theorem}{Theorem}[section]
\newtheorem{proposition}[theorem]{Proposition}
\newtheorem{corollary}[theorem]{Corollary}

\theoremstyle{definition}
\newtheorem{definition}[theorem]{Definition}
\newtheorem{example}[theorem]{Example}
\newtheorem{remark}[theorem]{Remark}

\numberwithin{equation}{section}

\newenvironment{red}{\relax\color{red}}{\hspace*{.5ex}\relax}
\newcommand{\ber}{\begin{red}}
\newcommand{\er}{\end{red}}

\def\seteq{\mathbin{:=}}

\def\OP{\mathscr{P}}
\def\OS{\mathscr{S}}
\def\g{\mathfrak g}
\def\h{\mathfrak h}

\def\ve{\varepsilon}
\def\re{{}^\triangleright}
\def\N{\mathbb N}
\def\Z{\mathbb Z}
\def\A{\mathbb A}
\def\z{\mathsf z}
\def\U{\mathsf u}
\def\V{\mathsf v}
\def\Q{\mathbb Q}

\def\wt{{\rm wt}}
\def\het{{\rm ht}}

\def\dim{{\rm dim}}
\def\vch{\overset{\circ}{\chi}}
\newcommand{\nc}{\newcommand}
\nc{\hs}{\hspace*}

\newcommand{\dtri}{
\bsegment
\move(-1 0)\lvec(0 1)\lvec(0 0)\lvec(-1 0)\ifill f:0.7
\esegment
}

\newcommand{\dsqa}{
\bsegment
\move(-1 0)\lvec(-1 1)\lvec(0 1)\lvec(0 0)\lvec(-1 0)\ifill f:0.8
\esegment
}
\newcommand{\dhsqa}{
\bsegment
\move(-0.2 0)\lvec(-0.2 1)\lvec(0.2 1)\lvec(0.2 0)\lvec(-0.2 0)\ifill f:0.7
\esegment
}
\newcommand{\dhsqA}{
\bsegment
\move(-1 0.2)\lvec(0 0.2)\lvec(0 -0.2)\lvec(-1 -0.2)\lvec(-1 0.2)\ifill f:0.7
\esegment
}
\newcommand{\dtrj}{
\bsegment
\move(0 -1)\lvec(1 -1)\lvec(1 -0.5)\lvec(0 -0.5)\ifill f:0.7
\esegment
}

\newcommand{\ublock}{
\begin{texdraw}
\drawdim em
\setunitscale 1.0
\move(0 -0.2)
\bsegment
\move(0 0)\lvec(0 1)\lvec(1 1)\lvec(1 0)\lvec(0 0)
\move(0 1)\lvec(0.5 1.5)\lvec(1.5 1.5)\lvec(1 1)
\move(1.5 1.5)\lvec(1.5 0.5)\lvec(1 0)
\htext(2 0.6){$=$}
\move(3 0.3)\lvec(3 1.3)\lvec(4 1.3)\lvec(4 0.3)\lvec(3 0.3)
\esegment
\end{texdraw}
}

\newcommand{\hhblock}{
\begin{texdraw}
\drawdim em
\setunitscale 1.0
\move(0 -0.2)
\bsegment
\move(0 0)\lvec(0 0.5)\lvec(1 0.5)\lvec(1 0)\lvec(0 0)
\move(0 0.5)\lvec(0.5 1)\lvec(1.5 1)\lvec(1 0.5)
\move(1.5 1)\lvec(1.5 0.5)\lvec(1 0)
\htext(2 0.4){$=$}
\move(3 0.3)\lvec(3 0.8)\lvec(4 0.8)\lvec(4 0.3)\lvec(3 0.3)
\esegment
\end{texdraw}
}

\newcommand{\SPatDnpO}{
 \begin{texdraw}
\fontsize{7}{7}\selectfont
\drawdim em
\setunitscale 1.9
\move(0 0)\lvec(-4 0)\lvec(-4 0.5)\lvec(0 0.5)\ifill f:0.7
\move(0 0)\rlvec(-4.3 0) \move(0 0.5)\rlvec(-4.3 0) \move(0 1)\rlvec(-4.3 0)
\move(0 2)\rlvec(-4.3 0) \move(0 3.5)\rlvec(-4.3 0) \move(0 4.5)\rlvec(-4.3 0)
\move(0 4.5)\rlvec(-4.3 0) \move(0 5)\rlvec(-4.3 0) \move(0 5.5)\rlvec(-4.3 0)
\move(0 6.5)\rlvec(-4.3 0) \move(0 8)\rlvec(-4.3 0) \move(0 9)\rlvec(-4.3 0)
\move(0 9.5)\rlvec(-4.3 0) \move(0 10)\rlvec(-4.3 0) \move(0 11)\rlvec(-4.3 0)
\move(0 0)\rlvec(0 11.3) \move(-1 0)\rlvec(0 11.3) \move(-2 0)\rlvec(0 11.3)
\move(-3 0)\rlvec(0 11.3) \move(-4 0)\rlvec(0 11.3)
\htext(-0.6 0.1){$n$} \htext(-1.6 0.1){$n$} \htext(-2.6 0.1){$n$}\htext(-3.6 0.1){$n$}
\htext(-0.6 0.6){$n$} \htext(-1.6 0.6){$n$} \htext(-2.6 0.6){$n$} \htext(-3.6 0.6){$n$}
\htext(-0.9 1.35){$n\!\!-\!\!1$} \htext(-1.9 1.35){$n\!\!-\!\!1$}
\htext(-2.9 1.35){$n\!\!-\!\!1$} \htext(-3.9 1.35){$n\!\!-\!\!1$}
\vtext(-0.6 2.6){$\cdots$} \vtext(-1.6 2.6){$\cdots$}\vtext(-2.6 2.6){$\cdots$} \vtext(-3.6 2.6){$\cdots$}
\htext(-0.6 3.85){$1$} \htext(-1.6 3.85){$1$}\htext(-2.6 3.85){$1$} \htext(-3.6 3.85){$1$}
\htext(-0.6 4.6){$0$} \htext(-1.6 4.6){$0$} \htext(-2.6 4.6){$0$}\htext(-3.6 4.6){$0$}
\htext(-0.6 5.1){$0$} \htext(-1.6 5.1){$0$} \htext(-2.6 5.1){$0$} \htext(-3.6 5.1){$0$}
\htext(-0.6 5.85){$1$} \htext(-1.6 5.85){$1$} \htext(-2.6 5.85){$1$} \htext(-3.6 5.85){$1$}
\vtext(-0.6 7.1){$\cdots$} \vtext(-1.6 7.1){$\cdots$} \vtext(-2.6 7.1){$\cdots$} \vtext(-3.6 7.1){$\cdots$}
\htext(-0.9 8.35){$n\!\!-\!\!1$} \htext(-1.9 8.35){$n\!\!-\!\!1$}
\htext(-2.9 8.35){$n\!\!-\!\!1$} \htext(-3.9 8.35){$n\!\!-\!\!1$}
\htext(-0.6 9.1){$n$} \htext(-1.6 9.1){$n$} \htext(-2.6 9.1){$n$} \htext(-3.6 9.1){$n$}
\htext(-0.6 9.6){$n$} \htext(-1.6 9.6){$n$} \htext(-2.6 9.6){$n$} \htext(-3.6 9.6){$n$}
\htext(-0.9 10.35){$n\!\!-\!\!1$} \htext(-1.9 10.35){$n\!\!-\!\!1$}
\htext(-2.9 10.35){$n\!\!-\!\!1$} \htext(-3.9 10.35){$n\!\!-\!\!1$}
\end{texdraw}
}
\newcommand{\FPatDnpO}{
\begin{texdraw}
\fontsize{7}{7}\selectfont
\drawdim em
\setunitscale 1.9
\move(0 0)\lvec(-4 0)\lvec(-4 0.5)\lvec(0 0.5)\ifill f:0.7
\move(0 0)\rlvec(-4.3 0) \move(0 0.5)\rlvec(-4.3 0) \move(0 1)\rlvec(-4.3 0)
\move(0 2)\rlvec(-4.3 0) \move(0 3.5)\rlvec(-4.3 0) \move(0 4.5)\rlvec(-4.3 0)
\move(0 4.5)\rlvec(-4.3 0) \move(0 5)\rlvec(-4.3 0) \move(0 5.5)\rlvec(-4.3 0)
\move(0 6.5)\rlvec(-4.3 0) \move(0 8)\rlvec(-4.3 0) \move(0 9)\rlvec(-4.3 0)
\move(0 9.5)\rlvec(-4.3 0) \move(0 10)\rlvec(-4.3 0) \move(0 11)\rlvec(-4.3 0)
\move(0 0)\rlvec(0 11.3) \move(-1 0)\rlvec(0 11.3) \move(-2 0)\rlvec(0 11.3)
\move(-3 0)\rlvec(0 11.3) \move(-4 0)\rlvec(0 11.3)
\htext(-0.6 0.1){$0$} \htext(-1.6 0.1){$0$} \htext(-2.6 0.1){$0$} \htext(-3.6 0.1){$0$}
\htext(-0.6 0.6){$0$} \htext(-1.6 0.6){$0$}\htext(-2.6 0.6){$0$} \htext(-3.6 0.6){$0$}
\htext(-0.6 1.35){$1$} \htext(-1.6 1.35){$1$} \htext(-2.6 1.35){$1$} \htext(-3.6 1.35){$1$}
\vtext(-0.6 2.6){$\cdots$} \vtext(-1.6 2.6){$\cdots$} \vtext(-2.6 2.6){$\cdots$} \vtext(-3.6 2.6){$\cdots$}
\htext(-0.9 3.85){$n\!\!-\!\!1$} \htext(-1.9 3.85){$n\!\!-\!\!1$}\htext(-2.9 3.85){$n\!\!-\!\!1$} \htext(-3.9 3.85){$n\!\!-\!\!1$}
\htext(-0.6 4.6){$n$} \htext(-1.6 4.6){$n$} \htext(-2.6 4.6){$n$}\htext(-3.6 4.6){$n$}
\htext(-0.6 5.1){$n$} \htext(-1.6 5.1){$n$} \htext(-2.6 5.1){$n$} \htext(-3.6 5.1){$n$}
\htext(-0.9 5.85){$n\!\!-\!\!1$} \htext(-1.9 5.85){$n\!\!-\!\!1$}\htext(-2.9 5.85){$n\!\!-\!\!1$} \htext(-3.9 5.85){$n\!\!-\!\!1$}
\vtext(-0.6 7.1){$\cdots$} \vtext(-1.6 7.1){$\cdots$}\vtext(-2.6 7.1){$\cdots$} \vtext(-3.6 7.1){$\cdots$}
\htext(-0.6 8.35){$1$} \htext(-1.6 8.35){$1$} \htext(-2.6 8.35){$1$} \htext(-3.6 8.35){$1$}
\htext(-0.6 9.1){$0$} \htext(-1.6 9.1){$0$} \htext(-2.6 9.1){$0$} \htext(-3.6 9.1){$0$}
\htext(-0.6 9.6){$0$} \htext(-1.6 9.6){$0$} \htext(-2.6 9.6){$0$} \htext(-3.6 9.6){$0$}
\htext(-0.6 10.35){$1$} \htext(-1.6 10.35){$1$} \htext(-2.6 10.35){$1$} \htext(-3.6 10.35){$1$}
\end{texdraw}
}
\newcommand{\Sevenone}{
\begin{texdraw}
\drawdim em
\move(0 -0.2)
\bsegment
\move(0 0)\lvec(1 0)\lvec(1 1)\lvec(0 1)\lvec(0 0)
\htext(0.3 0.15){$7$}
\esegment
\end{texdraw}
}
\newcommand{\Seventwo}{
\begin{texdraw}
\drawdim em
\move(0 -0.2)
\bsegment
\move(0 0)\lvec(1 0)\lvec(1 1)\lvec(0 1)\lvec(0 0)
\move(1 0)\lvec(2 0)\lvec(2 1)\lvec(1 1)\lvec(1 0)
\htext(0.3 0.15){$7$}\htext(1.3 0.15){$7$}
\esegment
\end{texdraw}
}
\newcommand{\Seventhree}{
\begin{texdraw}
\drawdim em
\move(0 -0.2)
\bsegment
\move(0 0)\lvec(1 0)\lvec(1 1)\lvec(0 1)\lvec(0 0)
\move(1 0)\lvec(2 0)\lvec(2 1)\lvec(1 1)\lvec(1 0)
\move(2 0)\lvec(3 0)\lvec(3 1)\lvec(2 1)\lvec(2 0)
\htext(0.3 0.15){$7$}\htext(1.3 0.15){$7$}\htext(2.3 0.15){$7$}
\esegment
\end{texdraw}
}
\newcommand{\Sevenfour}{
\begin{texdraw}
\drawdim em
\move(0 -1)
\bsegment
\move(0 0)\lvec(1 0)\lvec(1 1)\lvec(0 1)\lvec(0 0)
\move(1 0)\lvec(2 0)\lvec(2 1)\lvec(1 1)\lvec(1 0)
\move(2 0)\lvec(3 0)\lvec(3 1)\lvec(2 1)\lvec(2 0)
\move(3 0)\lvec(4 0)\lvec(4 1)\lvec(3 1)\lvec(3 0)
\htext(0.3 0.15){$7$}\htext(1.3 0.15){$7$}\htext(2.3 0.15){$7$}\htext(3.3 0.15){$7$}
\esegment
\end{texdraw}
}

\newcommand{\modularone}{
\begin{texdraw}
\drawdim em
\setunitscale 1.2
\move(7 3)\dsqa
\move(7 2)\dsqa
\move(7 1)\dsqa
\move(7 0)\dsqa
\move(8 0)\dsqa
\move(9 0)\dsqa
\move(10 0)\dsqa
\move(0 0)\lvec(10 0)\lvec(10 1)\lvec(0 1)\lvec(0 0)
\move(2 1)\lvec(2 2)\lvec(10 2)\lvec(10 1)
\move(3 2)\lvec(3 3)\lvec(10 3)\lvec(10 2)
\move(6 3)\lvec(6 4)\lvec(10 4)\lvec(10 3)
\move(8 4)\lvec(8 5)\lvec(10 5)\lvec(10 4)
\move(9 0)\lvec(9 5)
\move(8 0)\lvec(8 4)
\move(7 0)\lvec(7 4)
\move(6 0)\lvec(6 3)
\move(5 0)\lvec(5 3)
\move(4 0)\lvec(4 3)
\move(3 0)\lvec(3 2)
\move(2 0)\lvec(2 1)
\move(1 0)\lvec(1 1)
\htext(0.3 0.15){$1$}
\htext(1.3 0.15){$7$}
\htext(2.3 0.15){$7$}\htext(2.3 1.15){$2$}
\htext(3.3 0.15){$7$}\htext(3.3 1.15){$7$}\htext(3.3 2.15){$1$}
\htext(4.3 0.15){$7$}\htext(4.3 1.15){$7$}\htext(4.3 2.15){$7$}
\htext(5.3 0.15){$7$}\htext(5.3 1.15){$7$}\htext(5.3 2.15){$\overline{7}$}
\htext(6.3 0.15){$7$}\htext(6.3 1.15){$7$}\htext(6.3 2.15){$7$}\htext(6.3 3.15){$7$}
\htext(7.3 0.15){$7$}\htext(7.3 1.15){$7$}\htext(7.3 2.15){$7$}\htext(7.3 3.15){$\overline{7}$}
\htext(8.3 0.15){$7$}\htext(8.3 1.15){$7$}\htext(8.3 2.15){$7$}\htext(8.3 3.15){$7$}\htext(8.3 4.15){$3$}
\htext(9.3 0.15){$7$}\htext(9.3 1.15){$7$}\htext(9.3 2.15){$7$}\htext(9.3 3.15){$7$}\htext(9.3 4.15){$5$}
\end{texdraw}
}
\newcommand{\modulartwo}{
\begin{texdraw}
\drawdim em
\setunitscale 1.2
\move(5 2)\dsqa
\move(5 1)\dsqa
\move(5 0)\dsqa
\move(0 0)\lvec(10 0)\lvec(10 1)\lvec(0 1)\lvec(0 0)
\move(2 1)\lvec(2 2)\lvec(10 2)\lvec(10 1)
\move(3 2)\lvec(3 3)\lvec(10 3)\lvec(10 2)
\move(6 3)\lvec(6 4)\lvec(10 4)\lvec(10 3)
\move(8 4)\lvec(8 5)\lvec(10 5)\lvec(10 4)
\move(9 0)\lvec(9 5)
\move(8 0)\lvec(8 4)
\move(7 0)\lvec(7 4)
\move(6 0)\lvec(6 3)
\move(5 0)\lvec(5 3)
\move(4 0)\lvec(4 3)
\move(3 0)\lvec(3 2)
\move(2 0)\lvec(2 1)
\move(1 0)\lvec(1 1)
\htext(0.3 0.15){$1$}
\htext(1.3 0.15){$7$}
\htext(2.3 0.15){$7$}\htext(2.3 1.15){$2$}
\htext(3.3 0.15){$7$}\htext(3.3 1.15){$7$}\htext(3.3 2.15){$1$}
\htext(4.3 0.15){$7$}\htext(4.3 1.15){$7$}\htext(4.3 2.15){$7$}
\htext(5.3 0.15){$7$}\htext(5.3 1.15){$7$}\htext(5.3 2.15){$\overline{7}$}
\htext(6.3 0.15){$7$}\htext(6.3 1.15){$7$}\htext(6.3 2.15){$7$}\htext(6.3 3.15){$7$}
\htext(7.3 0.15){$7$}\htext(7.3 1.15){$7$}\htext(7.3 2.15){$7$}\htext(7.3 3.15){$\overline{7}$}
\htext(8.3 0.15){$7$}\htext(8.3 1.15){$7$}\htext(8.3 2.15){$7$}\htext(8.3 3.15){$7$}\htext(8.3 4.15){$3$}
\htext(9.3 0.15){$7$}\htext(9.3 1.15){$7$}\htext(9.3 2.15){$7$}\htext(9.3 3.15){$7$}\htext(9.3 4.15){$5$}
\end{texdraw}
}
\newcommand{\modularthree}{
\begin{texdraw}
\drawdim em
\setunitscale 1.2
\move(6 3)\dsqa
\move(6 2)\dsqa
\move(6 1)\dsqa
\move(6 0)\dsqa
\move(0 0)\lvec(9 0)\lvec(9 1)\lvec(0 1)\lvec(0 0)
\move(2 1)\lvec(2 2)\lvec(9 2)\lvec(9 1)
\move(3 2)\lvec(3 3)\lvec(9 3)\lvec(9 2)
\move(5 3)\lvec(5 4)\lvec(9 4)\lvec(9 3)
\move(7 4)\lvec(7 5)\lvec(9 5)\lvec(9 4)
\move(8 0)\lvec(8 5)
\move(7 0)\lvec(7 5)
\move(6 0)\lvec(6 4)
\move(5 0)\lvec(5 3)
\move(4 0)\lvec(4 3)
\move(3 0)\lvec(3 2)
\move(2 0)\lvec(2 1)
\move(1 0)\lvec(1 1)
\htext(0.3 0.15){$1$}
\htext(1.3 0.15){$7$}
\htext(2.3 0.15){$7$}\htext(2.3 1.15){$2$}
\htext(3.3 0.15){$7$}\htext(3.3 1.15){$7$}\htext(3.3 2.15){$1$}
\htext(4.3 0.15){$7$}\htext(4.3 1.15){$7$}\htext(4.3 2.15){$\overline{7}$}
\htext(5.3 0.15){$7$}\htext(5.3 1.15){$7$}\htext(5.3 2.15){$7$}\htext(5.3 3.15){$7$}
\htext(6.3 0.15){$7$}\htext(6.3 1.15){$7$}\htext(6.3 2.15){$7$}\htext(6.3 3.15){$\overline{7}$}
\htext(7.3 0.15){$7$}\htext(7.3 1.15){$7$}\htext(7.3 2.15){$7$}\htext(7.3 3.15){$7$}\htext(7.3 4.15){$3$}
\htext(8.3 0.15){$7$}\htext(8.3 1.15){$7$}\htext(8.3 2.15){$7$}\htext(8.3 3.15){$7$}\htext(8.3 4.15){$5$}
\end{texdraw}
}

\newcommand{\modularfour}{
\begin{texdraw}
\drawdim em
\setunitscale 1.2
\move(8 0)\dsqa
\move(7 0)\dsqa
\move(6 0)\dsqa
\move(5 0)\dsqa
\move(4 0)\dsqa
\move(3 0)\dsqa
\move(2 0)\dsqa
\move(0 0)\lvec(8 0)\lvec(8 1)\lvec(0 1)\lvec(0 0)
\move(2 1)\lvec(2 2)\lvec(8 2)\lvec(8 1)
\move(3 2)\lvec(3 3)\lvec(8 3)\lvec(8 2)
\move(5 3)\lvec(5 4)\lvec(8 4)\lvec(8 3)
\move(6 4)\lvec(6 5)\lvec(8 5)\lvec(8 4)
\move(7 4)\lvec(7 5)\lvec(8 5)\lvec(8 4)
\move(7 0)\lvec(7 4)
\move(6 0)\lvec(6 4)
\move(5 0)\lvec(5 3)
\move(4 0)\lvec(4 3)
\move(3 0)\lvec(3 2)
\move(2 0)\lvec(2 1)
\move(1 0)\lvec(1 1)
\htext(0.3 0.15){$1$}
\htext(1.3 0.15){$7$}
\htext(2.3 0.15){$7$}\htext(2.3 1.15){$2$}
\htext(3.3 0.15){$7$}\htext(3.3 1.15){$7$}\htext(3.3 2.15){$1$}
\htext(4.3 0.15){$7$}\htext(4.3 1.15){$7$}\htext(4.3 2.15){$\overline{7}$}
\htext(5.3 0.15){$7$}\htext(5.3 1.15){$7$}\htext(5.3 2.15){$7$}\htext(5.3 3.15){$\overline{7}$}
\htext(6.3 0.15){$7$}\htext(6.3 1.15){$7$}\htext(6.3 2.15){$7$}\htext(6.3 3.15){$7$}\htext(6.3 4.15){$3$}
\htext(7.3 0.15){$7$}\htext(7.3 1.15){$7$}\htext(7.3 2.15){$7$}\htext(7.3 3.15){$7$}\htext(7.3 4.15){$5$}
\end{texdraw}
}

\newcommand{\modularsix}{
\begin{texdraw}
\drawdim em
\setunitscale 1.2
\move(5 2)\dsqa
\move(5 1)\dsqa
\move(5 0)\dsqa
\move(0 0)\lvec(10 0)\lvec(10 1)\lvec(0 1)\lvec(0 0)
\move(2 1)\lvec(2 2)\lvec(10 2)\lvec(10 1)
\move(3 2)\lvec(3 3)\lvec(10 3)\lvec(10 2)
\move(6 3)\lvec(6 4)\lvec(10 4)\lvec(10 3)
\move(8 4)\lvec(8 5)\lvec(10 5)\lvec(10 4)
\move(9 0)\lvec(9 5)
\move(8 0)\lvec(8 4)
\move(7 0)\lvec(7 4)
\move(6 0)\lvec(6 3)
\move(5 0)\lvec(5 3)
\move(4 0)\lvec(4 3)
\move(3 0)\lvec(3 2)
\move(2 0)\lvec(2 1)
\move(1 0)\lvec(1 1)
\htext(0.3 0.15){$1$}
\htext(1.3 0.15){$7$}
\htext(2.3 0.15){$7$}\htext(2.3 1.15){$2$}
\htext(3.3 0.15){$7$}\htext(3.3 1.15){$7$}\htext(3.3 2.15){$1$}
\htext(4.3 0.15){$7$}\htext(4.3 1.15){$7$}\htext(4.3 2.15){$7$}
\htext(5.3 0.15){$7$}\htext(5.3 1.15){$7$}\htext(5.3 2.15){$\overline{7}$}
\htext(6.3 0.15){$7$}\htext(6.3 1.15){$7$}\htext(6.3 2.15){$7$}\htext(6.3 3.15){$7$}
\htext(7.3 0.15){$7$}\htext(7.3 1.15){$7$}\htext(7.3 2.15){$7$}\htext(7.3 3.15){$7$}
\htext(8.3 0.15){$7$}\htext(8.3 1.15){$7$}\htext(8.3 2.15){$7$}\htext(8.3 3.15){$7$}\htext(8.3 4.15){$3$}
\htext(9.3 0.15){$7$}\htext(9.3 1.15){$7$}\htext(9.3 2.15){$7$}\htext(9.3 3.15){$7$}\htext(9.3 4.15){$5$}
\end{texdraw}
}

\newcommand{\modularseven}{
\begin{texdraw}
\drawdim em
\setunitscale 1.2
\move(6 3)\dsqa
\move(6 2)\dsqa
\move(6 1)\dsqa
\move(6 0)\dsqa
\move(0 0)\lvec(9 0)\lvec(9 1)\lvec(0 1)\lvec(0 0)
\move(2 1)\lvec(2 2)\lvec(9 2)\lvec(9 1)
\move(3 2)\lvec(3 3)\lvec(9 3)\lvec(9 2)
\move(5 3)\lvec(5 4)\lvec(9 4)\lvec(9 3)
\move(7 4)\lvec(7 5)\lvec(9 5)\lvec(9 4)
\move(8 0)\lvec(8 5)
\move(7 0)\lvec(7 5)
\move(6 0)\lvec(6 4)
\move(5 0)\lvec(5 3)
\move(4 0)\lvec(4 3)
\move(3 0)\lvec(3 2)
\move(2 0)\lvec(2 1)
\move(1 0)\lvec(1 1)
\htext(0.3 0.15){$1$}
\htext(1.3 0.15){$7$}
\htext(2.3 0.15){$7$}\htext(2.3 1.15){$2$}
\htext(3.3 0.15){$7$}\htext(3.3 1.15){$7$}\htext(3.3 2.15){$1$}
\htext(4.3 0.15){$7$}\htext(4.3 1.15){$7$}\htext(4.3 2.15){$\overline{7}$}
\htext(5.3 0.15){$7$}\htext(5.3 1.15){$7$}\htext(5.3 2.15){$7$}\htext(5.3 3.15){$7$}
\htext(6.3 0.15){$7$}\htext(6.3 1.15){$7$}\htext(6.3 2.15){$7$}\htext(6.3 3.15){$7$}
\htext(7.3 0.15){$7$}\htext(7.3 1.15){$7$}\htext(7.3 2.15){$7$}\htext(7.3 3.15){$7$}\htext(7.3 4.15){$3$}
\htext(8.3 0.15){$7$}\htext(8.3 1.15){$7$}\htext(8.3 2.15){$7$}\htext(8.3 3.15){$7$}\htext(8.3 4.15){$5$}
\end{texdraw}
}

\newcommand{\modulareight}{
\begin{texdraw}
\drawdim em
\setunitscale 1.2
\move(8 0)\dsqa
\move(7 0)\dsqa
\move(6 0)\dsqa
\move(6 1)\dsqa
\move(6 2)\dsqa
\move(6 3)\dsqa
\move(0 0)\lvec(8 0)\lvec(8 1)\lvec(0 1)\lvec(0 0)
\move(2 1)\lvec(2 2)\lvec(8 2)\lvec(8 1)
\move(3 2)\lvec(3 3)\lvec(8 3)\lvec(8 2)
\move(5 3)\lvec(5 4)\lvec(8 4)\lvec(8 3)
\move(6 4)\lvec(6 5)\lvec(8 5)\lvec(8 4)
\move(7 4)\lvec(7 5)\lvec(8 5)\lvec(8 4)
\move(7 0)\lvec(7 4)
\move(6 0)\lvec(6 4)
\move(5 0)\lvec(5 3)
\move(4 0)\lvec(4 3)
\move(3 0)\lvec(3 2)
\move(2 0)\lvec(2 1)
\move(1 0)\lvec(1 1)
\htext(0.3 0.15){$1$}
\htext(1.3 0.15){$7$}
\htext(2.3 0.15){$7$}\htext(2.3 1.15){$2$}
\htext(3.3 0.15){$7$}\htext(3.3 1.15){$7$}\htext(3.3 2.15){$1$}
\htext(4.3 0.15){$7$}\htext(4.3 1.15){$7$}\htext(4.3 2.15){$\overline{7}$}
\htext(5.3 0.15){$7$}\htext(5.3 1.15){$7$}\htext(5.3 2.15){$7$}\htext(5.3 3.15){$7$}
\htext(6.3 0.15){$7$}\htext(6.3 1.15){$7$}\htext(6.3 2.15){$7$}\htext(6.3 3.15){$7$}\htext(6.3 4.15){$3$}
\htext(7.3 0.15){$7$}\htext(7.3 1.15){$7$}\htext(7.3 2.15){$7$}\htext(7.3 3.15){$7$}\htext(7.3 4.15){$5$}
\end{texdraw}
}
\newcommand{\modularnine}{
\begin{texdraw}
\drawdim em
\setunitscale 1.2
\move(2 0)\dsqa
\move(3 0)\dsqa
\move(4 0)\dsqa
\move(5 0)\dsqa
\move(6 0)\dsqa
\move(7 0)\dsqa
\move(0 0)\lvec(0 1)\lvec(1 1)
\move(0 0)\lvec(1 0)
\move(1 0)\lvec(7 0)\lvec(7 1)\lvec(1 1)\lvec(1 0)
\move(2 1)\lvec(2 2)\lvec(7 2)\lvec(7 1)
\move(3 2)\lvec(3 3)\lvec(7 3)\lvec(7 2)
\move(5 3)\lvec(5 4)\lvec(7 4)\lvec(7 3)
\move(7 0)\lvec(7 4)
\move(6 0)\lvec(6 4)
\move(5 0)\lvec(5 3)
\move(4 0)\lvec(4 3)
\move(3 0)\lvec(3 2)
\move(2 0)\lvec(2 1)
\move(1 0)\lvec(1 1)
\htext(0.3 0.15){$1$}
\htext(1.3 0.15){$7$}
\htext(2.3 0.15){$7$}\htext(2.3 1.15){$2$}
\htext(3.3 0.15){$7$}\htext(3.3 1.15){$7$}\htext(3.3 2.15){$1$}
\htext(4.3 0.15){$7$}\htext(4.3 1.15){$7$}\htext(4.3 2.15){$\overline{7}$}
\htext(5.3 0.15){$7$}\htext(5.3 1.15){$7$}\htext(5.3 2.15){$7$}\htext(5.3 3.15){$3$}
\htext(6.3 0.15){$7$}\htext(6.3 1.15){$7$}\htext(6.3 2.15){$7$}\htext(6.3 3.15){$5$}
\end{texdraw}
}

\newcommand{\modularNine}{
\begin{texdraw}
\drawdim em
\setunitscale 1.2
\move(1 0)\dhsqa
\move(2 0)\dhsqA
\move(3 0)\dhsqA
\move(4 0)\dhsqA
\move(5 0)\dhsqA
\move(6 0)\dhsqA
\move(0 0)\lvec(0 1)\lvec(1 1)
\move(0 0)\lvec(1 0)
\move(1 0)\lvec(6 0)\lvec(6 1)\lvec(1 1)\lvec(1 0)
\move(2 1)\lvec(2 2)\lvec(6 2)\lvec(6 1)
\move(4 2)\lvec(4 3)\lvec(6 3)\lvec(6 2)
\move(5 0)\lvec(5 3)
\move(4 0)\lvec(4 3)
\move(3 0)\lvec(3 2)
\move(2 0)\lvec(2 1)
\move(1 0)\lvec(1 1)
\htext(0.3 0.15){$1$}
\htext(1.3 0.15){$2$}
\htext(2.3 0.15){$7$}\htext(2.3 1.15){$1$}
\htext(3.3 0.15){$7$}\htext(3.3 1.15){$\overline{7}$}
\htext(4.3 0.15){$7$}\htext(4.3 1.15){$7$}\htext(4.3 2.15){$3$}
\htext(5.3 0.15){$7$}\htext(5.3 1.15){$7$}\htext(5.3 2.15){$5$}
\end{texdraw}
}
\newcommand{\modularEight}{
\begin{texdraw}
\drawdim em
\setunitscale 1.2
\move(0 0)\lvec(0 1)\lvec(1 1)
\move(0 0)\lvec(1 0)
\move(1 0)\lvec(7 0)\lvec(7 1)\lvec(1 1)\lvec(1 0)
\move(2 1)\lvec(2 2)\lvec(7 2)\lvec(7 1)
\move(3 2)\lvec(3 3)\lvec(7 3)\lvec(7 2)
\move(5 3)\lvec(5 4)\lvec(7 4)\lvec(7 3)
\move(7 0)\lvec(7 4)
\move(6 0)\lvec(6 4)
\move(5 0)\lvec(5 3)
\move(4 0)\lvec(4 3)
\move(3 0)\lvec(3 2)
\move(2 0)\lvec(2 1)
\move(1 0)\lvec(1 1)
\htext(0.3 0.15){$1$}
\htext(1.3 0.15){$7$}
\htext(2.3 0.15){$7$}\htext(2.3 1.15){$2$}
\htext(3.3 0.15){$7$}\htext(3.3 1.15){$7$}\htext(3.3 2.15){$1$}
\htext(4.3 0.15){$7$}\htext(4.3 1.15){$7$}\htext(4.3 2.15){$\overline{7}$}
\htext(5.3 0.15){$7$}\htext(5.3 1.15){$7$}\htext(5.3 2.15){$7$}\htext(5.3 3.15){$3$}
\htext(6.3 0.15){$7$}\htext(6.3 1.15){$7$}\htext(6.3 2.15){$7$}\htext(6.3 3.15){$5$}
\end{texdraw}
}
\newcommand{\modularSeven}{
\begin{texdraw}
\drawdim em
\setunitscale 1.2
\move(0 0)
\bsegment
\move(4 0)\dhsqa \move(4 1)\dhsqa \move(4 2)\dhsqa
\move(5 0)\dhsqA
\move(6 0)\dhsqA
\move(7 0)\dhsqA
\move(0 0)\lvec(0 1)\lvec(1 1)
\move(0 0)\lvec(1 0)
\move(1 0)\lvec(7 0)\lvec(7 1)\lvec(1 1)\lvec(1 0)
\move(2 1)\lvec(2 2)\lvec(7 2)\lvec(7 1)
\move(3 2)\lvec(3 3)\lvec(7 3)\lvec(7 2)
\move(5 3)\lvec(5 4)\lvec(7 4)\lvec(7 3)
\move(7 0)\lvec(7 4)
\move(6 0)\lvec(6 4)
\move(5 0)\lvec(5 3)
\move(4 0)\lvec(4 3)
\move(3 0)\lvec(3 2)
\move(2 0)\lvec(2 1)
\move(1 0)\lvec(1 1)
\htext(0.3 0.15){$1$}
\htext(1.3 0.15){$7$}
\htext(2.3 0.15){$7$}\htext(2.3 1.15){$2$}
\htext(3.3 0.15){$7$}\htext(3.3 1.15){$7$}\htext(3.3 2.15){$1$}
\htext(4.3 0.15){$7$}\htext(4.3 1.15){$7$}\htext(4.3 2.15){$\overline{7}$}
\htext(5.3 0.15){$7$}\htext(5.3 1.15){$7$}\htext(5.3 2.15){$7$}\htext(5.3 3.15){$3$}
\htext(6.3 0.15){$7$}\htext(6.3 1.15){$7$}\htext(6.3 2.15){$7$}\htext(6.3 3.15){$5$}
\esegment
\move(11 0)
\bsegment
\move(5 0)\dhsqa \move(5 1)\dhsqa \move(5 2)\dhsqa \move(5 3)\dhsqa
\move(6 0)\dhsqA
\move(7 0)\dhsqA
\move(0 0)\lvec(0 1)\lvec(1 1)
\move(0 0)\lvec(1 0)
\move(1 0)\lvec(7 0)\lvec(7 1)\lvec(1 1)\lvec(1 0)
\move(2 1)\lvec(2 2)\lvec(7 2)\lvec(7 1)
\move(3 2)\lvec(3 3)\lvec(7 3)\lvec(7 2)
\move(5 3)\lvec(5 4)\lvec(7 4)\lvec(7 3)
\move(7 0)\lvec(7 4)
\move(6 0)\lvec(6 4)
\move(5 0)\lvec(5 3)
\move(4 0)\lvec(4 3)
\move(3 0)\lvec(3 2)
\move(2 0)\lvec(2 1)
\move(1 0)\lvec(1 1)
\htext(0.3 0.15){$1$}
\htext(1.3 0.15){$7$}
\htext(2.3 0.15){$7$}\htext(2.3 1.15){$2$}
\htext(3.3 0.15){$7$}\htext(3.3 1.15){$7$}\htext(3.3 2.15){$1$}
\htext(4.3 0.15){$7$}\htext(4.3 1.15){$7$}\htext(4.3 2.15){$\overline{7}$}
\htext(5.3 0.15){$7$}\htext(5.3 1.15){$7$}\htext(5.3 2.15){$7$}\htext(5.3 3.15){$3$}
\htext(6.3 0.15){$7$}\htext(6.3 1.15){$7$}\htext(6.3 2.15){$7$}\htext(6.3 3.15){$5$}
\esegment
\end{texdraw}
}

\newcommand{\modularSix}{
\begin{texdraw}
\drawdim em
\setunitscale 1.2
\move(0 0)
\bsegment
\move(0 0)\lvec(8 0)\lvec(8 1)\lvec(0 1)\lvec(0 0)
\move(2 1)\lvec(2 2)\lvec(8 2)\lvec(8 1)
\move(3 2)\lvec(3 3)\lvec(8 3)\lvec(8 2)
\move(5 3)\lvec(5 4)\lvec(8 4)\lvec(8 3)
\move(6 4)\lvec(6 5)\lvec(8 5)\lvec(8 4)
\move(7 4)\lvec(7 5)\lvec(8 5)\lvec(8 4)
\move(7 0)\lvec(7 4)
\move(6 0)\lvec(6 4)
\move(5 0)\lvec(5 3)
\move(4 0)\lvec(4 3)
\move(3 0)\lvec(3 2)
\move(2 0)\lvec(2 1)
\move(1 0)\lvec(1 1)
\htext(0.3 0.15){$1$}
\htext(1.3 0.15){$7$}
\htext(2.3 0.15){$7$}\htext(2.3 1.15){$2$}
\htext(3.3 0.15){$7$}\htext(3.3 1.15){$7$}\htext(3.3 2.15){$1$}
\htext(4.3 0.15){$7$}\htext(4.3 1.15){$7$}\htext(4.3 2.15){$7$}
\htext(5.3 0.15){$7$}\htext(5.3 1.15){$7$}\htext(5.3 2.15){$7$}\htext(5.3 3.15){$\overline{7}$}
\htext(6.3 0.15){$7$}\htext(6.3 1.15){$7$}\htext(6.3 2.15){$7$}\htext(6.3 3.15){$7$}\htext(6.3 4.15){$3$}
\htext(7.3 0.15){$7$}\htext(7.3 1.15){$7$}\htext(7.3 2.15){$7$}\htext(7.3 3.15){$7$}\htext(7.3 4.15){$5$}
\esegment
\move(11 0)
\bsegment
\move(5 0)\dhsqa \move(5 1)\dhsqa \move(5 2)\dhsqa \move(5 3)\dhsqa
\move(0 0)\lvec(8 0)\lvec(8 1)\lvec(0 1)\lvec(0 0)
\move(2 1)\lvec(2 2)\lvec(8 2)\lvec(8 1)
\move(3 2)\lvec(3 3)\lvec(8 3)\lvec(8 2)
\move(5 3)\lvec(5 4)\lvec(8 4)\lvec(8 3)
\move(6 4)\lvec(6 5)\lvec(8 5)\lvec(8 4)
\move(7 4)\lvec(7 5)\lvec(8 5)\lvec(8 4)
\move(7 0)\lvec(7 4)
\move(6 0)\lvec(6 4)
\move(5 0)\lvec(5 3)
\move(4 0)\lvec(4 3)
\move(3 0)\lvec(3 2)
\move(2 0)\lvec(2 1)
\move(1 0)\lvec(1 1)
\htext(0.3 0.15){$1$}
\htext(1.3 0.15){$7$}
\htext(2.3 0.15){$7$}\htext(2.3 1.15){$2$}
\htext(3.3 0.15){$7$}\htext(3.3 1.15){$7$}\htext(3.3 2.15){$1$}
\htext(4.3 0.15){$7$}\htext(4.3 1.15){$7$}\htext(4.3 2.15){$\overline{7}$}
\htext(5.3 0.15){$7$}\htext(5.3 1.15){$7$}\htext(5.3 2.15){$7$}\htext(5.3 3.15){$7$}
\htext(6.3 0.15){$7$}\htext(6.3 1.15){$7$}\htext(6.3 2.15){$7$}\htext(6.3 3.15){$7$}\htext(6.3 4.15){$3$}
\htext(7.3 0.15){$7$}\htext(7.3 1.15){$7$}\htext(7.3 2.15){$7$}\htext(7.3 3.15){$7$}\htext(7.3 4.15){$5$}
\esegment
\end{texdraw}
}
\newcommand{\modularFive}{
\begin{texdraw}
\drawdim em
\setunitscale 1.2
\move(4 0)\dhsqa \move(4 1)\dhsqa \move(4 2)\dhsqa
\move(0 0)\lvec(9 0)\lvec(9 1)\lvec(0 1)\lvec(0 0)
\move(2 1)\lvec(2 2)\lvec(9 2)\lvec(9 1)
\move(3 2)\lvec(3 3)\lvec(9 3)\lvec(9 2)
\move(5 3)\lvec(5 4)\lvec(9 4)\lvec(9 3)
\move(7 4)\lvec(7 5)\lvec(9 5)\lvec(9 4)
\move(8 0)\lvec(8 5)
\move(7 0)\lvec(7 5)
\move(6 0)\lvec(6 4)
\move(5 0)\lvec(5 3)
\move(4 0)\lvec(4 3)
\move(3 0)\lvec(3 2)
\move(2 0)\lvec(2 1)
\move(1 0)\lvec(1 1)
\htext(0.3 0.15){$1$}
\htext(1.3 0.15){$7$}
\htext(2.3 0.15){$7$}\htext(2.3 1.15){$2$}
\htext(3.3 0.15){$7$}\htext(3.3 1.15){$7$}\htext(3.3 2.15){$1$}
\htext(4.3 0.15){$7$}\htext(4.3 1.15){$7$}\htext(4.3 2.15){$\overline{7}$}
\htext(5.3 0.15){$7$}\htext(5.3 1.15){$7$}\htext(5.3 2.15){$7$}\htext(5.3 3.15){$7$}
\htext(6.3 0.15){$7$}\htext(6.3 1.15){$7$}\htext(6.3 2.15){$7$}\htext(6.3 3.15){$7$}
\htext(7.3 0.15){$7$}\htext(7.3 1.15){$7$}\htext(7.3 2.15){$7$}\htext(7.3 3.15){$7$}\htext(7.3 4.15){$3$}
\htext(8.3 0.15){$7$}\htext(8.3 1.15){$7$}\htext(8.3 2.15){$7$}\htext(8.3 3.15){$7$}\htext(8.3 4.15){$5$}
\end{texdraw}
}

    \newcommand{\htblock}{
    \begin{texdraw}
    \drawdim em
    \setunitscale 1.0
    \move(0 -0.2)
    \bsegment
    \move(0.25 0.25)\lvec(0.25 1.25)\lvec(1.25 1.25)\lvec(1.25 0.25)\lvec(0.25 0.25)
    \move(0.25 1.25)\lvec(0.5 1.5)\lvec(1.5 1.5)\lvec(1.25 1.25)
    \move(1.5 1.5)\lvec(1.5 0.5)\lvec(1.25 0.25)
    \move(1.25 0.25)\lvec(1 0)\lvec(0 0)\lvec(0.25 0.25)
    \htext(2 0.6){$=$}
    \move(3 0.3)\lvec(4 0.3)\lvec(4 1.3)\lvec(3 0.3)
    \htext(4.5 0.3){$,$}
    \esegment
    \move(5.5 -0.2)
    \bsegment
    \move(0 0)\lvec(0 1)\lvec(1 1)\lvec(1 0)\lvec(0 0)
    \move(0 1)\lvec(0.25 1.25)\lvec(1.25 1.25)\lvec(1 1)
    \move(1.25 1.25)\lvec(1.25 0.25)\lvec(1 0)
    \move(1.25 0.25)\lvec(1.5 0.5)\lvec(1.25 0.5)
    \htext(2 0.6){$=$}
    \move(3 0.3)\lvec(3 1.3)\lvec(4 1.3)\lvec(3 0.3)
    \esegment
    \end{texdraw}
    }

\newcommand{\PatAtnt}{
\begin{texdraw}
\fontsize{7}{7}\selectfont
\drawdim em
\setunitscale 1.9
\move(0 0)\lvec(-4 0)\lvec(-4 0.5)\lvec(0 0.5)\ifill f:0.7
\move(0 0)\rlvec(-4.3 0) \move(0 0.5)\rlvec(-4.3 0) \move(0 1)\rlvec(-4.3 0)
\move(0 2)\rlvec(-4.3 0) \move(0 3.5)\rlvec(-4.3 0) \move(0 4.5)\rlvec(-4.3 0)
\move(0 6)\rlvec(-4.3 0) \move(0 7)\rlvec(-4.3 0) \move(0 7.5)\rlvec(-4.3 0)
\move(0 8)\rlvec(-4.3 0) \move(0 9)\rlvec(-4.3 0) \move(0 0)\rlvec(0 9.3)
\move(-1 0)\rlvec(0 9.3) \move(-2 0)\rlvec(0 9.3) \move(-3 0)\rlvec(0 9.3)
\move(-4 0)\rlvec(0 9.3) \move(0 11)\rlvec(-4.3 0)
\move(0 0)\rlvec(0 11.3) \move(-1 0)\rlvec(0 11.3) \move(-2 0)\rlvec(0 11.3)
\move(-3 0)\rlvec(0 11.3) \move(-4 0)\rlvec(0 11.3)
\htext(-0.6 0.1){$0$} \htext(-1.6 0.1){$0$} \htext(-2.6 0.1){$0$}\htext(-3.6 0.1){$0$}
\htext(-0.6 0.6){$0$} \htext(-1.6 0.6){$0$} \htext(-2.6 0.6){$0$} \htext(-3.6 0.6){$0$}
\htext(-0.6 1.35){$1$} \htext(-1.6 1.35){$1$} \htext(-2.6 1.35){$1$} \htext(-3.6 1.35){$1$}
\htext(-0.6 3.85){$n$} \htext(-1.6 3.85){$n$} \htext(-2.6 3.85){$n$} \htext(-3.6 3.85){$n$}
\htext(-0.6 6.35){$1$} \htext(-1.6 6.35){$1$} \htext(-2.6 6.35){$1$} \htext(-3.6 6.35){$1$}
\htext(-0.6 7.1){$0$} \htext(-1.6 7.1){$0$} \htext(-2.6 7.1){$0$} \htext(-3.6 7.1){$0$}
\htext(-0.6 7.6){$0$} \htext(-1.6 7.6){$0$} \htext(-2.6 7.6){$0$} \htext(-3.6 7.6){$0$}
\htext(-0.6 8.35){$1$} \htext(-1.6 8.35){$1$} \htext(-2.6 8.35){$1$} \htext(-3.6 8.35){$1$}
\vtext(-0.4 2.6){$\cdots$} \vtext(-1.4 2.6){$\cdots$} \vtext(-2.4 2.6){$\cdots$} \vtext(-3.4 2.6){$\cdots$}
\vtext(-0.4 5.1){$\cdots$} \vtext(-1.4 5.1){$\cdots$} \vtext(-2.4 5.1){$\cdots$} \vtext(-3.4 5.1){$\cdots$}
\vtext(-0.4 9.6){$\cdots$} \vtext(-1.4 9.6){$\cdots$} \vtext(-2.4 9.6){$\cdots$} \vtext(-3.4 9.6){$\cdots$}
\end{texdraw}
}
\newcommand{\SPatAtnmo}{
\begin{texdraw}
\fontsize{7}{7}\selectfont
\drawdim em
\setunitscale 1.9
\move(0 0)\dtri \move(-1 0)\dtri \move(-2 0)\dtri \move(-3 0)\dtri
\move(0 0)\rlvec(-4.3 0) \move(0 1)\rlvec(-4.3 0) \move(0 2)\rlvec(-4.3 0)
\move(0 3.5)\rlvec(-4.3 0) \move(0 4.5)\rlvec(-4.3 0) \move(0 6)\rlvec(-4.3 0)
\move(0 7)\rlvec(-4.3 0) \move(0 8)\rlvec(-4.3 0) \move(0 9)\rlvec(-4.3 0)
\move(0 0)\rlvec(0 9.3) \move(-1 0)\rlvec(0 9.3) \move(-2 0)\rlvec(0 9.3)
\move(-3 0)\rlvec(0 9.3) \move(-4 0)\rlvec(0 9.3)
\move(-1 0)\rlvec(1 1) \move(-2 0)\rlvec(1 1) \move(-3 0)\rlvec(1 1)
\move(-4 0)\rlvec(1 1) \move(-1 7)\rlvec(1 1) \move(-2 7)\rlvec(1 1)
\move(-3 7)\rlvec(1 1) \move(-4 7)\rlvec(1 1) \move(0 11)\rlvec(-4.3 0)
\move(0 0)\rlvec(0 11.3) \move(-1 0)\rlvec(0 11.3) \move(-2 0)\rlvec(0 11.3)
\move(-3 0)\rlvec(0 11.3) \move(-4 0)\rlvec(0 11.3)
\vtext(-0.4 2.5){$\cdots$} \vtext(-0.4 5.1){$\cdots$}\vtext(-1.4 2.6){$\cdots$} \vtext(-1.4 5.1){$\cdots$}
\vtext(-2.4 2.5){$\cdots$} \vtext(-2.4 5.1){$\cdots$}\vtext(-3.4 2.6){$\cdots$} \vtext(-3.4 5.1){$\cdots$}
\vtext(-0.4 9.6){$\cdots$} \vtext(-1.4 9.6){$\cdots$} \vtext(-2.4 9.6){$\cdots$} \vtext(-3.4 9.6){$\cdots$}
\htext(-0.35 7.12){$0$} \htext(-0.85 7.6){$1$} \htext(-1.35 7.12){$1$}
\htext(-1.85 7.6){$0$} \htext(-2.35 7.12){$0$} \htext(-2.85 7.6){$1$}
\htext(-3.35 7.12){$1$} \htext(-3.85 7.6){$0$}
\htext(-0.35 0.12){$0$} \htext(-0.85 0.6){$1$} \htext(-1.35 0.12){$1$}
\htext(-1.85 0.6){$0$} \htext(-2.35 0.12){$0$} \htext(-2.85 0.6){$1$}
\htext(-3.35 0.12){$1$} \htext(-3.85 0.6){$0$}
\htext(-0.6 1.35){$2$} \htext(-1.6 1.35){$2$} \htext(-2.6 1.35){$2$}\htext(-3.6 1.35){$2$}
\htext(-0.6 6.35){$2$} \htext(-1.6 6.35){$2$} \htext(-2.6 6.35){$2$}\htext(-3.6 6.35){$2$}
\htext(-0.6 8.35){$2$} \htext(-1.6 8.35){$2$} \htext(-2.6 8.35){$2$}\htext(-3.6 8.35){$2$}
\htext(-0.6 3.85){$n$} \htext(-1.6 3.85){$n$} \htext(-2.6 3.85){$n$} \htext(-3.6 3.85){$n$}
\end{texdraw}
}
\newcommand{\FPatAtnmo}{
\begin{texdraw}
\fontsize{7}{7}\selectfont
\drawdim em
\setunitscale 1.9
\move(0 0)\dtri \move(-1 0)\dtri \move(-2 0)\dtri \move(-3 0)\dtri
\move(0 0)\rlvec(-4.3 0) \move(0 1)\rlvec(-4.3 0) \move(0 2)\rlvec(-4.3 0)
\move(0 3.5)\rlvec(-4.3 0) \move(0 4.5)\rlvec(-4.3 0) \move(0 6)\rlvec(-4.3 0)
\move(0 7)\rlvec(-4.3 0) \move(0 8)\rlvec(-4.3 0) \move(0 9)\rlvec(-4.3 0)
\move(0 0)\rlvec(0 9.3) \move(-1 0)\rlvec(0 9.3) \move(-2 0)\rlvec(0 9.3)
\move(-3 0)\rlvec(0 9.3) \move(-4 0)\rlvec(0 9.3)
\move(-1 0)\rlvec(1 1) \move(-2 0)\rlvec(1 1) \move(-3 0)\rlvec(1 1)
\move(-4 0)\rlvec(1 1) \move(-1 7)\rlvec(1 1) \move(-2 7)\rlvec(1 1)
\move(-3 7)\rlvec(1 1) \move(-4 7)\rlvec(1 1) \move(0 11)\rlvec(-4.3 0)
\move(0 0)\rlvec(0 11.3) \move(-1 0)\rlvec(0 11.3) \move(-2 0)\rlvec(0 11.3)
\move(-3 0)\rlvec(0 11.3) \move(-4 0)\rlvec(0 11.3)
\vtext(-0.4 2.6){$\cdots$} \vtext(-0.4 5.1){$\cdots$} \vtext(-1.4 2.6){$\cdots$} \vtext(-1.4 5.1){$\cdots$}
\vtext(-2.4 2.6){$\cdots$} \vtext(-2.4 5.1){$\cdots$} \vtext(-3.4 2.6){$\cdots$} \vtext(-3.4 5.1){$\cdots$}
\htext(-0.35 7.12){$1$} \htext(-0.85 7.6){$0$} \htext(-1.35 7.12){$0$}
\htext(-1.85 7.6){$1$} \htext(-2.35 7.12){$1$} \htext(-2.85 7.6){$0$}
\htext(-3.35 7.12){$0$} \htext(-3.85 7.6){$1$}
\htext(-0.35 0.12){$1$} \htext(-0.85 0.6){$0$} \htext(-1.35 0.12){$0$}
\htext(-1.85 0.6){$1$} \htext(-2.35 0.12){$1$} \htext(-2.85 0.6){$0$}
\htext(-3.35 0.12){$0$} \htext(-3.85 0.6){$1$}
\htext(-0.6 1.35){$2$} \htext(-1.6 1.35){$2$} \htext(-2.6 1.35){$2$} \htext(-3.6 1.35){$2$}
\htext(-0.6 6.35){$2$} \htext(-1.6 6.35){$2$} \htext(-2.6 6.35){$2$} \htext(-3.6 6.35){$2$}
\htext(-0.6 8.35){$2$} \htext(-1.6 8.35){$2$} \htext(-2.6 8.35){$2$} \htext(-3.6 8.35){$2$}
\htext(-0.6 3.85){$n$} \htext(-1.6 3.85){$n$} \htext(-2.6 3.85){$n$}\htext(-3.6 3.85){$n$}
\vtext(-0.4 9.6){$\cdots$} \vtext(-1.4 9.6){$\cdots$} \vtext(-2.4 9.6){$\cdots$} \vtext(-3.4 9.6){$\cdots$}
\end{texdraw}
}

\newcommand{\TPatBnO}{
\begin{texdraw}
\fontsize{7}{7}\selectfont
\drawdim em
\setunitscale 1.9
\move(0 0)\lvec(-4 0)\lvec(-4 0.5)\lvec(0 0.5)\ifill f:0.7
\move(0 0)\rlvec(-4.3 0) \move(0 1)\rlvec(-4.3 0) \move(0 2)\rlvec(-4.3 0)
\move(0 3.5)\rlvec(-4.3 0) \move(0 4.5)\rlvec(-4.3 0)
\move(0 5.5)\rlvec(-4.3 0) \move(0 6.5)\rlvec(-4.3 0) \move(0 8)\rlvec(-4.3 0)
\move(0 9)\rlvec(-4.3 0) \move(0 10)\rlvec(-4.3 0) \move(0 11)\rlvec(-4.3 0)
\move(0 0)\rlvec(0 11.3) \move(-1 0)\rlvec(0 11.3) \move(-2 0)\rlvec(0 11.3)
\move(-3 0)\rlvec(0 11.3) \move(-4 0)\rlvec(0 11.3)
\move(-1 4.5)\rlvec(1 1) \move(-2 4.5)\rlvec(1 1) \move(-3 4.5)\rlvec(1 1)
\move(-4 4.5)\rlvec(1 1)
\htext(-0.6 0.1){$n$}
\htext(-0.9 1.35){$n\!\!-\!\!1$} \vtext(-0.4 2.6){$\cdots$} \htext(-0.6 3.85){$2$}
\htext(-0.4 4.6){$1$} \htext(-0.85 5.1){$0$} \htext(-0.6 5.85){$2$}
\vtext(-0.4 7.1){$\cdots$} \htext(-0.9 8.35){$n\!\!-\!\!1$}
\htext(-0.6 9.1){$n$}
\htext(-0.9 10.35){$n\!\!-\!\!1$}
\htext(-2.6 0.1){$n$}
\htext(-2.9 1.35){$n\!\!-\!\!1$} \vtext(-2.4 2.6){$\cdots$} \htext(-2.6 3.85){$2$}
\htext(-2.4 4.6){$1$} \htext(-2.85 5.1){$0$} \htext(-2.6 5.85){$2$}
\vtext(-2.4 7.1){$\cdots$} \htext(-2.9 8.35){$n\!\!-\!\!1$}
\htext(-2.6 9.1){$n$}
\htext(-2.9 10.35){$n\!\!-\!\!1$}
\htext(-1.6 0.6){$n$}
\htext(-1.9 1.35){$n\!\!-\!\!1$} \vtext(-1.4 2.6){$\cdots$} \htext(-1.6 3.85){$2$}
\htext(-1.4 4.6){$0$} \htext(-1.85 5.1){$1$} \htext(-1.6 5.85){$2$}
\vtext(-1.4 7.1){$\cdots$} \htext(-1.9 8.35){$n\!\!-\!\!1$}
\htext(-1.6 9.6){$n$}
\htext(-1.9 10.35){$n\!\!-\!\!1$}
\htext(-3.6 0.6){$n$}
\htext(-3.9 1.35){$n\!\!-\!\!1$} \vtext(-3.4 2.6){$\cdots$} \htext(-3.6 3.85){$2$}
\htext(-3.4 4.6){$0$} \htext(-3.85 5.1){$1$} \htext(-3.6 5.85){$2$}
\vtext(-3.4 7.1){$\cdots$} \htext(-3.9 8.35){$n\!\!-\!\!1$}
\htext(-3.6 9.6){$n$}
\htext(-3.9 10.35){$n\!\!-\!\!1$}
\move(0 0.5)\rlvec(-4.3 0)
\move(0 9.5)\rlvec(-4.3 0)
\htext(-3.6 9.1){$n$} \htext(-3.6 0.1){$n$}
\htext(-2.6 9.6){$n$} \htext(-2.6 0.6){$n$}
\htext(-1.6 0.1){$n$} \htext(-1.6 9.1){$n$}
\htext(-0.6 9.6){$n$} \htext(-0.6 0.6){$n$}
\end{texdraw}
}
\newcommand{\SPatBnO}{
\begin{texdraw}
\fontsize{7}{7}\selectfont
\drawdim em
\setunitscale 1.9
\move(0 0)\dtri \move(-1 0)\dtri \move(-2 0)\dtri \move(-3 0)\dtri
\move(0 0)\rlvec(-4.3 0) \move(0 1)\rlvec(-4.3 0) \move(0 2)\rlvec(-4.3 0)
\move(0 3.5)\rlvec(-4.3 0) \move(0 4.5)\rlvec(-4.3 0)
\move(0 5.5)\rlvec(-4.3 0) \move(0 6.5)\rlvec(-4.3 0)
\move(0 8)\rlvec(-4.3 0) \move(0 9)\rlvec(-4.3 0)
\move(0 10)\rlvec(-4.3 0) \move(0 11)\rlvec(-4.3 0)
\move(0 0)\rlvec(0 11.3) \move(-1 0)\rlvec(0 11.3) \move(-2 0)\rlvec(0 11.3)
\move(-3 0)\rlvec(0 11.3) \move(-4 0)\rlvec(0 11.3)
\move(-1 0)\rlvec(1 1) \move(-2 0)\rlvec(1 1) \move(-3 0)\rlvec(1 1)
\move(-4 0)\rlvec(1 1) \move(-1 9)\rlvec(1 1) \move(-2 9)\rlvec(1 1)
\move(-3 9)\rlvec(1 1) \move(-4 9)\rlvec(1 1)
\move(0 5)\rlvec(-4.3 0)
\htext(-0.4 0.1){$0$} \htext(-0.85 0.6){$1$} \htext(-0.6 1.35){$2$}
\vtext(-0.6 2.6){$\cdots$} \htext(-0.9 3.85){$n\!\!-\!\!1$}
\htext(-0.9 5.85){$n\!\!-\!\!1$} \htext(-0.6 8.35){$2$}
\htext(-0.4 9.1){$0$} \htext(-0.85 9.6){$1$} \htext(-0.6 10.35){$2$}
\htext(-2.4 0.1){$0$} \htext(-2.85 0.6){$1$} \htext(-2.6 1.35){$2$}
\vtext(-2.4 2.6){$\cdots$} \htext(-2.9 3.85){$n\!\!-\!\!1$}
\htext(-2.9 5.85){$n\!\!-\!\!1$} \htext(-2.6 8.35){$2$}
\htext(-2.4 9.1){$0$} \htext(-2.85 9.6){$1$} \htext(-2.6 10.35){$2$}
\htext(-1.4 0.1){$1$} \htext(-1.85 0.6){$0$} \htext(-1.6 1.35){$2$}
\vtext(-1.4 2.6){$\cdots$} \htext(-1.9 3.85){$n\!\!-\!\!1$}
\htext(-1.9 5.85){$n\!\!-\!\!1$} \htext(-1.6 8.35){$2$}
\htext(-1.4 9.1){$1$} \htext(-1.85 9.6){$0$} \htext(-1.6 10.35){$2$}
\htext(-3.4 0.1){$1$} \htext(-3.85 0.6){$0$} \htext(-3.6 1.35){$2$}
\vtext(-3.4 2.6){$\cdots$} \htext(-3.9 3.85){$n\!\!-\!\!1$}
\htext(-3.9 5.85){$n\!\!-\!\!1$} \htext(-3.6 8.35){$2$}
\htext(-3.4 9.1){$1$} \htext(-3.85 9.6){$0$} \htext(-3.6 10.35){$2$}
\htext(-0.6 4.6){$n$} \htext(-2.6 4.6){$n$} \htext(-1.6 5.1){$n$}
\htext(-3.6 5.1){$n$} \htext(-1.6 4.6){$n$} \htext(-3.6 4.6){$n$}
\htext(-0.6 5.1){$n$} \htext(-2.6 5.1){$n$}
\vtext(-0.6 7.1){$\cdots$} \vtext(-1.4 7.1){$\cdots$}
\vtext(-2.4 7.1){$\cdots$} \vtext(-3.4 7.1){$\cdots$}
\end{texdraw}
}
\newcommand{\FPatBnO}{
\begin{texdraw}
\fontsize{7}{7}\selectfont
\drawdim em
\setunitscale 1.9
\move(0 0)\dtri \move(-1 0)\dtri \move(-2 0)\dtri \move(-3 0)\dtri
\move(0 0)\rlvec(-4.3 0) \move(0 1)\rlvec(-4.3 0) \move(0 2)\rlvec(-4.3 0)
\move(0 3.5)\rlvec(-4.3 0) \move(0 4.5)\rlvec(-4.3 0)
\move(0 5.5)\rlvec(-4.3 0) \move(0 6.5)\rlvec(-4.3 0)
\move(0 8)\rlvec(-4.3 0) \move(0 9)\rlvec(-4.3 0)
\move(0 10)\rlvec(-4.3 0) \move(0 11)\rlvec(-4.3 0)
\move(0 0)\rlvec(0 11.3) \move(-1 0)\rlvec(0 11.3) \move(-2 0)\rlvec(0 11.3)
\move(-3 0)\rlvec(0 11.3) \move(-4 0)\rlvec(0 11.3)
\move(-1 0)\rlvec(1 1) \move(-2 0)\rlvec(1 1) \move(-3 0)\rlvec(1 1)
\move(-4 0)\rlvec(1 1) \move(-1 9)\rlvec(1 1) \move(-2 9)\rlvec(1 1)
\move(-3 9)\rlvec(1 1) \move(-4 9)\rlvec(1 1)
\move(0 5)\rlvec(-4.3 0)
\htext(-0.4 0.1){$1$} \htext(-0.85 0.6){$0$}
\vtext(-0.4 2.6){$\cdots$}
\htext(-0.9 3.85){$n\!\!-\!\!1$} \htext(-0.9 5.85){$n\!\!-\!\!1$}
\htext(-2.9 3.85){$n\!\!-\!\!1$} \htext(-2.9 5.85){$n\!\!-\!\!1$}
\htext(-1.9 3.85){$n\!\!-\!\!1$} \htext(-1.9 5.85){$n\!\!-\!\!1$}
\htext(-3.9 3.85){$n\!\!-\!\!1$} \htext(-3.9 5.85){$n\!\!-\!\!1$}
\htext(-0.6 8.35){$2$}\htext(-0.6 10.35){$2$}\htext(-2.6 1.35){$2$}
\htext(-2.6 10.35){$2$}\htext(-1.6 1.35){$2$}\htext(-0.6 1.35){$2$}
\htext(-2.6 8.35){$2$}\htext(-1.6 8.35){$2$}\htext(-1.6 10.35){$2$}
\htext(-3.6 1.35){$2$}\htext(-3.6 8.35){$2$}\htext(-3.6 10.35){$2$}
\htext(-0.4 9.1){$1$} \htext(-0.85 9.6){$0$}
\htext(-2.4 0.1){$1$} \htext(-2.85 0.6){$0$}
\vtext(-2.4 2.75){$\cdots$}
\htext(-2.4 9.1){$1$} \htext(-2.85 9.6){$0$}
\htext(-1.4 0.1){$0$} \htext(-1.85 0.6){$1$}
\vtext(-1.4 2.6){$\cdots$}
\htext(-1.4 9.1){$0$} \htext(-1.85 9.6){$1$}
\htext(-3.4 0.1){$0$} \htext(-3.85 0.6){$1$}
\vtext(-3.4 2.6){$\cdots$}
\htext(-3.4 9.1){$0$} \htext(-3.85 9.6){$1$}
\htext(-0.6 4.6){$n$} \htext(-2.6 4.6){$n$} \htext(-1.6 5.1){$n$}
\htext(-3.6 5.1){$n$} \htext(-1.6 4.6){$n$} \htext(-3.6 4.6){$n$}
\htext(-0.6 5.1){$n$} \htext(-2.6 5.1){$n$}
\vtext(-0.4 7.1){$\cdots$} \vtext(-1.4 7.1){$\cdots$}
\vtext(-2.4 7.1){$\cdots$} \vtext(-3.4 7.1){$\cdots$}
\end{texdraw}
}
\newcommand{\Bthreeoneex}{
\begin{texdraw}
\drawdim em
\setunitscale 1.7
\move(0 0)
\bsegment
\move(3 0)\dtri
\move(2 0)\dtri
\move(1 0)\dtri
\move(2 4)\lvec(3 5)\lvec(3 4)
\move(2 3)\lvec(2 4)\lvec(3 4)\lvec(3 3)
\move(2 1)\lvec(2 3)\lvec(3 3)\lvec(3 2)
\move(1 2)\lvec(1 2.5)\lvec(3 2.5)\lvec(3 2)
\move(1 1)\lvec(1 2)\lvec(3 2)\lvec(3 1)
\move(0 0)\lvec(0 1)\lvec(3 1)\lvec(3 0)
\move(0 0)\lvec(1 0)\lvec(1 1)\lvec(0 0)
\move(1 0)\lvec(2 0)\lvec(2 1)\lvec(1 0)
\move(2 0)\lvec(3 0)\lvec(3 1)\lvec(2 0)
\htext(0.65 0.12){$1$}
\htext(1.65 0.12){$0$}
\htext(2.65 0.12){$1$}
\htext(0.15 0.52){$0$}
\htext(1.15 0.52){$1$}
\htext(2.15 0.52){$0$}
\htext(2.65 4.12){$1$}
\htext(1.4 1.35){$2$}\htext(2.4 1.35){$2$}
\htext(1.4 2.1){$3$}\htext(2.4 2.1){$3$}\htext(2.4 2.6){$3$}
\htext(2.4 3.35){$2$}
\esegment
\end{texdraw}
}

\newcommand{\Bthreeoneextwo}{
\begin{texdraw}
\drawdim em
\setunitscale 1.7
\move(0 0)
\bsegment
\htext(-1.5 2){$Y^1=$}
\move(3 0)\dtri
\move(2 0)\dtri
\move(1 0)\dtri
\move(2 1)\lvec(2 2)
\move(1 1)\lvec(1 2)\lvec(3 2)\lvec(3 1)
\move(0 0)\lvec(0 1)\lvec(3 1)\lvec(3 0)
\move(0 0)\lvec(1 0)\lvec(1 1)\lvec(0 0)
\move(1 0)\lvec(2 0)\lvec(2 1)\lvec(1 0)
\move(2 0)\lvec(3 0)\lvec(3 1)\lvec(2 0)
\htext(0.65 0.12){$1$}
\htext(1.65 0.12){$0$}
\htext(2.65 0.12){$1$}
\htext(0.15 0.52){$0$}
\htext(1.15 0.52){$1$}
\htext(2.15 0.52){$0$}
\htext(1.4 1.35){$2$}\htext(2.4 1.35){$2$}
\esegment
\move(4.5 0)
\bsegment
\htext(-0.5 2){$Y^2=$}
\move(3 0)\dtri
\move(2 0)\dtri
\move(2 4)\lvec(3 5)\lvec(3 4)
\move(2 3)\lvec(2 4)\lvec(3 4)\lvec(3 3)
\move(2 1)\lvec(2 3)\lvec(3 3)\lvec(3 2)
\move(1 2.5)\lvec(1 4)\lvec(2 5)\lvec(2 4)\lvec(1 4)
\move(1 3)\lvec(2 3)
\move(1 2)\lvec(1 2.5)\lvec(3 2.5)\lvec(3 2)
\move(1 1)\lvec(1 2)\lvec(3 2)\lvec(3 1)
\move(1 0)\lvec(1 1)\lvec(3 1)\lvec(3 0)
\move(1 0)\lvec(2 0)\lvec(2 1)\lvec(1 0)
\move(2 0)\lvec(3 0)\lvec(3 1)\lvec(2 0)
\htext(1.65 0.12){$0$}
\htext(2.65 0.12){$1$}
\htext(1.15 0.52){$1$}
\htext(2.15 0.52){$0$}
\htext(2.65 4.12){$1$}\htext(1.65 4.12){$0$}
\htext(1.4 1.35){$2$}\htext(2.4 1.35){$2$}
\htext(1.4 2.6){$3$}\htext(1.4 2.1){$3$}\htext(2.4 2.1){$3$}\htext(2.4 2.6){$3$}
\htext(2.4 3.35){$2$}\htext(1.4 3.35){$2$}
\esegment
\move(10 0)
\bsegment
\htext(-1.5 2){$Y^3=$}
\move(3 0)\dtri
\move(2 0)\dtri
\move(1 0)\dtri
\move(2 4)\lvec(3 5)\lvec(3 4)
\move(2 3)\lvec(2 4)\lvec(3 4)\lvec(3 3)
\move(2 1)\lvec(2 3)\lvec(3 3)\lvec(3 2)
\move(1 2.5)\lvec(1 4)\lvec(2 5)\lvec(2 4)\lvec(1 4)
\move(1 3)\lvec(2 3)
\move(1 2)\lvec(1 2.5)\lvec(3 2.5)\lvec(3 2)
\move(1 1)\lvec(1 2)\lvec(3 2)\lvec(3 1)
\move(0 0)\lvec(0 1)\lvec(3 1)\lvec(3 0)
\move(0 0)\lvec(1 0)\lvec(1 1)\lvec(0 0)
\move(1 0)\lvec(2 0)\lvec(2 1)\lvec(1 0)
\move(2 0)\lvec(3 0)\lvec(3 1)\lvec(2 0)
\htext(0.65 0.12){$1$}
\htext(1.65 0.12){$0$}
\htext(2.65 0.12){$1$}
\htext(0.15 0.52){$0$}
\htext(1.15 0.52){$1$}
\htext(2.15 0.52){$0$}
\htext(2.65 4.12){$1$}\htext(1.65 4.12){$0$}
\htext(1.4 1.35){$2$}
\htext(1.4 1.35){$2$}\htext(2.4 1.35){$2$}
\htext(1.4 2.6){$3$}
\htext(1.4 2.1){$3$}\htext(2.4 2.1){$3$}\htext(2.4 2.6){$3$}
\htext(1.4 3.35){$2$}\htext(2.4 3.35){$2$}
\esegment
\move(16 0)
\bsegment
\htext(-1.5 2){$Y^4=$}
\move(3 0)\dtri
\move(2 0)\dtri
\move(1 0)\dtri
\move(2 4)\lvec(3 5)\lvec(2 5)
\move(2 3)\lvec(2 4)\lvec(3 4)\lvec(3 3)
\move(2 1)\lvec(2 3)\lvec(3 3)\lvec(3 2)
\move(1 2.5)\lvec(1 4)\lvec(2 5)\lvec(2 4)\lvec(1 4)
\move(1 3)\lvec(2 3)
\move(1 4)\lvec(1 5)\lvec(2 5)
\move(1 2)\lvec(1 2.5)\lvec(3 2.5)\lvec(3 2)
\move(1 1)\lvec(1 2)\lvec(3 2)\lvec(3 1)
\move(0 0)\lvec(0 1)\lvec(3 1)\lvec(3 0)
\move(0 0)\lvec(1 0)\lvec(1 1)\lvec(0 0)
\move(1 0)\lvec(2 0)\lvec(2 1)\lvec(1 0)
\move(2 0)\lvec(3 0)\lvec(3 1)\lvec(2 0)
\htext(0.65 0.12){$1$}
\htext(1.65 0.12){$0$}
\htext(2.65 0.12){$1$}
\htext(0.15 0.52){$0$}
\htext(1.15 0.52){$1$}
\htext(2.15 0.52){$0$}
\htext(1.15 4.52){$1$}
\htext(2.15 4.52){$0$}
\htext(1.4 1.35){$2$}
\htext(1.4 1.35){$2$}\htext(2.4 1.35){$2$}
\htext(1.4 2.6){$3$}
\htext(1.4 2.1){$3$}\htext(2.4 2.1){$3$}\htext(2.4 2.6){$3$}
\htext(1.4 3.35){$2$}\htext(2.4 3.35){$2$}
\esegment
\end{texdraw}
}

\newcommand{\FoPatDnO}{
\begin{texdraw}
\fontsize{7}{7}\selectfont
\drawdim em
\setunitscale 1.9
\move(0 0)\dtri \move(-1 0)\dtri \move(-2 0)\dtri \move(-3 0)\dtri
\move(0 0)\rlvec(-4.3 0) \move(0 1)\rlvec(-4.3 0) \move(0 2)\rlvec(-4.3 0)
\move(0 3.5)\rlvec(-4.3 0) \move(0 4.5)\rlvec(-4.3 0)
\move(0 5.5)\rlvec(-4.3 0) \move(0 6.5)\rlvec(-4.3 0)
\move(0 8)\rlvec(-4.3 0) \move(0 9)\rlvec(-4.3 0)
\move(0 10)\rlvec(-4.3 0) \move(0 11)\rlvec(-4.3 0)
\move(0 0)\rlvec(0 11.3) \move(-1 0)\rlvec(0 11.3)
\move(-2 0)\rlvec(0 11.3) \move(-3 0)\rlvec(0 11.3) \move(-4 0)\rlvec(0 11.3)
\move(-1 4.5)\rlvec(1 1) \move(-2 4.5)\rlvec(1 1) \move(-3 4.5)\rlvec(1 1)
\move(-4 4.5)\rlvec(1 1)
\htext(-0.85 0.6){$n$}
\htext(-0.9 1.35){$n\!\!-\!\!2$} \vtext(-0.4 2.6){$\cdots$} \htext(-0.6 3.85){$2$}
\htext(-0.4 4.6){$1$} \htext(-0.85 5.1){$0$} \htext(-0.6 5.85){$2$}
\vtext(-0.4 7.1){$\cdots$} \htext(-0.9 8.35){$n\!\!-\!\!2$}
\htext(-0.85 9.6){$n$}
\htext(-0.9 10.35){$n\!\!-\!\!2$}
\htext(-2.85 0.6){$n$}
\htext(-2.9 1.35){$n\!\!-\!\!2$} \vtext(-2.4 2.6){$\cdots$} \htext(-2.6 3.85){$2$}
\htext(-2.4 4.6){$1$} \htext(-2.85 5.1){$0$} \htext(-2.6 5.85){$2$}
\vtext(-2.4 7.1){$\cdots$} \htext(-2.9 8.35){$n\!\!-\!\!2$}
\htext(-2.85 9.6){$n$}
\htext(-2.9 10.35){$n\!\!-\!\!2$}
\htext(-1.4 0.1){$n$}
\htext(-1.9 1.35){$n\!\!-\!\!2$} \vtext(-1.4 2.6){$\cdots$} \htext(-1.6 3.85){$2$}
\htext(-1.4 4.6){$0$} \htext(-1.85 5.1){$1$} \htext(-1.6 5.85){$2$}
\vtext(-1.4 7.1){$\cdots$} \htext(-1.9 8.35){$n\!\!-\!\!2$}
\htext(-1.4 9.1){$n$}
\htext(-1.9 10.35){$n\!\!-\!\!2$}
\htext(-3.4 0.1){$n$}
\htext(-3.9 1.35){$n\!\!-\!\!2$} \vtext(-3.4 2.6){$\cdots$} \htext(-3.6 3.85){$2$}
\htext(-3.4 4.6){$0$} \htext(-3.85 5.1){$1$} \htext(-3.6 5.85){$2$}
\vtext(-3.4 7.1){$\cdots$} \htext(-3.9 8.35){$n\!\!-\!\!2$}
\htext(-3.4 9.1){$n$}
\htext(-3.9 10.35){$n\!\!-\!\!2$}
\move(-2 0)\rlvec(0.5 0.5)\rmove(0.45 0.45)\rlvec(0.05 0.05)
\move(-1 0)\rlvec(0.05 0.05)\rmove(0.45 0.45)\rlvec(0.5 0.5)
\move(-4 0)\rlvec(0.5 0.5)\rmove(0.45 0.45)\rlvec(0.05 0.05)
\move(-3 0)\rlvec(0.05 0.05)\rmove(0.45 0.45)\rlvec(0.5 0.5)
\move(-2 9)\rlvec(0.5 0.5)\rmove(0.45 0.45)\rlvec(0.05 0.05)
\move(-1 9)\rlvec(0.05 0.05)\rmove(0.45 0.45)\rlvec(0.5 0.5)
\move(-4 9)\rlvec(0.5 0.5)\rmove(0.45 0.45)\rlvec(0.05 0.05)
\move(-3 9)\rlvec(0.05 0.05)\rmove(0.45 0.45)\rlvec(0.5 0.5)
\htext(-0.9 0.1){$n\!\!-\!\!1$} \htext(-0.9 9.1){$n\!\!-\!\!1$}
\htext(-2.9 0.1){$n\!\!-\!\!1$} \htext(-2.9 9.1){$n\!\!-\!\!1$}
\htext(-1.9 0.6){$n\!\!-\!\!1$} \htext(-1.9 9.6){$n\!\!-\!\!1$}
\htext(-3.9 0.6){$n\!\!-\!\!1$} \htext(-3.9 9.6){$n\!\!-\!\!1$}
\end{texdraw}
}

\newcommand{\TPatDnO}{
\begin{texdraw}
\fontsize{7}{7}\selectfont
\drawdim em
\setunitscale 1.9
\move(0 0)\dtri \move(-1 0)\dtri \move(-2 0)\dtri \move(-3 0)\dtri
\move(0 0)\rlvec(-4.3 0) \move(0 1)\rlvec(-4.3 0) \move(0 2)\rlvec(-4.3 0)
\move(0 3.5)\rlvec(-4.3 0) \move(0 4.5)\rlvec(-4.3 0) \move(0 5.5)\rlvec(-4.3 0)
\move(0 6.5)\rlvec(-4.3 0) \move(0 8)\rlvec(-4.3 0) \move(0 9)\rlvec(-4.3 0)
\move(0 10)\rlvec(-4.3 0) \move(0 11)\rlvec(-4.3 0)
\move(0 0)\rlvec(0 11.3) \move(-1 0)\rlvec(0 11.3) \move(-2 0)\rlvec(0 11.3)
\move(-3 0)\rlvec(0 11.3) \move(-4 0)\rlvec(0 11.3)
\move(-1 4.5)\rlvec(1 1) \move(-2 4.5)\rlvec(1 1) \move(-3 4.5)\rlvec(1 1)
\move(-4 4.5)\rlvec(1 1)
\htext(-0.4 0.1){$n$}
\htext(-0.9 1.35){$n\!\!-\!\!2$} \vtext(-0.4 2.6){$\cdots$} \htext(-0.6 3.85){$2$}
\htext(-0.4 4.6){$1$} \htext(-0.85 5.1){$0$} \htext(-0.6 5.85){$2$}
\vtext(-0.4 7.1){$\cdots$} \htext(-0.9 8.35){$n\!\!-\!\!2$}
\htext(-0.4 9.1){$n$}
\htext(-0.9 10.35){$n\!\!-\!\!2$}
\htext(-2.4 0.1){$n$}
\htext(-2.9 1.35){$n\!\!-\!\!2$} \vtext(-2.4 2.6){$\cdots$} \htext(-2.6 3.85){$2$}
\htext(-2.4 4.6){$1$} \htext(-2.85 5.1){$0$} \htext(-2.6 5.85){$2$}
\vtext(-2.4 7.1){$\cdots$} \htext(-2.9 8.35){$n\!\!-\!\!2$}
\htext(-2.4 9.1){$n$}
\htext(-2.9 10.35){$n\!\!-\!\!2$}
\htext(-1.85 0.6){$n$}
\htext(-1.9 1.35){$n\!\!-\!\!2$} \vtext(-1.4 2.6){$\cdots$} \htext(-1.6 3.85){$2$}
\htext(-1.4 4.6){$0$} \htext(-1.85 5.1){$1$} \htext(-1.6 5.85){$2$}
\vtext(-1.4 7.1){$\cdots$} \htext(-1.9 8.35){$n\!\!-\!\!2$}
\htext(-1.85 9.6){$n$}
\htext(-1.9 10.35){$n\!\!-\!\!2$}
\htext(-3.85 0.6){$n$}
\htext(-3.9 1.35){$n\!\!-\!\!2$} \vtext(-3.4 2.6){$\cdots$} \htext(-3.6 3.85){$2$}
\htext(-3.4 4.6){$0$} \htext(-3.85 5.1){$1$} \htext(-3.6 5.85){$2$}
\vtext(-3.4 7.1){$\cdots$} \htext(-3.9 8.35){$n\!\!-\!\!2$}
\htext(-3.85 9.6){$n$}
\htext(-3.9 10.35){$n\!\!-\!\!2$}
\move(-1 0)\rlvec(0.5 0.5)\rmove(0.45 0.45)\rlvec(0.05 0.05)
\move(-2 0)\rlvec(0.05 0.05)\rmove(0.45 0.45)\rlvec(0.5 0.5)
\move(-3 0)\rlvec(0.5 0.5)\rmove(0.45 0.45)\rlvec(0.05 0.05)
\move(-4 0)\rlvec(0.05 0.05)\rmove(0.45 0.45)\rlvec(0.5 0.5)
\move(-1 9)\rlvec(0.5 0.5)\rmove(0.45 0.45)\rlvec(0.05 0.05)
\move(-2 9)\rlvec(0.05 0.05)\rmove(0.45 0.45)\rlvec(0.5 0.5)
\move(-3 9)\rlvec(0.5 0.5)\rmove(0.45 0.45)\rlvec(0.05 0.05)
\move(-4 9)\rlvec(0.05 0.05)\rmove(0.45 0.45)\rlvec(0.5 0.5)
\htext(-3.9 9.1){$n\!\!-\!\!1$} \htext(-3.9 0.1){$n\!\!-\!\!1$}
\htext(-2.9 9.6){$n\!\!-\!\!1$} \htext(-2.9 0.6){$n\!\!-\!\!1$}
\htext(-1.9 0.1){$n\!\!-\!\!1$} \htext(-1.9 9.1){$n\!\!-\!\!1$}
\htext(-0.9 9.6){$n\!\!-\!\!1$} \htext(-0.9 0.6){$n\!\!-\!\!1$}
\end{texdraw}
}

\newcommand{\SPatDnO}{
\begin{texdraw}
\fontsize{7}{7}\selectfont
\drawdim em
\setunitscale 1.9
\move(0 0)\dtri \move(-1 0)\dtri \move(-2 0)\dtri \move(-3 0)\dtri
\move(0 0)\rlvec(-4.3 0) \move(0 1)\rlvec(-4.3 0) \move(0 2)\rlvec(-4.3 0)
\move(0 3.5)\rlvec(-4.3 0) \move(0 4.5)\rlvec(-4.3 0)
\move(0 5.5)\rlvec(-4.3 0) \move(0 6.5)\rlvec(-4.3 0)
\move(0 8)\rlvec(-4.3 0) \move(0 9)\rlvec(-4.3 0) \move(0 10)\rlvec(-4.3 0)
\move(0 11)\rlvec(-4.3 0) \move(0 0)\rlvec(0 11.3) \move(-1 0)\rlvec(0 11.3)
\move(-2 0)\rlvec(0 11.3) \move(-3 0)\rlvec(0 11.3) \move(-4 0)\rlvec(0 11.3)
\move(-1 0)\rlvec(1 1) \move(-2 0)\rlvec(1 1)
\move(-3 0)\rlvec(1 1) \move(-4 0)\rlvec(1 1)
\move(-1 9)\rlvec(1 1) \move(-2 9)\rlvec(1 1)
\move(-3 9)\rlvec(1 1) \move(-4 9)\rlvec(1 1)
\htext(-0.4 0.1){$0$} \htext(-0.85 0.6){$1$} \htext(-0.6 1.35){$2$}
\vtext(-0.4 2.6){$\cdots$} \htext(-0.9 3.85){$n\!\!-\!\!2$}
\htext(-0.9 5.85){$n\!\!-\!\!2$} \htext(-0.6 8.35){$2$}
\htext(-0.4 9.1){$0$} \htext(-0.85 9.6){$1$} \htext(-0.6 10.35){$2$}
\htext(-2.4 0.1){$0$} \htext(-2.85 0.6){$1$} \htext(-2.6 1.35){$2$}
\vtext(-2.4 2.6){$\cdots$} \htext(-2.9 3.85){$n\!\!-\!\!2$}
\htext(-2.9 5.85){$n\!\!-\!\!2$} \htext(-2.6 8.35){$2$}
\htext(-2.4 9.1){$0$} \htext(-2.85 9.6){$1$} \htext(-2.6 10.35){$2$}
\htext(-1.4 0.1){$1$} \htext(-1.85 0.6){$0$} \htext(-1.6 1.35){$2$}
\vtext(-1.4 2.6){$\cdots$} \htext(-1.9 3.85){$n\!\!-\!\!2$}
\htext(-1.9 5.85){$n\!\!-\!\!2$} \htext(-1.6 8.35){$2$}
\htext(-1.4 9.1){$1$} \htext(-1.85 9.6){$0$} \htext(-1.6 10.35){$2$}
\htext(-3.4 0.1){$1$} \htext(-3.85 0.6){$0$} \htext(-3.6 1.35){$2$}
\vtext(-3.4 2.6){$\cdots$} \htext(-3.9 3.85){$n\!\!-\!\!2$}
\htext(-3.9 5.85){$n\!\!-\!\!2$} \htext(-3.6 8.35){$2$}
\htext(-3.4 9.1){$1$} \htext(-3.85 9.6){$0$} \htext(-3.6 10.35){$2$}
\htext(-0.4 4.6){$n$} \htext(-2.4 4.6){$n$} \htext(-1.85 5.1){$n$}
\htext(-3.85 5.1){$n$}
\move(-1 4.5)\rlvec(0.5 0.5)\rmove(0.45 0.45)\rlvec(0.05 0.05)
\move(-2 4.5)\rlvec(0.05 0.05)\rmove(0.45 0.45)\rlvec(0.5 0.5)
\move(-3 4.5)\rlvec(0.5 0.5)\rmove(0.45 0.45)\rlvec(0.05 0.05)
\move(-4 4.5)\rlvec(0.05 0.05)\rmove(0.45 0.45)\rlvec(0.5 0.5)
\htext(-1.9 4.6){$n\!\!-\!\!1$} \htext(-3.9 4.6){$n\!\!-\!\!1$}
\htext(-0.9 5.1){$n\!\!-\!\!1$} \htext(-2.9 5.1){$n\!\!-\!\!1$}
\vtext(-0.4 7.1){$\cdots$} \vtext(-1.4 7.1){$\cdots$}
\vtext(-2.4 7.1){$\cdots$} \vtext(-3.4 7.1){$\cdots$}
\end{texdraw}
}
\newcommand{\FPatDnO}{
\begin{texdraw}
\fontsize{7}{7}\selectfont
\drawdim em
\setunitscale 1.9
\move(0 0)\dtri \move(-1 0)\dtri \move(-2 0)\dtri \move(-3 0)\dtri
\move(0 0)\rlvec(-4.3 0) \move(0 1)\rlvec(-4.3 0) \move(0 2)\rlvec(-4.3 0)
\move(0 3.5)\rlvec(-4.3 0) \move(0 4.5)\rlvec(-4.3 0)
\move(0 5.5)\rlvec(-4.3 0) \move(0 6.5)\rlvec(-4.3 0)
\move(0 8)\rlvec(-4.3 0) \move(0 9)\rlvec(-4.3 0)
\move(0 10)\rlvec(-4.3 0) \move(0 11)\rlvec(-4.3 0)
\move(0 0)\rlvec(0 11.3) \move(-1 0)\rlvec(0 11.3) \move(-2 0)\rlvec(0 11.3)
\move(-3 0)\rlvec(0 11.3) \move(-4 0)\rlvec(0 11.3)
\move(-1 0)\rlvec(1 1) \move(-2 0)\rlvec(1 1) \move(-3 0)\rlvec(1 1)
\move(-4 0)\rlvec(1 1) \move(-1 9)\rlvec(1 1) \move(-2 9)\rlvec(1 1)
\move(-3 9)\rlvec(1 1) \move(-4 9)\rlvec(1 1)
\htext(-0.4 0.1){$1$} \htext(-0.85 0.6){$0$} \htext(-0.6 1.35){$2$}
\vtext(-0.4 2.6){$\cdots$} \htext(-0.9 3.85){$n\!\!-\!\!2$}
\htext(-0.9 5.85){$n\!\!-\!\!2$} \htext(-0.6 8.35){$2$}
\htext(-0.4 9.1){$1$} \htext(-0.85 9.6){$0$} \htext(-0.6 10.35){$2$}
\htext(-2.4 0.1){$1$} \htext(-2.85 0.6){$0$} \htext(-2.6 1.35){$2$}
\vtext(-2.6 2.6){$\cdots$} \htext(-2.9 3.85){$n\!\!-\!\!2$}
\htext(-2.9 5.85){$n\!\!-\!\!2$} \htext(-2.6 8.35){$2$}
\htext(-2.4 9.1){$1$} \htext(-2.85 9.6){$0$} \htext(-2.6 10.35){$2$}
\htext(-1.4 0.1){$0$} \htext(-1.85 0.6){$1$} \htext(-1.6 1.35){$2$}
\vtext(-1.4 2.75){$\cdots$} \htext(-1.9 3.85){$n\!\!-\!\!2$}
\htext(-1.9 5.85){$n\!\!-\!\!2$} \htext(-1.6 8.35){$2$}
\htext(-1.4 9.1){$0$} \htext(-1.85 9.6){$1$} \htext(-1.6 10.35){$2$}
\htext(-3.4 0.1){$0$} \htext(-3.85 0.6){$1$} \htext(-3.6 1.35){$2$}
\vtext(-3.4 2.6){$\cdots$} \htext(-3.9 3.85){$n\!\!-\!\!2$}
\htext(-3.9 5.85){$n\!\!-\!\!2$} \htext(-3.6 8.35){$2$}
\htext(-3.4 9.1){$0$} \htext(-3.85 9.6){$1$} \htext(-3.6 10.35){$2$}
\htext(-0.4 4.6){$n$} \htext(-2.4 4.6){$n$} \htext(-1.85 5.1){$n$}
\htext(-3.85 5.1){$n$}
\move(-1 4.5)\rlvec(0.5 0.5)\rmove(0.45 0.45)\rlvec(0.05 0.05)
\move(-2 4.5)\rlvec(0.05 0.05)\rmove(0.45 0.45)\rlvec(0.5 0.5)
\move(-3 4.5)\rlvec(0.5 0.5)\rmove(0.45 0.45)\rlvec(0.05 0.05)
\move(-4 4.5)\rlvec(0.05 0.05)\rmove(0.45 0.45)\rlvec(0.5 0.5)
\htext(-1.9 4.6){$n\!\!-\!\!1$} \htext(-3.9 4.6){$n\!\!-\!\!1$}
\htext(-0.9 5.1){$n\!\!-\!\!1$} \htext(-2.9 5.1){$n\!\!-\!\!1$}
\vtext(-0.4 7.1){$\cdots$} \vtext(-1.4 7.1){$\cdots$}
\vtext(-2.4 7.1){$\cdots$} \vtext(-3.4 7.1){$\cdots$}
\end{texdraw}
}

\newcommand{\sbk}[1]
{
\fontsize{9}{9}\selectfont
\xy (-1.5,1.5)*{}="T1"; (1.5,1.5)*{}="T2"; (-1.5,-1.5)*{}="B1"; (1.5,-1.5)*{}="B2";
"T1"; "T2" **\dir{-};"T1"; "B1" **\dir{-};"T2"; "B2" **\dir{-}; "B1"; "B2" **\dir{-};
(-0.2,0)*{#1};
\endxy
\fontsize{10}{10}\selectfont
}
\newcommand{\Bhbk}[1]
{
\fontsize{5}{5}\selectfont
\xy (-1.5,1.5)*{}="T1"; (1.5,1.5)*{}="T2"; (-1.5,0)*{}="B1"; (1.5,0)*{}="B2";
"T1"; "T2" **\dir{-};"T1"; "B1" **\dir{-};"T2"; "B2" **\dir{-}; "B1"; "B2" **\dir{-};
(-0.2,0.6)*{#1};
\endxy
\fontsize{10}{10}\selectfont
}

\newcommand{\hbk}[1]
{
\fontsize{5}{5}\selectfont
\xy (-1.5,1.5)*{}="T1"; (1.5,1.5)*{}="T2"; (-1.5,0)*{}="B1"; (1.5,0)*{}="B2";
"T1"; "T2" **\dir{-};"T1"; "B1" **\dir{-};"T2"; "B2" **\dir{-};
(-0.2,0.6)*{#1};
\endxy
\fontsize{10}{10}\selectfont
}

\newcommand{\sone}
{\xy (0,0)*++{\Bhbk{0}};\endxy}
\newcommand{\stwo}
{\xy (0,0)*++{\Bhbk{0}};(0,2.8)*++{\sbk{1}};\endxy}
\newcommand{\sthree}
{\xy (0,0)*++{\Bhbk{0}};(0,2.8)*++{\sbk{1}};(0,4.5)*++{\hbk{2}};\endxy}
\newcommand{\sfour}
{\xy (0,0)*++{\Bhbk{0}};(0,2.8)*++{\sbk{1}};(0,4.5)*++{\hbk{2}};(0,6.1)*++{\hbk{2}};\endxy}
\newcommand{\sfive}
{\xy (0,0)*++{\Bhbk{0}};(0,2.8)*++{\sbk{1}};(0,4.5)*++{\hbk{2}};(0,6.1)*++{\hbk{2}};(0,8.9)*++{\sbk{1}};\endxy}
\newcommand{\ssix}
{\xy (0,0)*++{\Bhbk{0}};(0,2.8)*++{\sbk{1}};(0,4.5)*++{\hbk{2}};(0,6.1)*++{\hbk{2}};(0,8.9)*++{\sbk{1}};
     (0,10.7)*++{\hbk{0}};\endxy}
\newcommand{\sseven}
{\xy (0,0)*++{\Bhbk{0}};(0,2.8)*++{\sbk{1}};(0,4.5)*++{\hbk{2}};(0,6.1)*++{\hbk{2}};(0,8.9)*++{\sbk{1}};
     (0,10.7)*++{\hbk{0}};(0,12.3)*++{\hbk{0}};\endxy}
\newcommand{\seight}
{\xy (0,0)*++{\Bhbk{0}};(0,2.8)*++{\sbk{1}};(0,4.5)*++{\hbk{2}};(0,6.1)*++{\hbk{2}};(0,8.9)*++{\sbk{1}};
     (0,10.7)*++{\hbk{0}};(0,12.3)*++{\hbk{0}};(0,15.1)*++{\sbk{1}};\endxy}
\newcommand{\snine}
{\xy (0,0)*++{\Bhbk{0}};(0,2.8)*++{\sbk{1}};(0,4.5)*++{\hbk{2}};(0,6.1)*++{\hbk{2}};(0,8.9)*++{\sbk{1}};
     (0,10.7)*++{\hbk{0}};(0,12.3)*++{\hbk{0}};(0,15.1)*++{\sbk{1}};(0,16.9)*++{\hbk{2}};\endxy}

\newcommand{\Go}{
{\xy (0,-5)*++{\sone};\endxy}
}
\newcommand{\Gt}{
{\xy (0,-3.4)*++{\stwo};\endxy}
}
\newcommand{\Gto}{
{\xy (-3,-4.8)*++{\sone};(0,-3.4)*++{\stwo};\endxy}
}
\newcommand{\Gth}{
{\xy (0,-2.6)*++{\sthree};\endxy}
}
\newcommand{\Gtho}{
{\xy (-3,-4.8)*++{\sone};(0,-2.6)*++{\sthree};\endxy}
}
\newcommand{\Gf}{
{\xy (0,-1.8)*++{\sfour};\endxy}
}
\newcommand{\Gtht}{
{\xy (-3,-3.4)*++{\stwo};(0,-2.6)*++{\sthree};\endxy}
}
\newcommand{\Gfo}{
{\xy (-3,-4.8)*++{\sone};(0,-1.8)*++{\sfour};\endxy}
}
\newcommand{\Gfi}{
{\xy (0,-0.3)*++{\sfive};\endxy}
}
\newcommand{\Gfio}{
{\xy (-3,-4.8)*++{\sone};(0,-0.3)*++{\sfive};\endxy}
}
\newcommand{\Gft}{
{\xy (-3,-3.4)*++{\stwo};(0,-0.3)*++{\sfive};\endxy}
}
\newcommand{\Gthto}{
{\xy (-6,-4.8)*++{\sone};(-3,-3.4)*++{\stwo};(0,-2.6)*++{\sthree};\endxy}
}
\newcommand{\Gs}{
{\xy (0,0.5)*++{\ssix};\endxy}
}
\newcommand{\Gso}{
{\xy (-3,-5)*++{\sone};(0,0.5)*++{\ssix};\endxy}
}
\newcommand{\Gfit}{
{\xy (-3,-3.4)*++{\stwo};(0,-0.3)*++{\sfive};\endxy}
}
\newcommand{\Gfth}{
{\xy (-3,-2.6)*++{\sthree};(0,-1.8)*++{\sfour};\endxy}
}
\newcommand{\Gfto}{
{\xy (-6,-4.8)*++{\sone};(-3,-3.4)*++{\stwo};(0,-1.8)*++{\sfour};\endxy}
}
\newcommand{\Gse}{
{\xy (0,1.3)*++{\sseven};\endxy}
}
\newcommand{\Gseo}{
{\xy (-3,-4.8)*++{\sone};(0,1.3)*++{\sseven};\endxy}
}
\newcommand{\Gst}{
{\xy (-3,-3.4)*++{\stwo};(0,0.5)*++{\ssix};\endxy}
}
\newcommand{\Gftho}{
{\xy (-6,-5)*++{\sone};(-3,-2.6)*++{\sthree};(0,-1.8)*++{\sfour};\endxy}
}
\newcommand{\Gfith}{
{\xy (-3,-2.6)*++{\sthree};(0,-0.3)*++{\sfive};\endxy}
}
\newcommand{\Gfito}{
{\xy (-6,-4.8)*++{\sone};(-3,-3.4)*++{\stwo};(0,-0.3)*++{\sfive};\endxy}
}
\newcommand{\Gei}{
{\xy (0,2.8)*++{\seight};\endxy}
}
\newcommand{\Geio}{
{\xy (-3,-4.8)*++{\sone};(0,2.8)*++{\seight};\endxy}
}
\newcommand{\Gset}{
{\xy (-3,-3.4)*++{\stwo};(0,1.3)*++{\sseven};\endxy}
}
\newcommand{\Gsth}{
{\xy (-3,-2.6)*++{\sthree};(0,0.5)*++{\ssix};\endxy}
}
\newcommand{\Gsto}{
{\xy (-6,-4.8)*++{\sone};(-3,-3.4)*++{\stwo};(0,0.5)*++{\ssix};\endxy}
}
\newcommand{\Gfif}{
{\xy (-3,-1.8)*++{\sfour};(0,-0.3)*++{\sfive};\endxy}
}
\newcommand{\Gfitho}{
{\xy (-6,-4.9)*++{\sone};(-3,-2.6)*++{\sthree};(0,-0.4)*++{\sfive};\endxy}
}
\newcommand{\Gftht}{
{\xy (-6,-3.4)*++{\stwo};(-3,-2.6)*++{\sthree};(0,-1.8)*++{\sfour};\endxy}
}
\newcommand{\Gn}{
{\xy (0,3.6)*++{\snine};\endxy}
}

\newcommand{\GraphF}
{
\fontsize{16}{16}\selectfont
\scalebox{.63}{\xymatrix@R=-2pc@H=-5pc{
 \emptyset \ar[d]^0 & & & & & &\Gei\ar[r]^2 \ar[dr]^0 &\Gn \cdots \\
\Go \ar[dd]^1 & & &\Gfi \ar[r]^0 &\Gs \ar[r]^0 &\Gse \ar[ur]^1 \ar[r]^0 &\Gseo\ar[r]^1 &\Geio\cdots \\
 &\Gth \ar[r]^2 \ar[ddr]^0 &\Gf \ar[ur]^1 \ar[dr]^0 & & &\Gso\ar[r]^1 &\Gst\ar[dr]^2 \ar[r]^0 &\Gset\cdots \\
\Gt \ar[ur]^2 \ar[dr]^0& & &\Gfo \ar[r]^1 &\Gfio \ar[r]^1 \ar[ur]^0 &\Gfit\ar[dr]^2 \ar[ur]^0 & &\Gsth\cdots \\
 &\Gto\ar[r]^2 &\Gtho\ar[ur]^2 \ar[dr]^1 & &\Gft\ar[r]^2 \ar[dr]^0 &\Gfth\ar[r]^1 \ar[dr]^0 &\Gfith\ar[r]^2 \ar[ur]^0 &\Gfif\cdots \\
 & & &\Gtht\ar[ur]^2 \ar[r]^0 &\Gthto\ar[r]^2 &\Gfto\ar[r]^2\ar[dr]^1 &\Gftho\ar[r]^1 &\Gftht\cdots \\
 & & & & & &\Gfito\ar[dr]^2 \ar[r]^0 &\Gsto\cdots \\
 & & & & & & &\Gfitho\cdots
}}
\fontsize{10}{10}\selectfont
}

\newcommand{\shade}
{\xy
(0,0)*{}="B0";(-6,0)*{}="B1";(-12,0)*{}="B2";(-18,0)*{}="B3";(-24,0)*{}="B4";
(-1.5,-1.5)*{}="BB0";(-7.5,-1.5)*{}="BB1";(-13.5,-1.5)*{}="BB2";(-19.5,-1.5)*{}="BB3";(-25.5,-1.5)*{}="BB4";
"B0"; "BB0" **\dir{-};"B1"; "BB1" **\dir{-};"B2"; "BB2" **\dir{-};"B3"; "BB3" **\dir{-};"B4"; "BB4" **\dir{-};
"BB0"; "BB4"+(-20,0) **\dir{-};
\endxy}

\newcommand{\hhbox}[1]
{\xy
(0,-3)*{}="T1";(6,-3)*{}="T2";
(0,-6)*{}="B1";(6,-6)*{}="B2";
(3,0)*{}="TR1";(9,0)*{}="TR2";
(9,-3)*{}="BR2";
"T1"; "T2" **\dir{-}; "T1"; "B1" **\dir{-};
"T1"; "TR1" **\dir{-}; "T2"; "TR2" **\dir{-};
"B2"; "BR2" **\dir{-}; "BR2"; "TR2" **\dir{-};
"TR1"; "TR2" **\dir{-};
"B1"; "B2" **\dir{-}; "T2"; "B2" **\dir{-};
(3,-4.5)*{#1};
\endxy}

\newcommand{\ubox}[1]
{\xy
(0,0)*{}="T1";(6,0)*{}="T2";
(0,-6)*{}="B1";(6,-6)*{}="B2";
(3,3)*{}="TR1";(9,3)*{}="TR2";
(9,-3)*{}="BR2";
"T1"; "T2" **\dir{-}; "T1"; "B1" **\dir{-};
"T1"; "TR1" **\dir{-}; "T2"; "TR2" **\dir{-};
"B2"; "BR2" **\dir{-}; "BR2"; "TR2" **\dir{-};
"TR1"; "TR2" **\dir{-};
"B1"; "B2" **\dir{-}; "T2"; "B2" **\dir{-};
(3,-3)*{#1};
\endxy}

\newcommand{\hhboxp}[1]
{\xy
(0,-3)*{}="T1";(6,-3)*{}="T2";
(0,-6)*{}="B1";(6,-6)*{}="B2";
(3,0)*{}="TR1";(9,0)*{}="TR2";
(9,-3)*{}="BR2";
"T1"; "T2" **\dir{-}; "T1"; "B1" **\dir{-};
"T1"; "TR1" **\dir{-}; 
"TR1"; "TR2" **\dir{-};
"B1"; "B2" **\dir{-}; "T2"; "B2" **\dir{-};
(3,-4.5)*{#1};
\endxy}

\newcommand{\hhboxpl}[1]
{\xy
(-20,-3)*{}="T0"; (-20,-6)*{}="B0";
(0,-3)*{}="T1";(6,-3)*{}="T2";
(0,-6)*{}="B1";(6,-6)*{}="B2";
(-17,0)*{}="TR0";(3,0)*{}="TR1";(9,0)*{}="TR2";
(9,-3)*{}="BR2";
"T1"; "T0" **\dir{-};"B1"; "B0" **\dir{-};
"T1"; "T2" **\dir{-}; "T1"; "B1" **\dir{-};
"T1"; "TR1" **\dir{-}; 
"TR1"; "TR2" **\dir{-}; "TR1"; "TR0" **\dir{-};
"B1"; "B2" **\dir{-}; "T2"; "B2" **\dir{-};
(3,-4.5)*{#1}; (-9,-4.5)*{\cdots};
\endxy}

\newcommand{\htbox}[1]
{\xy
(0,0)*{}="T1";(6,0)*{}="T2";
(0,-6)*{}="B1";(6,-6)*{}="B2";
(1.5,1.5)*{}="TR1";(7.5,1.5)*{}="TR2";
(7.5,-4.5)*{}="BR2";
"T1"; "T2" **\dir{-}; "T1"; "B1" **\dir{-};
"T1"; "TR1" **\dir{-}; "T2"; "TR2" **\dir{-};
"B2"; "BR2" **\dir{-}; "BR2"; "TR2" **\dir{-};
"TR1"; "TR2" **\dir{-};
"B1"; "B2" **\dir{-}; "T2"; "B2" **\dir{-};
(3,-3)*{#1};
\endxy}

\newcommand{\htboxs}[1]
{\xy
(0,0)*{}="T1";(6,0)*{}="T2";
(0,-6)*{}="B1";(6,-6)*{}="B2";
(1.5,1.5)*{}="TR1";(7.5,1.5)*{}="TR2";
(7.5,-4.5)*{}="BR2";
"T1"; "T2" **\dir{-}; "T1"; "B1" **\dir{-};
"T1"; "TR1" **\dir{-}; "T2"; "TR2" **\dir{-};
"B2"; "BR2" **\dir{-}; "BR2"; "TR2" **\dir{-};
"TR1"; "TR2" **\dir{-};
"B1"; "B2" **\dir{-}; "T2"; "B2" **\dir{-};
(3,-3)*{\scriptstyle #1};
\endxy}

\newcommand{\htboxp}[1]
{\xy
(0,0)*{}="T1";(6,0)*{}="T2";
(0,-6)*{}="B1";(6,-6)*{}="B2";
(1.5,1.5)*{}="TR1";(7.5,1.5)*{}="TR2";
(7.5,-4.5)*{}="BR2";
"T1"; "T2" **\dir{-}; "T1"; "B1" **\dir{-};
"T1"; "TR1" **\dir{-}; "T2"; "TR2" **\dir{-};
"TR1"; "TR2" **\dir{-};
"B1"; "B2" **\dir{-}; "T2"; "B2" **\dir{-};
(3,-3)*{#1};
\endxy}

\newcommand{\htboxps}[1]
{\xy
(0,0)*{}="T1";(6,0)*{}="T2";
(0,-6)*{}="B1";(6,-6)*{}="B2";
(1.5,1.5)*{}="TR1";(7.5,1.5)*{}="TR2";
(7.5,-4.5)*{}="BR2";
"T1"; "T2" **\dir{-}; "T1"; "B1" **\dir{-};
"T1"; "TR1" **\dir{-}; "T2"; "TR2" **\dir{-};
"TR1"; "TR2" **\dir{-};
"B1"; "B2" **\dir{-}; "T2"; "B2" **\dir{-};
(3,-3)*{\scriptstyle #1};
\endxy}

\newcommand{\htboxpp}[1]
{\xy
(-20,0)*{}="T0";(0,0)*{}="T1";(6,0)*{}="T2";
(-20,-6)*{}="B0";(0,-6)*{}="B1";(6,-6)*{}="B2";
(-18.5,1.5)*{}="TR0";(1.5,1.5)*{}="TR1";(7.5,1.5)*{}="TR2";
(7.5,-4.5)*{}="BR2";
"TR1"; "TR0" **\dir{-};
"T1"; "T0" **\dir{-};"B1"; "B0" **\dir{-};
"T1"; "T2" **\dir{-}; "T1"; "B1" **\dir{-};
"T1"; "TR1" **\dir{-}; "T2"; "TR2" **\dir{-};
"TR1"; "TR2" **\dir{-};
"B1"; "B2" **\dir{-}; "T2"; "B2" **\dir{-};
(3,-3)*{#1}; (-9,-3)*{\cdots};
\endxy}

\newcommand{\htboxpps}[1]
{\xy
(-20,0)*{}="T0";(0,0)*{}="T1";(6,0)*{}="T2";
(-20,-6)*{}="B0";(0,-6)*{}="B1";(6,-6)*{}="B2";
(-18.5,1.5)*{}="TR0";(1.5,1.5)*{}="TR1";(7.5,1.5)*{}="TR2";
(7.5,-4.5)*{}="BR2";
"TR1"; "TR0" **\dir{-};
"T1"; "T0" **\dir{-};"B1"; "B0" **\dir{-};
"T1"; "T2" **\dir{-}; "T1"; "B1" **\dir{-};
"T1"; "TR1" **\dir{-}; "T2"; "TR2" **\dir{-};
"TR1"; "TR2" **\dir{-};
"B1"; "B2" **\dir{-}; "T2"; "B2" **\dir{-};
(3,-3)*{\scriptstyle #1}; (-9,-3)*{\cdots};
\endxy}

\newcommand{\Gyzo}
{\xy (0,0)*++{\htbox{1}}; (-6,0)*++{\htboxp{0}}; (-12,0)*++{\htboxp{1}}; (-28,0)*++{\htboxpp{0}};
(-20.3,-4.5)*++{\shade};
 \endxy}

\newcommand{\Gyoz}
{\xy (0,0)*++{\htbox{0}}; (-6,0)*++{\htboxp{1}}; (-12,0)*++{\htboxp{0}}; (-28,0)*++{\htboxpp{1}};
(-20.3,-4.5)*++{\shade}; \endxy}

\newcommand{\Gynnmo}
{\xy (0,0)*++{\htbox{n}}; (-6,0)*++{\htboxps{n-1}}; (-12,0)*++{\htboxp{n}}; (-28,0)*++{\htboxpps{n-1}};
(-20.3,-4.5)*++{\shade};  \endxy}

\newcommand{\Gynmon}
{\xy (0,0)*++{\htboxs{n-1}}; (-6,0)*++{\htboxp{n}}; (-12,0)*++{\htboxps{n-1}}; (-28,0)*++{\htboxpp{n}};
(-20.3,-4.5)*++{\shade};  \endxy}

\newcommand{\Gyzz}
{\xy (0,0)*++{\hhbox{0}}; (-6,0)*++{\hhboxp{0}}; (-12,0)*++{\hhboxp{0}}; (-28,0)*++{\hhboxpl{0}};
  \endxy}

\newcommand{\Gynn}
{\xy (0,0)*++{\hhbox{n}}; (-6,0)*++{\hhboxp{n}}; (-12,0)*++{\hhboxp{n}}; (-28,0)*++{\hhboxpl{n}};
  \endxy}

\newcommand{\gyzz}
{
 \fontsize{8}{8}\selectfont
\begin{matrix}D_{n+1}^{(2)},A_{2n}^{(2)} \\ \Lambda_0 \end{matrix}
 \fontsize{10}{10}\selectfont
}

\newcommand{\gynn}
{
 \fontsize{8}{8}\selectfont
\begin{matrix}D_{n+1}^{(2)},B_{n}^{(1)} \\ \Lambda_n \end{matrix}
 \fontsize{10}{10}\selectfont
}

\newcommand{\gyoz}
{
 \fontsize{8}{8}\selectfont
\begin{matrix}A_{2n-1}^{(2)},B_{n}^{(1)},D_{n}^{(1)}  \\ \Lambda_0 \end{matrix}
 \fontsize{10}{10}\selectfont
}

\newcommand{\gyzo}
{
 \fontsize{8}{8}\selectfont
\begin{matrix}A_{2n-1}^{(2)},B_{n}^{(1)},D_{n}^{(1)}  \\ \Lambda_1 \end{matrix}
 \fontsize{10}{10}\selectfont
}

\newcommand{\gynnmo}
{
 \fontsize{8}{8}\selectfont
\begin{matrix}D_{n}^{(1)}  \\ \Lambda_{n-1} \end{matrix}
 \fontsize{10}{10}\selectfont
}
\newcommand{\gynmon}
{
 \fontsize{8}{8}\selectfont
\begin{matrix}D_{n}^{(1)}  \\ \Lambda_{n} \end{matrix}
 \fontsize{10}{10}\selectfont
}

\newcommand{\youngwall}
{\xy
(0,0)*{}="B1";(6,0)*{}="B2";
(0,24)*{}="T1";(6,24)*{}="T2";
"T1"; "B1" **\dir{-};"T2"; "B2" **\dir{-};"T1"; "T2" **\dir{-};"B2"; "B1" **\dir{-};
"B1"+(-6,0);"B1"+(-6,15) **\dir{-};
"B1"+(-6,0);"B1" **\dir{-};
"B1"+(-6,6);"B1"+(0,6) **\dir{-};
"B1"+(-6,12);"B1"+(0,12) **\dir{-};
"B1"+(-6,15);"B1"+(0,15) **\dir{-};
"B1"+(-6,15);"B1"+(-3,18) **\dir{-};
"B1"+(0,18);"B1"+(-3,18) **\dir{-};
"B1"+(-12,0);"B1"+(-12,6) **\dir{-};
"B1"+(-12,0);"B1"+(-6,0) **\dir{-};
"B1"+(-12,6);"B1"+(-6,6) **\dir{-};
"B1"+(-12,6);"B1"+(-9,9)**\dir{-};
"B1"+(-10.5,7.5);"B1"+(-23.5,7.5)**\dir{-};
"B1"+(-6,8);"B1"+(-23,8)**\dir{.};
"B1"+(-6,8.5);"B1"+(-22.5,8.5)**\dir{.};
"B1"+(-9,9);"B1"+(-22.5,9)**\dir{-};
"B1"+(-16.5,7.5);"B1"+(-16.5,1.5) **\dir{-};
"B1"+(-16.5,7.5);"B1"+(-15,9) **\dir{-};
"B1"+(-22.5,7.5);"B1"+(-22.5,1.5) **\dir{-}; 
"B1"+(-22.5,7.5);"B1"+(-21,9) **\dir{-};
"B1"+(-12,1.5);"B1"+(-23.5,1.5) **\dir{-};
"B1"+(-6,9); "B1"+(-9,9)**\dir{-};
"B1"+(-6,7.5);"B1"+(-10.5,7.5)**\dir{-};
"B1"+(0,6);"B2"+(0,6) **\dir{-};
"B1"+(0,12);"B2"+(0,12) **\dir{-};
"B1"+(0,15);"B2"+(0,15) **\dir{-};
"B1"+(0,18);"B2"+(0,18) **\dir{-};
"B2";"B2"+(3,3) **\dir{-};
"B2"+(1.5,1.5);"B2"+(1.5,7.5) **\dir{-};
"B2"+(2,2);"B2"+(2,8) **\dir{.};
"B2"+(2.5,2.5);"B2"+(2.5,8.5) **\dir{.};
"B2"+(0,6);"B2"+(3,9) **\dir{-};
"B2"+(0,12);"B2"+(3,15) **\dir{-};
"B2"+(0,15);"B2"+(3,18) **\dir{-};
"B2"+(0,18);"B2"+(3,21) **\dir{-};
"B2"+(0,24);"B2"+(3,27) **\dir{-};
"B2"+(3,3);"B2"+(3,33) **\dir{-};
"B2"+(1.5,31.5);"B2"+(3,33) **\dir{-};
"B2"+(1.5,31.5);"B2"+(1.5,25.5) **\dir{-};
"B2"+(-4.5,25.5);"B2"+(1.5,25.5) **\dir{-};
"B2"+(-4.5,25.5);"B2"+(-4.5,31.5) **\dir{-};
"B2"+(-4.5,31.5);"B2"+(1.5,31.5) **\dir{-};
"B2"+(-4.5,25.5);"B2"+(-6,24) **\dir{-};
"B2"+(-4.5,31.5);"B2"+(-3,33) **\dir{-};
"B2"+(3,33);"B2"+(-3,33) **\dir{-};
(3,3)*{0}; (3,9)*{2};(3,13.5)*{3};(3,16.5)*{3};(3,21)*{2};
(4.2,28.5)*{1};
(-3,3)*{1}; (-3,9)*{2};(-3,13.5)*{3};
(-9,3)*{0};
(-13.5,4.5)*{0};(-19.5,4.5)*{1};
(15,16.5)*{=};
\endxy}

\newcommand{\deltacol}
{\xy
(0,0)*{}="B1";(6,0)*{}="B2";
(0,24)*{}="T1";(6,24)*{}="T2";
"T1"; "B1" **\dir{-};"T2"; "B2" **\dir{-};"T1"; "T2" **\dir{-};"B2"; "B1" **\dir{-};
"B1"+(0,6);"B2"+(0,6) **\dir{-};
"B1"+(0,12);"B2"+(0,12) **\dir{-};
"B1"+(0,15);"B2"+(0,15) **\dir{-};
"B1"+(0,18);"B2"+(0,18) **\dir{-};
"B2";"B2"+(1.5,1.5) **\dir{-};
"B2"+(1.5,1.5);"B2"+(1.5,7.5) **\dir{-};
"B2"+(0,6);"B2"+(3,9) **\dir{-};
"B2"+(0,12);"B2"+(3,15) **\dir{-};
"B2"+(0,15);"B2"+(3,18) **\dir{-};
"B2"+(0,18);"B2"+(3,21) **\dir{-};
"B2"+(0,24);"B2"+(3,27) **\dir{-};
"B2"+(3,9);"B2"+(3,33) **\dir{-};
"B2"+(1.5,31.5);"B2"+(3,33) **\dir{-};
"B2"+(1.5,31.5);"B2"+(1.5,25.5) **\dir{-};
"B2"+(-4.5,25.5);"B2"+(1.5,25.5) **\dir{-};
"B2"+(-4.5,25.5);"B2"+(-4.5,31.5) **\dir{-};
"B2"+(-4.5,31.5);"B2"+(1.5,31.5) **\dir{-};
"B2"+(-4.5,25.5);"B2"+(-6,24) **\dir{-};
"B2"+(-4.5,31.5);"B2"+(-3,33) **\dir{-};
"B2"+(3,33);"B2"+(-3,33) **\dir{-};
(3,3)*{0}; (3,9)*{2};(3,13.5)*{3};(3,16.5)*{3};(3,21)*{2};
(4.2,28.5)*{1};
(15,16)*{ = };
\endxy}

\newcommand{\Deltacol}{
\begin{texdraw}
\drawdim em
\setunitscale 1.7
\move(0 0)
\bsegment
\move(0 4)\lvec(1 5)\lvec(1 4)
\move(0 3)\lvec(1 3)
\move(0 2.5)\lvec(1 2.5)
\move(0 2)\lvec(1 2)
\move(0 1)\lvec(1 1)
\move(0 0)\lvec(0 4)\lvec(1 4)\lvec(1 1)
\move(0 0)\lvec(1 1)
\htext(0.15 0.52){$0$}
\htext(0.65 4.12){$1$}
\htext(0.4 1.35){$2$}
\htext(0.4 2.1){$3$}\htext(0.4 2.6){$3$}
\htext(0.4 3.35){$2$}
\esegment
\end{texdraw}
}

\newcommand{\Done}
{ {\xy (0,14)*++{\deltacol} \endxy} \quad \Deltacol }

\newcommand{\deltacols}
{\xy
(0,0)*{}="B1";(6,0)*{}="B2";
(0,24)*{}="T1";(6,24)*{}="T2";
"T1"; "B1"+(0,6) **\dir{-};
"T2"; "B2"+(0,6) **\dir{-};
"T1"; "T2" **\dir{-};
"B2"+(3,3);"B2"+(3,9) **\dir{-};
"B2"+(1.5,1.5); "B1"+(1.5,1.5) **\dir{-};
"B1"+(1.5,1.5); "B1"+(1.5,6) **\dir{-};
"B1"+(0,6);"B2"+(0,6) **\dir{-};
"B1"+(0,12);"B2"+(0,12) **\dir{-};
"B1"+(0,15);"B2"+(0,15) **\dir{-};
"B1"+(0,18);"B2"+(0,18) **\dir{-};
"B2"+(3,3);"B2"+(1.5,1.5) **\dir{-};
"B2"+(1.5,1.5);"B2"+(1.5,7.5) **\dir{-};
"B2"+(0,6);"B2"+(3,9) **\dir{-};
"B2"+(0,12);"B2"+(3,15) **\dir{-};
"B2"+(0,15);"B2"+(3,18) **\dir{-};
"B2"+(0,18);"B2"+(3,21) **\dir{-};
"B2"+(3,9);"B2"+(3,27) **\dir{-};
"B2"+(0,24);"B2"+(3,27) **\dir{-};
"B1"+(0,24);"B1"+(0,30) **\dir{-};
"B1"+(1.5,31.5);"B1"+(0,30) **\dir{-};
"B2"+(1.5,31.5);"B2"+(0,30) **\dir{-};
"B2"+(0,30);"B1"+(0,30) **\dir{-};
"B2"+(0,24);"B2"+(0,30) **\dir{-};
"B2"+(1.5,25.5);"B2"+(1.5,31.5) **\dir{-};
"B2"+(3,27);"B2"+(1.5,27) **\dir{-};
"B1"+(1.5,31.5);"B2"+(1.5,31.5) **\dir{-};
(4.2,4.5)*{0}; (3,9)*{2};(3,13.5)*{3};(3,16.5)*{3};(3,21)*{2};(3,27)*{1};
(15,16)*{ = };
\endxy}

\newcommand{\Deltacols}{
\begin{texdraw}
\drawdim em
\setunitscale 1.7
\move(0 0)
\bsegment
\move(0 4)\lvec(0 5)\lvec(1 5)\lvec(0 4)
\move(0 3)\lvec(1 3)
\move(0 2.5)\lvec(1 2.5)
\move(0 2)\lvec(1 2)
\move(0 1)\lvec(1 1)
\move(0 1)\lvec(0 4)\lvec(1 4)\lvec(1 1)
\move(0 0)\lvec(1 1)\lvec(1 0)\lvec(0 0)
\htext(0.15 4.52){$1$}
\htext(0.65 0.12){$0$}
\htext(0.4 1.35){$2$}
\htext(0.4 2.1){$3$}\htext(0.4 2.6){$3$}
\htext(0.4 3.35){$2$}
\esegment
\end{texdraw}
}

\newcommand{\Dtwo}
{ {\xy (0,14)*++{\deltacols} \endxy} \quad \Deltacols}

\newcommand{\deltacolp}
{\xy
(0,0)*{}="B1";(6,0)*{}="B2";
"B1";"B2" **\dir{-};
"B1";"B1"+(0,24) **\dir{-};
"B2";"B2"+(0,24) **\dir{-};
"B1"+(0,3);"B2"+(0,3) **\dir{-};
"B1"+(0,9);"B2"+(0,9) **\dir{-};
"B1"+(0,15);"B2"+(0,15) **\dir{-};
"B1"+(0,21);"B2"+(0,21) **\dir{-};
"B1"+(0,24);"B2"+(0,24) **\dir{-};
"B2";"B2"+(3,3) **\dir{-};
"B2"+(0,3);"B2"+(3,6) **\dir{-};
"B2"+(1.5,10.5);"B2"+(1.5,16.5) **\dir{-};
"B2"+(0,9);"B2"+(3,12) **\dir{-};
"B2"+(0,15);"B2"+(3,18) **\dir{-};
"B2"+(0,21);"B2"+(3,24) **\dir{-};
"B2"+(0,24);"B2"+(3,27) **\dir{-};
"B2"+(3,3);"B2"+(3,27) **\dir{-};
"B1"+(3,27);"B2"+(3,27) **\dir{-};
"B1"+(3,27);"B1"+(0,24) **\dir{-};
(3,1.5)*{3};
(3,6)*{2};
(3,12)*{0};
(3,18)*{2};
(3,22.5)*{3};
(15,16)*{ = };
\endxy}

\newcommand{\Deltacolp}{
\begin{texdraw}
\drawdim em
\setunitscale 1.7
\move(0 0)
\bsegment
\move(0 4)\lvec(0 4.5)\lvec(1 4.5)\lvec(1 4)
\move(0 3)\lvec(1 3)
\move(0 2)\lvec(1 3)
\move(0 2)\lvec(1 2)
\move(0 1)\lvec(1 1)
\move(0 1)\lvec(0 4)\lvec(1 4)\lvec(1 1)
\move(0 0.5)\lvec(0 0.5)
\move(0 1)\lvec(1 1)
\move(0 0.5)\lvec(1 0.5)
\move(0 0.5)\lvec(0 1)
\move(1 0.5)\lvec(1 1)
\htext(0.15 2.52){$0$}
\htext(0.65 2.12){$1$}
\htext(0.4 1.35){$2$}
\htext(0.4 0.6){$3$}\htext(0.4 4.1){$3$}
\htext(0.4 3.35){$2$}
\esegment
\end{texdraw}
}

\newcommand{\Dthree}
{ {\xy (0,14)*++{\deltacolp} \endxy} \quad \Deltacolp}

\newcommand{\DthreetwoEx}{
\begin{texdraw}
\drawdim em
\setunitscale 1.7
\move(1 1)\dtrj
\move(2 1)\dtrj
\move(3 1)\dtrj
\move(4 0)\lvec(4 2.5)\lvec(3 2.5)\lvec(3 0)
\move(3 0)\lvec(3 2)\lvec(2 2)\lvec(2 0)
\move(2 0)\lvec(2 1)\lvec(1 1)
\move(1 0)\lvec(1 1)
\move(1 0)\lvec(4 0)
\move(4 0.5)\lvec(1 0.5)
\move(4 1)\lvec(1 1)
\move(4 2)\lvec(2 2)
\htext(3.4 2.1){$2$}
\htext(2.4 1.35){$1$}\htext(3.4 1.35){$1$}
\htext(1.4 0.1){$0$}\htext(2.4 0.1){$0$}\htext(3.4 0.1){$0$}
\htext(1.4 0.6){$0$}\htext(2.4 0.6){$0$}\htext(3.4 0.6){$0$}
\end{texdraw}
}

\begin{document}

\title[The Andrews-Olsson identity and Bessenrodt insertion algorithm on Young walls]
{The Andrews-Olsson identity and Bessenrodt insertion algorithm on Young walls}

\author[Se-jin Oh]{Se-jin Oh$^{1}$}

\address{Department of Mathematical Sciences, Seoul National University Gwanak-ro 1, Gwanak-gu, Seoul 151-747, Korea}
         \email{sejin092@gmail.com}

\thanks{$^{1}$This work was supported by BK21 PLUS SNU Mathematical Sciences Division}


\begin{abstract}
We extend the Andrews-Olsson identity to two-colored
partitions. Regarding the sets of proper Young walls of quantum affine algebras
$\g_n=A^{(2)}_{2n}$, $A^{(2)}_{2n-1}$, $B^{(1)}_{n}$, $D^{(1)}_{n}$ and
$D^{(2)}_{n+1}$ as the sets of two-colored partitions, the extended
Andrews-Olsson identity implies that the generating functions of the
sets of reduced Young walls have very simple formulae:
\begin{center}
$\prod^{\infty}_{i=1}(1+t^i)^{\kappa_i}$
 where $\kappa_i=0$, $1$ or $2$, and $\kappa_i$ varies periodically.
\end{center}
Moreover, we generalize the Bessenrodt's algorithms to prove the
extended Andrews-Olsson identity in an alternative way. From these
algorithms, we can give crystal structures on certain subsets of
pair of strict partitions which are isomorphic to the crystal bases
$B(\Lambda)$ of the level $1$ highest weight modules $V(\Lambda)$
over $U_q(\g_n)$.
\end{abstract}

\maketitle

\section*{Introduction}
A weakly decreasing sequence of nonnegative integers
$\lambda=(\lambda_1,\lambda_2,\ldots)$ is called a {\it partition}
of $m$, denoted by $\lambda \vdash m$, if $m = \sum_i \lambda_i$. A
partition $\lambda$ is called a {\it strict partition} if all
parts are strictly decreasing and an {\it odd partition} if
all parts are odd. Let $\OP[m]$ (respectively, $\OS[m]$
and $\mathscr{O}[m]$) be the set of all (respectively, strict and odd)
partitions of $m$. Denote by $\OP$ (respectively, $\OS$ and
${\mathscr O}$)
 the set of all (respectively, strict and odd) partitions.

\medskip

\noindent \textbf{Theorem} (Euler's partition theorem) \  For all $m
\in \Z_{\ge 0}$,
\begin{equation} \label{eq: Euler id}
|\OS[m]| =|{\mathscr O}[m]|.
\end{equation}

\medskip

Euler proved the identity \eqref{eq: Euler id} by showing the equality of the
corresponding generating functions:
$$ \sum_{m=0}^\infty |\mathscr{P}[m]|t^m=\prod_{i=1}^{\infty}(1+t^i)=  \prod_{i=1}^{\infty} \dfrac{1}{(1-t^{2i-1})}=\sum_{m=0}^\infty |\mathscr{O}[m]|t^m.$$
In 1882, Sylvester showed the identity in an alternative way. He
proved the identity by constructing a combinatorial algorithm which
establishes a bijection between $\mathscr{O}[m]$ and $\OS[m]$.

The generalization of Euler's partition theorem is still one of the main topics in combinatorics
\cite{B94,KY99,La04,PP98,SY07}.
Thus, for subsets $\mathcal{A}$ and $\mathcal{B}$ of partitions,
$$ \text{ showing $|\mathcal{A}[m]|=|\mathcal{B}[m]|$
and constructing a combinatorial bijection between $\mathcal{A}[m]$
and $\mathcal{B}[m]$}$$ are interesting problems. In 1991
\cite{AO91}, Andrews and Olsson defined the series of subsets of
partitions, $\mathcal{AO}_1^{X_N}$ and $\mathcal{AO}_2^{X_N}$ ($N
\in \N$) (see Definition \ref{def: AOXN1} and \ref{def: AOXN2}), and showed that
$$|\mathcal{AO}_1^{X_N}[m]| = |\mathcal{AO}_2^{X_N}[m]| \text{ for all } m \in \Z_{\ge 0}.$$
Shortly after, Bessenrodt constructed a combinatorial insertion
algorithm which gives a bijection between them.
Several generalizations of Andrews-Olsson identity have been developed in \cite{B91,B95,WHKM10}.

On the other hand, the notion of partitions has been generalized such as overpartitions, multi-colored partitions.
Using the generalized notions, many mathematicians interpreted or proved combinatorial identities arising from
hypergeometric series \cite{CL,CLY,T}.

The characters of integrable modules over quantum groups $U_q(\mathsf{g})$
are important algebraic invariants which {\it determine} the
isomorphism classes in the sense that $M \cong N$ if and only if
${\rm ch}M = {\rm ch}N$. In \cite{Kash90, Kash91}, Kashiwara
developed the {\it crystal basis theory} for integrable
$U_q(\mathsf{g})$-modules from which many combinatorial properties of an
integrable module can be deduced.
By using realization
of crystal bases, one can compute the characters of integrable
modules (see Section \ref{Sec: Quantum affine}).

In \cite{Ha90}, Hayashi gave a $U_q(A^{(1)}_n)$-module structure on the space of {\it Young diagrams},
which can be understood as the set of all partitions with coloring.
In \cite{MM90}, Misra and Miwa showed that
the {\it reduced} Young diagrams provide a realization of the crystal basis of the level $1$ highest weight modules (see Section \ref{Sec: Quantum affine} for
definitions).
In \cite{K03}, Kang introduced the notion of \emph{Young walls} (which can be understood as a generalization of
Young diagrams) as new combinatorial scheme for realizing the crystal bases of the level $1$ highest weight
modules over all classical quantum affine algebras $\g$.
In that paper, it was shown that the set $\mathtt{Z}(\Lambda)$
of \emph{proper} Young walls has a crystal structure. Moreover, he proved that
the crystal $B(\Lambda)$ of the level 1 highest weight module $V(\Lambda)$ can be realized by the set
of \emph{reduced} Young walls $\mathtt{Y}(\Lambda)$. In \cite{KK04,KK08}, Kang and Kwon gave the $U_q(\g)$-module structure
on the space of proper Young walls and proved the decomposition formulae for the space into level 1 highest weight
modules (see \cite{HK02, KK04} for more details).

In this paper, we extend the Andrews-Olsson identity and
Bessenrodt's insertion algorithm to {\it two-colored partitions} by
regarding Young walls as two-colored partitions.
For each $\g$ and a level $1$ highest weight $\Lambda$, we define two subsets of two-colored partitions,
denoted by $\mathcal{AO}_1(\Lambda)$ and $\mathcal{AO}_2(\Lambda)$, such that
\begin{itemize}
\item $\mathcal{AO}_1(\Lambda)$ can be identified with the sets of reduced Young walls,
\item $\mathcal{AO}_2(\Lambda)$ can be identified with a certain pair of subsets of
strict partitions (see Section \ref{subsec: var par}).
\end{itemize}

We first show that the numbers of two-colored partitions of $m \in \Z_{\ge
0}$ in $\mathcal{AO}_1(\Lambda)$ and in $\mathcal{AO}_2(\Lambda)$
coincide with each other by developing a new combinatorial algorithm
(see Section \ref{Sec: Anderws-Olsson identity}). As a corollary, we
can compute the generating functions of sets
$\mathcal{AO}_i(\Lambda)$ ($i=1,2$); i.e., the formal power series
$g(t)$ in one indeterminate $t$ such that
\begin{equation} \label{eq: Yw generating}
g(t)= \sum_{m=0}^\infty |\mathcal{AO}_i(\Lambda)[m]|t^m
\end{equation}
where $\mathcal{AO}_i(\Lambda)[m]=\{ \lambda \in \mathcal{AO}_i(\Lambda) \ | \ \lambda \vdash m \}$.
The generating functions of the sets of
reduced Young walls \eqref{eq: Yw generating} are given as follows:
\begin{center}
\begin{tabular}{ | c | c | c  c | } \hline
Type & $\Lambda$ &  \multicolumn{2}{c|}{generating function}  \\ \hline
$A^{(2)}_{2n}$ & $\Lambda_0$
& $ \ \ \ \prod_{i=1}^{\infty} (1+t^i)^{\kappa_i}$,
& \fontsize{8}{8}\selectfont
$\kappa_i=0$ if $i \equiv 0 \ {\rm mod} \ 2n+1$, and $\kappa_i=1$ otherwise.
\fontsize{10}{10}\selectfont
\\ \hline
$A^{(2)}_{2n-1}$ & $\Lambda_0$, $\Lambda_1$
& \multicolumn{2}{c|}{ $\prod_{i=1}^{\infty} (1+t^i)$. }  \\ \hline
$B^{(1)}_{n}$ & $\Lambda_0$, $\Lambda_1$
& $ \ \ \  \prod^{\infty}_{i=1} (1+t^i)^{\kappa_i}$,
& \fontsize{8}{8}\selectfont
 $\kappa_i=2$ if $i \equiv 0 \ {\rm mod} \ 2n$, and $\kappa_i=1$ otherwise. $\quad$
 \fontsize{10}{10}\selectfont
 \\ \cline {2-4}
 & $\Lambda_n$
& $ \ \ \  \prod^{\infty}_{i=1} (1+t^i)^{\kappa_i}$,
& \fontsize{8}{8}\selectfont
$\kappa_i=2$ if $i \equiv n \ {\rm mod} \ 2n$, and $\kappa_i=1$ otherwise. $\quad$
\fontsize{10}{10}\selectfont
\\ \hline
$D^{(1)}_{n}$ & $\Lambda_0$, $\Lambda_1$, $\Lambda_{n-1}$, $\Lambda_{n}$
& $ \ \ \ \prod^{\infty}_{i=1} (1+t^i)^{\kappa_i}$,
& \fontsize{8}{8}\selectfont
 $\kappa_i=2$ if $i \equiv 0 \ {\rm mod} \ n-1$, and $\kappa_i=1$ otherwise. $ $
 \fontsize{10}{10}\selectfont
 \\ \hline
$D^{(2)}_{n+1}$& $\Lambda_{0}$, $\Lambda_{n}$
& \multicolumn{2}{c|}{ $\prod_{i=1}^{\infty} (1+t^i)$. }  \\ \hline
\end{tabular}
\end{center}
Note that these formulae can be interpreted as the {\it principally
specialized characters} which were studied in \cite{KKLW,LM}. Using a
vertex operator technique, Nakajima and Yamada also proved the
identities for types of $D^{(2)}_{n+1}$ and $A^{(2)}_{2n}$
\cite{NY94}. The bijection between $\mathcal{AO}_1(\Lambda)$ and
$\mathcal{AO}_2(\Lambda)$ is extended to a more general form. More precisely, for any set $X_{ \z_3}$ satisfying certain conditions,
we can define the subset $\mathcal{AO}^{X_{ \z_3}}_i(\Lambda)$ of $\mathcal{AO}_i(\Lambda)$ $(i=1,2)$ and prove that
$$ |\mathcal{AO}^{X_{ \z_3}}_1(\Lambda)[m]|=|\mathcal{AO}^{X_{ \z_3}}_2(\Lambda)[m]| \qquad \text{ for any } m \in \Z_{\ge 0}$$
(see Corollary \ref{Cor: more general AO id}).

Moreover, we construct an explicit bijection between
$\mathcal{AO}_1(\Lambda)$ and $\mathcal{AO}_2(\Lambda)$ by
generalizing Bessenrodt's insertion algorithm (see Section \ref{Sec:
Bessenrodt's refinement}). The restriction of this bijection to $\mathcal{AO}^{X_{ \z_3}}_i(\Lambda)$ yields an explicit
bijection between $\mathcal{AO}^{X_{ \z_3}}_1(\Lambda)$ and $\mathcal{AO}^{X_{ \z_3}}_2(\Lambda)$.
From the bijection, we can {\it assign} a crystal structure on $\mathcal{AO}_2(\Lambda)$ which also realizes the
crystal $B(\Lambda)$.

\vskip 1em
\noindent
{\bf Acknowledgements.} The author would like to thank Prof. Seok-Jin Kang and Prof. Jae-Hoon Kwon
for many valuable discussions and suggestions, and Ph.D. Hye Yeon Lee for helping with computer programming.
The author would also like to thank the anonymous reviewers for their valuable comments and suggestions.

\section{The quantum affine algebras} \label{Sec: Quantum affine}

Let $I=\{ 0,1,...,n \}$ be the index set. The \emph{affine Cartan
datum} $(A,P^{\vee},P,\Pi^{\vee},\Pi)$ consists of
\begin{enumerate}
\item[({\rm a})] a matrix $A$ of corank $1$, called the {\it affine Cartan matrix} satisfying
$$ ({\rm i}) \ a_{ii}=2 \ (i \in I), \quad ({\rm ii}) \ a_{ij} \in \Z_{\le 0}, \quad  ({\rm iii}) \ a_{ij}=0 \text{ if } a_{ji}=0$$
with $D= {\rm diag}(\mathsf{d}_i \in \Z_{>0}  \mid i \in I)$ making $DA$ symmetric,
\item[({\rm b})] a free abelian group $P^{\vee}=\bigoplus_{i=0}^{n} \Z h_i \oplus \Z d $, the \emph{dual weight lattice},
\item[({\rm c})] a free abelian group $P=\bigoplus_{i=0}^{n} \Z \Lambda_i \oplus \Z \delta \subset \h^*=\Q \otimes_\Z P^{\vee} $, the \emph{weight lattice},
\item[({\rm d})] an independent set $\Pi^{\vee} = \{ h_i \mid i\in I\} \subset P^{\vee}$, the set of {\it simple coroots},
\item[({\rm e})] an independent set $\Pi = \{ \alpha_i \mid i\in I \} \subset P$, the set of {\it simple roots},
\end{enumerate}
which satisfy
\begin{eqnarray}&&\parbox{100ex}{
\begin{itemize}
\item $\langle h_i, \alpha_j \rangle  = a_{ij}$ for all $i,j\in I$,
\item for each $i \in I$, there exists $\Lambda_i \in P$ such that $\langle h_j, \Lambda_i \rangle =\delta_{ij}$ for all $j \in I$.
\end{itemize} }\label{eq: fundament weight} \end{eqnarray}

We denote by $P^{+} \seteq \{\Lambda \in P \mid  \langle h_i,\Lambda \rangle \in \Z_{\ge 0},\  i \in I \}$ the set of \emph{dominant integral weights}. The free abelian group $Q\seteq\sum_{i \in I} \Z\alpha_i$ is called the \emph{root lattice} and we denote by $Q^{+}\seteq\bigoplus_{i \in I}\Z_{\ge 0}\alpha_i$.
For $\alpha=\sum_{i \in I} k_i \alpha_i \in Q^+$, we define the \emph{height} of $\alpha$ to be $\het(\alpha)\seteq\sum_{i \in I} k_i$.

Let $q$ be an indeterminate. For $i\in I$ and $m,n \in \Z_{\ge 0}$, define
$$ q_i=q^{s_i}, \ \
[n]_{q_i} =\frac{ {q_i}^n - {q_i}^{-n} }{ {q_i} - {q_i}^{-1} }, \ \
[n]_{q_i}! = \prod^{n}_{k=1} [k]_{q_i} , \ \
\left[\begin{matrix}m \\ n\\ \end{matrix} \right]_{q_i}=  \frac{ [m]_{q_i}! }{[m-n]_{q_i}! [n]_{q_i}! }.
$$

\begin{definition} The {\em quantum affine algebra} $U_q(\g_n)$ with an affine Cartan datum $(A,P^{\vee},P,\Pi^{\vee},\Pi)$
is the associative algebra over $\Q(q)$ with ${\bf 1}$ generated by $e_i,f_i$ $(i \in I)$ and
$q^{h}$ $(h \in P^{\vee})$ satisfying the following relations:

\begin{enumerate}

\item  $q^0=1, q^{h} q^{h'}=q^{h+h'} $ for $ h,h' \in P^{\vee},$

\item  $q^{h}e_i q^{-h}= q^{ \langle h,\alpha_i \rangle} e_i,
          \ q^{h}f_i q^{-h} = q^{- \langle h,\alpha_i \rangle }f_i$ for $h \in P^{\vee}, i \in I$,

\item  $e_if_j - f_je_i = \delta_{ij} \dfrac{K_i -K^{-1}_i}{q_i- q^{-1}_i }, \ \ \mbox{ where } K_i=q_i^{ h_i},$

\item  $\displaystyle \sum^{1-a_{ij}}_{k=0}(-1)^k \left[\begin{matrix}1-a_{ij} \\ k\\ \end{matrix} \right]_{q_i}
e^{1-a_{ij}-k}_i e_j e^{k}_i = \displaystyle \sum^{1-a_{ij}}_{k=0} (-1)^k \left[\begin{matrix}1-a_{ij} \\ k\\
\end{matrix} \right]_{q_i} f^{1-a_{ij}-k}_if_j f^{k}_i=0 \quad \text{ if }  i \ne j. $

\end{enumerate}

\end{definition}

A $U_q(\g_n)$-module $V$ is called a \emph{weight module} if it admits a \emph{weight space decomposition}
$$V= \bigoplus_{\mu \in P} V_{\mu}$$ where $V_{\mu}=\{ v \in V | \ q^hv=q^{\langle h, \mu \rangle } v \text{ for all }
 h \in P^{\vee} \}.$ If $\dim_{\Q(q)} V_{\mu} < \infty \text{ for all } \mu \in P$,
we define the \emph{character} of $V$ by
$$\chi_{\g_n}(V)=\sum_{\mu \in P} (\dim_{\Q(q)} V_{\mu} )e(\mu).$$
Here $\chi_{\g_n}(V)$ is a formal sum and $e(\mu)$ is a basis element of the group algebra $\Z[P]$ with the multiplication given by $e(\mu)e(\nu)=e(\mu+\nu) \text{ for all } \mu,\nu \in P$.

A weight module $V$ over $U_q(\g_n)$ is {\it integrable} if all $e_i$ and $f_i$ ($i \in I$) are locally nilpotent on $V$.

\begin{definition}
The category $\mathcal{O}_{{\rm int}}$ consists of integrable $U_q(\g_n)$-modules $V$ satisfying the following conditions:
\begin{enumerate}
\item $V$ admits a weight space decomposition $V=\bigoplus_{\mu \in P} V_{\mu}$ and each weight space of $V$ is finite dimensional.
\item There exists a finite number of elements $\lambda_1,\ldots,\lambda_s \in P$ such that
$$ {\rm wt}(V) \subset D(\lambda_1) \cup \cdots \cup D(\lambda_s).$$
Here ${\rm wt}(V): =\{ \mu \in P \ | \ V_\mu \ne 0 \}$ and $D(\lambda):= \{ \lambda - \sum_{i \in I} k_i\alpha_i \ | \ k_i \in \Z_{\ge 0} \}$.
\end{enumerate}
\end{definition}

Then it is proved in \cite[Chapter 3]{HK02}, \cite{Lus93} that the
category $\mathcal{O}_{{\rm int}}$ is semisimple with its irreducible
objects being isomorphic to {\it the highest modules $V(\Lambda)$} for some highest weight $\Lambda \in P^+$. Here $V(\Lambda)$ is defined as follows:
\begin{itemize}
\item it is generated by a unique highest weight vector $v_{\Lambda}$ of highest weight $\Lambda$,
\item $e_i$ and $f_i^{\langle  h_i ,\Lambda \rangle+1}$ act trivially on $v_{\Lambda}$ for all $i \in I$,
\item it admits a weight space decomposition,
$V(\Lambda)=\bigoplus_{\mu \in P}V(\Lambda)_\mu,$
where ${\rm wt}(V(\Lambda)) \subset D(\Lambda)$.
\end{itemize}

Thus the character of $V(\Lambda)$ can be written in the following
form:
\begin{equation} \label{eq: ch VLambda}
\chi_{\g_n}(V(\Lambda))=\sum_{\mu \in {\rm wt}(V(\Lambda)) } (\dim_{\Q(q)} V(\Lambda)_{\mu} )e(\Lambda) e(-\alpha_0)^{\mu_0} \cdots e(-\alpha_n)^{\mu_n},
\end{equation}
where $\Lambda-\mu =\sum_{i \in I}\mu_i \alpha_i$ for some $\mu_i \in \Z_{\ge 0}$. For $V(\Lambda)$ $(\Lambda \in P^+)$, we set
$\chi_{\g_n}^{\Lambda} = \chi_{\g_n}(V(\Lambda))$.

\begin{definition} We define the \emph{principally specialized character} of $V(\Lambda)$ as follows:
$$\chi_{\g_n}^{\Lambda}(t) = \chi_{\g_n}^{\Lambda}|_{\substack{e(\Lambda)=1 \\ e(-\alpha_i)=t} }, \text{ for all } i \in I.$$
Here $t$ is an indeterminate. Considering \eqref{eq: ch VLambda}, $\chi_{\g_n}^{\Lambda}(t)$ is written as
$$\chi^\Lambda_{\g_n}=\sum_{\mu \in {\rm wt}(V(\Lambda)) } (\dim_{\Q(q)} V_{\mu} )t^{\sum_{i \in I}\mu_i}
=\sum_{m \in \Z_{\ge 0}} \left( \sum_{\substack{\mu \in {\rm wt}(V(\Lambda)),\\ \het(\Lambda-\mu)=m} } (\dim_{\Q(q)} V_{\mu} ) \right)t^m.$$
\end{definition}

The \emph{level} of $\Lambda \in P^{+}$ is defined to be the nonnegative integer
$\langle \mathsf{c}, \Lambda \rangle$, where $\mathsf{c}$ is the \emph{center} of $\g_n$ defined as follows (\cite[Chapter 4]{Kac}):
$$ \langle \mathsf{c},\alpha_i \rangle=0 \qquad \text{ for all } i \in I.$$
Then the level $1$ dominant integral weights of $\g_n=A^{(2)}_{2n}$, $A^{(2)}_{2n-1}$, $B^{(1)}_{n}$, $D^{(1)}_{n}$ and
$D^{(2)}_{n+1}$ are given as follows (see \eqref{eq: fundament weight} for the definition of $\Lambda_i$):
\begin{equation} \label{eqa:level 1}
\begin{tabular}{ | c | c | c | c | c | c | } \hline
$\g_n$ & $A^{(2)}_{2n}$ & $A^{(2)}_{2n-1}$ & $B^{(1)}_{n}$ & $D^{(1)}_{n}$ & $D^{(2)}_{n+1}$ \\ \hline
$\Lambda$ & $\Lambda_0$  & $\Lambda_0,\Lambda_1$ & $\Lambda_0,\Lambda_1,\Lambda_n$  & $\Lambda_0,\Lambda_1,\Lambda_{n-1},\Lambda_n$ & $\Lambda_0,\Lambda_n$ \\ \hline
\end{tabular}
\end{equation}

Let $\delta=d_0 \alpha_0 + \cdots + d_n \alpha_n$ be the \emph{null root} of $\g_n$ (\cite[Chapter 4]{Kac}); i.e.,
$$ \langle h_i,\delta \rangle=0 \qquad \text{ for all } i \in I.$$
Set $a_i=d_i$ if $\g_n \neq D^{(2)}_{n+1}$, $a_i=2d_i$ if $\g_n = D^{(2)}_{n+1}$.
Set $\Delta = \sum_{i \in I} a_i$.
Then $\Delta$ becomes an integer depending on $\g_n$ as follows:
\begin{equation} \label{eqa:Delta}
\begin{tabular}{ | c | c | c | c | c | c | } \hline
Type & $A^{(2)}_{2n}$ & $A^{(2)}_{2n-1}$ & $B^{(1)}_{n}$ & $D^{(1)}_{n}$ & $D^{(2)}_{n+1}$ \\ \hline
$\Delta$ & $2n+1$  & $2n-1$ & $2n$  & $2n-2$ & $2n+2$ \\ \hline
\end{tabular}
\end{equation}
Let $\A_0=\{f/g \in \Q(q) \mid f,g \in \Q[q],g(0) \neq 0 \}$ and $M$ be a weight $U_q(\g_n)$-module.

\begin{definition} A {\it crystal basis} of $M$ consists of a pair $(L,B)$ with {\it the Kashiwara operators}
$\tilde{e}_i$ and $\tilde{f}_i$ ($i \in I$) as follows:
\begin{enumerate}
\item $L= \bigoplus_\mu L_\mu$ is a free $\A_0$-submodule of $M$ such that
$$ M \simeq \Q(q) \otimes_{\A_0} L \quad\text{ and }\quad L_\mu = L \cap M_\mu,$$
\item $B= \bigoplus_\mu B_\mu$ is a basis of the $\Q$-vector space $L/qL$, where $B_\mu=B \cap (L_\mu/qL_\mu)$,
\item $\tilde{e}_i$ and $\tilde{f}_i$ ($i \in I$) are defined on $L$; i.e., $\tilde{e}_iL, \tilde{f}_iL \subset L$,
\item the induced maps $\tilde{e}_i$ and $\tilde{f}_i$ on $L/qL$ satisfy
$$ \tilde{e}_iB, \tilde{f}_iB \subset B \sqcup \{0\}, \quad \text{ and } \quad  \tilde{f}_ib=b' \quad \text{ if and only if } \quad b=\tilde{e}_ib' \quad  \text{ for } b,b' \in B.$$
\end{enumerate}

\end{definition}
The set $B$ has a colored oriented graph structure as follows:
$$ b \overset{i}{\longrightarrow} b' \quad\text{ if and only if }\quad \tilde{f}_ib=b'.$$
The graph structure encodes the structure information of $M$. For example,
\begin{itemize}
\item $|B_\mu| =\dim_{\Q(q)}M_\mu$ for all $\mu \in {\rm wt}(M)$.
\item $B$ is connected if and only if $M$ is irreducible.
\end{itemize}

It is shown in \cite{Kash91} that $V(\Lambda)$
has a unique \emph{crystal basis} $(L(\Lambda),B(\Lambda))$.
Thus $\chi^\Lambda_{\g_n}$ and $\chi_{\g_n}^{\Lambda}(t)$ can be expressed as follows:
\begin{align} \label{eqn:ch in crystal}
 \chi^{\Lambda}_{\g_n}= \sum_{\mu \in {\rm wt}(V(\Lambda))} |B(\Lambda)_\mu|e(\mu), \ \
\chi_{\g_n}^{\Lambda}(t)= \sum_{m \in \Z_{\ge 0}}
\left(\sum_{ \substack{ \mu \in {\rm wt}(V(\Lambda)), \\ \ \het(\Lambda -\mu)=m} } |B(\Lambda)_\mu|\right)t^m.
\end{align}
(See \cite{HK02, Kash93} for more details on the crystal basis theory.)

\section{Young walls and various partitions} \label{Sec: Young wall Partition}

\subsection{Young walls}
In \cite{K03}, Kang gave realizations of level $1$ highest weight crystals $B(\Lambda)$ for all classical
quantum affine algebras in terms of \emph{reduced Young walls} (see \cite{HKL04} as well).
From now on, we assume that $\g_n$ is of type $A^{(2)}_{2n}$, $A^{(2)}_{2n-1}$, $B^{(1)}_n$, $D^{(1)}_n$ or
$D^{(2)}_{n+1}$. Basically, Young walls are built from colored blocks.
There are three types of blocks whose shapes are different and which appear depending on type as follows:

\begin{center}
\begin{tabular}{|c c c c |c|} \hline
Shape & Width & Thickness & Height & Type \\ \hline
\ublock  & 1 & 1 & 1 &  all types \\
\hhblock & 1 & 1 & 1/2 & $A^{(2)}_{2n},B^{(1)}_{n},D^{(2)}_{n+1}$ \\
\htblock & 1 & 1/2 & 1 & $A^{(2)}_{2n-1},B^{(1)}_{n},D^{(1)}_{n}$\\ \hline
\end{tabular}
\end{center}

Given $\g_n$ and a dominant integral weight $\Lambda$ of level 1, we fix a frame $Y_\Lambda$ called the
\emph{ground-state Young wall} of weight $\Lambda$.
The set of Young walls is the set of blocks built on the ground-state Young wall by the following
rules:
\begin{enumerate}
\item[(a)] All blocks should be placed on the top of the ground-state Young wall or another block.
\item[(b)] The colored blocks should be stacked in the given pattern depending on $\g_n$ and $\Lambda$.
\item[(c)] No block can be placed on the top of a column of half-thickness.
\item[(d)] Except for the right-most column, there should be no free space to the right of any blocks.
\end{enumerate}

The patterns of Young walls are given as follows:
\begin{align*}
& \begin{tabular}{|c|c|c|c|} \hline $D^{(2)}_{n+1}$ & $A^{(2)}_{2n}$
&  $A^{(2)}_{2n-1}$ & $B^{(1)}_{n}$\\ \hline
 &  &  & \\
 $ \ \FPatDnpO \ \ \SPatDnpO \ $ & $ \ \PatAtnt \ $ &
 $ \ \FPatAtnmo \ \ \SPatAtnmo \ $  & $\ \TPatBnO \ $\\
 $\quad$ $\Lambda_0$ $\qquad\qquad$ $\Lambda_n$
 & $\quad$ $\Lambda_0$
 & $\quad$ $\Lambda_0$ $\qquad\qquad$ $\Lambda_1$
 & $\quad$ $\Lambda_n$ \\ \hline
\end{tabular}
\allowdisplaybreaks \\
& \begin{tabular}{|c|c|} \hline $B^{(1)}_{n}$ & $D^{(1)}_{n}$   \\
\hline
 &   \\
$\ \FPatBnO \ \ \SPatBnO \ $ & $\quad \FPatDnO \ \ \ \SPatDnO \ \ \ \TPatDnO \ \ \ \FoPatDnO \ \ $ \\
$\quad$ $\Lambda_0$ $\qquad\qquad$ $\Lambda_1$ &
 $\Lambda_0$ $\quad\qquad\qquad$ $\Lambda_1$
$\qquad\qquad\qquad$ $\Lambda_{n-1}$ $\qquad\qquad\qquad$ $\Lambda_n$\\ \hline
\end{tabular}
\end{align*}
\label{table:Dynkin}
Here the shaded part denotes the ground-state Young wall $Y_\Lambda$. The ground-state Young walls are described in the following way:
\begin{align*}
\begin{tabular}[c]{|c|c|c|c|} \hline
$\gyzz$ & $\Gyzz$&
$\gynn$  &  $\Gynn$ \\ \hline
$\gyoz$ & $\Gyzo$ &
$\gyzo$ &  $\Gyoz$ \\ \hline
$\gynnmo$ & $\Gynnmo$ &
$\gynmon$ &  $\Gynmon$ \\ \hline
\end{tabular}
\end{align*}

We write a Young wall $Y=(y_k)^{\infty}_{k=1}$ as an infinite sequence of its columns
where the columns are enumerated from right to left.

\begin{example} For $\g=B_3^{(1)}$ and $\Lambda=\Lambda_0$, we use the colored block
$$ \htbox{0} \quad \ \   \htbox{1} \quad \ \  \ubox{2} \quad \ \  \hhbox{3}$$
The following object is a Young wall.
$$ 
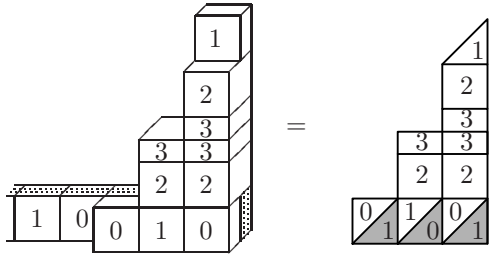
{\xy (0,14)*++{\youngwall} \endxy} \quad \Bthreeoneex$$
\end{example}

A column in a Young wall is called a \emph{full column} if its height is a
multiple of unit length and its top is of unit thickness.
We say that a Young wall is \emph{proper} if none
of the full columns have the same height.

\begin{example} \label{ex: forall}
 For $\g=B_3^{(1)}$ and $\Lambda=\Lambda_0$, let us consider the following four Young walls:
$$\Bthreeoneextwo$$
Then one can check that $Y^1$ is not proper while the others are proper.
\end{example}

The {\it part} of a column
consisting of $a_i$-many $i$-blocks, for each $i \in I$, in some
cyclic order is called a \emph{$\delta$-column}.

\begin{example} \label{ex: delta}
 For $\g=B_3^{(1)}$ and $\Lambda=\Lambda_0$, the following are $\delta$-columns.
$$ \Done \qquad \Dtwo \qquad \Dthree$$
\end{example}

\begin{definition} \
\begin{enumerate}
\item A column in a proper Young wall is said to contain a
{\em removable $\delta$} if we may remove a $\delta$-column
from $Y$ and still obtain a proper Young wall.
\item A proper Young wall is said to be {\em reduced}
 if none of its columns contain a removable $\delta$.
\end{enumerate}
\end{definition}

In Example \ref{ex: forall}, $Y^2$ is not a reduced Young wall, while $Y^3$ and $Y^4$ are reduced proper Young walls.

\medskip

For a given Young wall $Y$, we define the {\it weight} $\wt(Y)$ of $Y$ as follows:
\begin{align} \label{eqn:weight on partitions}
\wt(Y)= \Lambda-\sum_{i \in I} m_i \alpha_i.
\end{align}
Here $m_i$ is the number of $i$-blocks on the ground-state Young wall $Y_{\Lambda}$.

Let $\mathtt{Z}(\Lambda)$ be the set of all proper Young walls and $\mathtt{Y}(\Lambda)$ be the set of all reduced proper Young walls.

\begin{theorem} \cite{K03}
\begin{enumerate}
\item $\mathtt{Z}(\Lambda)$ has a crystal structure induced by Kashiwara operators $\tilde{e}_i$ and $\tilde{f}_i$.
\item The set $\mathtt{Y}(\Lambda)$ is closed under Kashiwara operators $\tilde{e}_i$ and $\tilde{f}_i$.
Moreover, there is a crystal isomorphism between
$\mathtt{Y}(\Lambda)$ and $B(\Lambda)$.
\end{enumerate}
\end{theorem}

\begin{definition} \label{Def:character of set}
For $m \in \Z{\ge 0}$, $\mu \in P$ and a subset $A$ of $\mathtt{Z}(\Lambda)$, we define
\begin{enumerate}
\item $A[m]$ to be the subset of $A$ which has $m$ blocks on the ground-state Young wall $Y_{\Lambda}$,
\item $A[\mu]$ to be the subset of $A$ consisting of Young walls with their
      weight $\mu$,
\item the {\it virtual character} $\vch_{\g_n}(A)$ of $A$ to be
$$ \vch_{\g_n}(A)= \sum_{\mu \in P} |A[\mu]|e(\mu).$$
\end{enumerate}
\end{definition}

Then the equations in $\eqref{eqn:ch in crystal}$ tell that
\begin{equation}  \label{eqn:ch in YW}
\chi^{\Lambda}_{\g_n} =\vch_{\g_n}(\mathtt{Y}(\Lambda))  =\sum_{\mu \in P} |\mathtt{Y}(\Lambda)[\mu]|e(\mu), \ \
\chi_{\g_n}^{\Lambda}(t)= \sum_{m}|\mathtt{Y}(\Lambda)[m]| t^m.
\end{equation}

Set
$$ \epsilon = \begin{cases} 2 & \text{ if $\g_n$ is of type $D^{(2)}_{n+1}$}, \\
1 & \text{ otherwise.} \end{cases} $$

\begin{proposition} \label{Prop:Fock decomposition} \cite[Corollary 2.5]{KK04} \
$$\mathtt{Z}(\Lambda) = \bigoplus_{k \in \Z_{\ge 0}} B(\Lambda-\epsilon k  \delta)^{\oplus |\OP(k)|},$$
where $B(\Lambda-m\delta)$ is the crystal of the highest weight module $V(\Lambda-m\delta)$ for $m \in \Z_{>0}$.
\end{proposition}

In terms of Young walls, Proposition \ref{Prop:Fock decomposition} can be interpreted as follows:
\begin{equation} \label{eqn:Fock decomposition}
|\mathtt{Z}(\Lambda)[m]| =  \displaystyle \sum_{ k \ge 0,  \ m- k \Delta  \ge
0} (|\mathtt{Y}(\Lambda)[m-k \Delta] | \times |\OP(k)|).
\end{equation}

\subsection{Various partitions} \label{subsec: var par}
Set $\N \seteq \Z_{> 0}$. We denote by
\begin{itemize}
\item $\overline{\N}$ the set of positive integers which are overlined,
\item $\mathsf{N}=\N \sqcup \overline{\N} \sqcup \{ 0 \}$.
\end{itemize}

We assign a linear order $\succ$ on $\mathsf{N}$ by defining
$$  \cdots \succ x \succ \overline{x} \succ x-1 \succ \overline{x-1} \succ \cdots \succ 1 \succ \overline{1} \succ 0.$$

To distinguish overlined integers and normal integers, we define a map
$c: \mathsf{N} \setminus \{ 0 \} \to \Z_2$ by
$$c(\overline{x})=1 \quad \text{and} \quad  c(x)=0 \quad \text{ for  all }  x \in \N.$$

\begin{definition} \label{Def: 2-partition} With the linear order $\succ$ on $\mathsf{N}$, we can define the notion
of the set of partitions $\OP^2$ which consists of sequences in $\mathsf{N}$; i.e.,
$$ \lambda = (\lambda_1,\lambda_2,\ldots) \in \OP^2 \ \ \text{ if and only if } \ \ \lambda_i \in \mathsf{N} \text{ and }
 \lambda_i \succeq \lambda_{i+1} \text{ for all } i \in \Z_{\ge 1}.$$
An element $\lambda \in \OP^2$ is called a {\it two-colored partition}.
\end{definition}

For a two-colored partition $\lambda = (k^{m_k},
\overline{k}^{m_{\overline{k}}} \cdots,
2^{m_2},\overline{2}^{m_{\overline{2}}},
1^{m_1},\overline{1}^{m_{\overline{1}}}) \in \OP^2$, we say that
$\lambda$ is a two-colored partition of $m \in \Z_{\ge 0}$, if
$$\sum^{k}_{i=1} i(m_i+m_{\overline{i}})=m, \text{ and write }
\Sigma_{\lambda}=m \text{ or } \lambda \vdash m.$$ For each subset $\mathscr{A}$ of $\OP^2$ and
$m \in \Z_{\ge 0}$, we denote $\mathscr{A}[m] \seteq \{ \lambda
\in \mathscr{A} \ | \ \Sigma_{\lambda}=m \}$.

For $a \succ b \in \mathsf{N}$ such that $\overline{a} \neq b$, we define $a-b$ to be the element in $\mathsf{N}$ whose value is given by the ordinary subtraction
of their values, and $c(a-b) \equiv c(a)-c(b) \ {\rm mod} \ 2$. In particular, if
$\overline{a} = b$, we define $a-b=0$. Similarly, we can define $a+b$, for $a,b \in \mathsf{N}$.
\begin{example}
$\overline{3}-2=\overline{1}$, $\overline{3}+2=\overline{5}$,
$\overline{2}-2=0$ and $\overline{2}+0=\overline{2}$.
\end{example}
We also define $\Z_{\ge 0}$-multiplication $\cdot: \Z_{\ge 0} \times \mathsf{N} \to \mathsf{N} $ by
$$(k,x) \longmapsto k\cdot x \quad \text{with} \quad
c(x)=c(k\cdot x).$$ For instance, $2 \cdot \overline{2} = \overline{4}$.

For sequences  $\lambda=(\lambda_1,\ldots,\lambda_t)$ and
$\mu=(\mu_1,\ldots,\mu_l)$ of $\mathsf{N}$, we define
\begin{itemize}
\item[(i)] $\ell(\lambda)=t$, the {\it length} of $\lambda$,
\item[(ii)] $\mu * \lambda=(\mu_1,\ldots,\mu_l,\lambda_1,
\ldots,\lambda_t)$, the \emph{concatenation} of $\mu$ and $\lambda$,
\item[(iii)] $\re \lambda=(\lambda_t,\lambda_{t-1},\ldots,\lambda_1)$, the
\emph{reverse} of $\lambda$.
\end{itemize}

\begin{definition} \cite{CL}\hfill
\begin{enumerate}
\item An {\it overpartition} is a weakly decreasing sequence of $\N$ in which the first occurrence
of the number may be overlined. We denote by $\OP^o$ the set of all overpartitions.
\item We denote by $\OP^t$ the subset of $\OP^2$ consisting of $\lambda$'s satisfying the following condition: \\
For every $x \in \N$, both $x$ and $\overline{x}$ can not be parts of $\lambda$, simultaneously. More precisely,
\begin{itemize}
\item  if $x$ is a part of $\lambda$, then $\overline{x}$ is not a part of $\lambda$,
\item  if $\overline{x}$ is a part of $\lambda$, then $x$ is not a part of $\lambda$.
\end{itemize}
\end{enumerate}
\end{definition}

\begin{example}
\begin{enumerate}
\item The number of overpartitions of $3$ is $8$:
$$ (3), \ (\overline{3}), \ (2,1), \ (\overline{2},1), \ (2,\overline{1}), \ (\overline{2},\overline{1}),
\ (1,1,1), \ (\overline{1},1,1).$$
\item $(5,\overline{3},\overline{3},2,2,1) \in \OP^t[16]$, $(5,3,\overline{3},2,2,1) \not \in \OP^t[16]$.
\end{enumerate}
\end{example}

\begin{definition} For a given proper Young wall $Y=(y_k)^{\infty}_{k=1}$, we define several sequences
which are associated to $Y$ as follows:
\begin{enumerate}
\item $\mathsf{P}(Y)=(p(y_k))^{\infty}_{k=1}$ is the sequence in $\Z_{\ge 0}$, where
$p(y_k)$ is the number of blocks in $k$th column of Y on $Y_\Lambda$,
\item $\mathsf{P}^t(Y)=(p^t(y_k))^{\infty}_{k=1}$ is the sequence in $\mathsf{N}$, where
$p^t(y_k)$ is $\overline{p(y_k)}$ if the top of the $k$th column is a half-thickness block
and placed in the front side, and $p(y_k)$ otherwise.
\end{enumerate}
\end{definition}

In Example \ref{ex: forall}, we have
$$\mathsf{P}^t(Y^3)=(6,6,1)  \quad \text{ and } \quad   \mathsf{P}^t(Y^4)=(\overline{6},\overline{6},1).$$

For $m_1,m_2,m_3 \in \Z_{\ge 0}$, we define
\begin{itemize}
\item $\mathsf{N}\{(m_1,m_2)\}\seteq \N \cup (m_1 \cdot \overline{\N}+m_2) \cup
\{ \overline{m_1}, \  0 \}$ and
\item $\OS^t\{(m_1,m_2);m_3\}$ to be
the subset of $\OP^t$ consisting of $\lambda=(\lambda_1,\lambda_2,\ldots)$'s satisfying the following:
\begin{itemize}
\item[(i)] $\lambda_i \in \mathsf{N}\{(m_1,m_2)\}$ for all $i \in \Z_{\ge 1}$,
\item[(ii)] $\lambda_i=\lambda_{i+1}\succ 0$ if and only if $\lambda_i=k \cdot m_1+m_2$, $k\cdot \overline{m_1}+m_2$, $m_1$,
$\overline{m_1}$ or $k \cdot m_3$ for some $k \in \Z_{\ge 0}$.
\end{itemize}
\end{itemize}

For each type of $\g_n$ and $\Lambda$, we define
$\mathsf{Z}=(\z_1,\z_2,\z_3)$, $\U$ and $\V$:
\begin{equation} \label{Table: W^1 W^2}
\begin{tabular}{|c | c | c | c | c |} \hline
Type & $\Lambda$ & $\mathsf{Z}=(\z_1,\z_2,\z_3)$ & $\U$& $\V$   \\ \hline
$A^{(2)}_{2n}$ & $\Lambda_0$ & $(0,0,2n+1)$ & $2n+1$ & $2n+1$\\
$A^{(2)}_{2n-1}$ &$\Lambda_0,\Lambda_1$ & $(2n-1,0,2n-1)$ & $2n-1$ & $2n-1$\\
$B^{(1)}_{n}$ &$\Lambda_0,\Lambda_1$ & $(2n,0,n)$ & $2n$ & $2n$\\
$B^{(1)}_{n}$ &$\Lambda_n$ & $(n,2n,n)$ & $2n$ & $2n$\\
$D^{(1)}_{n}$ &$\Lambda_0,\Lambda_1,\Lambda_{n-1},\Lambda_n$ & $(n-1,0,n-1)$ &
$2n-2$ & $n-1$ \\
$D^{(2)}_{n+1}$ & $\Lambda_0,\Lambda_n$ & $(0,0,n+1)$ & $2n+2$ & $n+1$\\  \hline
\end{tabular}
\end{equation}

For each given $\mathsf{Z}$, $\U$ and $\V$, we define the subsets $\mathcal{Z}(\Lambda)$ and $\mathcal{AO}_i(\Lambda)$
$(i=1,2)$ of $\OP^t$ as follows:
\begin{enumerate}
\item $\mathcal{Z}(\Lambda) \seteq \OS^t\{(\z_1,\z_2);\z_3\}$.
\item Let $\mathcal{AO}_1(\Lambda)$ be the subset of $\mathcal{Z}(\Lambda)$ satisfying \\
$\bullet$ the difference between successive parts is at most $\U$
and the smallest part is strictly less than $\U$ with respect to the linear order $\succ$, \\
$\bullet$ the difference between successive parts is strictly less than $\U$ with respect to the linear order $\succ$, if either part is congruent to $\z_3,
\overline{ \z_3}$ or $0$ modulo $\V$.
\item Let $\mathcal{AO}_2(\Lambda)$ be the subset of $\mathcal{Z}(\Lambda)$ as follows:
\begin{equation} \label{Table: AO_2}
\begin{tabular}{|c | c | } \hline
Type &
$\mathcal{AO}_2(\Lambda)$   \\ \hline
$A^{(2)}_{2n-1}$
& Each part is not a multiple of $\U$ and no part can be repeated.\\ \hline
$A^{(2)}_{2n}$, $B^{(1)}_{n} \ (\Lambda=\Lambda_n)$, $D^{(1)}_{n}$
& Each part is not a multiple of $\U$.
\\ \hline
$B^{(1)}_{n} \ (\Lambda \neq \Lambda_n)$ & Each part is not a multiple of $\U$ and \\
& only the parts congruent to $\z_1$ modulo $\U$ can be repeated.\\ \hline
$D^{(2)}_{n+1}$
& No part can be repeated. \\ \hline
\end{tabular}
\end{equation}
\end{enumerate}

\begin{remark} \label{Rmk: AO2}\hfill
\begin{enumerate}
\item The map $\mathsf{P}^t : \mathtt{Z}(\Lambda) \to \OP^2$ is injective and we have
$$ \mathrm{Im}(\mathtt{Z}(\Lambda))=\mathcal{Z}(\Lambda) \quad \text{ and } \quad \mathrm{Im}(\mathtt{Y}(\Lambda))=\mathcal{AO}_1(\Lambda).$$
\item The map $\mathsf{P} : \mathtt{Z}(\Lambda) \to \OP$ is injective for types $D^{(2)}_{n+1}$ and $A^{(2)}_{2n}$.
But in general, $\mathsf{P}$ is not injective. In Example \ref{ex: forall},
$\mathsf{P}(Y^3)=\mathsf{P}(Y^4)=(6,6,1)$ even though $Y^3 \ne Y^4$.
\item For type $D^{(2)}_{n+1}$, we have
$  \mathsf{N}\{(\z_1,\z_2)\}=\mathsf{N}\{(0,0)\}  = \Z_{\ge 0}$. Thus the condition given in \eqref{Table: AO_2} implies that
$\mathcal{AO}_2(\Lambda)$ is indeed the set of all strict partitions, denoted by $\OS$.
\item For type $A^{(2)}_{2n-1}$, we have $\mathsf{N}\{(\z_1,\z_2)\} = \mathsf{N}\{(2n-1,0)\} = \Z_{\ge 0} \sqcup 2n-1 \cdot \overline{\N}$.
Thus the condition given in \eqref{Table: AO_2} implies that
$\mathcal{AO}_2(\Lambda)$ can be identified with the set of all strict partitions $\OS$.
\item For type $A^{(2)}_{2n}$, we have
$  \mathsf{N}\{(\z_1,\z_2)\}=\mathsf{N}\{(0,0)\}  = \Z_{\ge 0}$. Thus the condition given in \eqref{Table: AO_2} implies that
we can identify $\mathcal{AO}_2(\Lambda)$ with the subset of strict partitions
$\OS_{(\U)}$ which is defined as follows:
\begin{align} \label{Def: S_{U}}
\lambda=(\lambda_1,\lambda_2,\ldots) \in \OS_{(\U)} \text{ if and only if } \lambda \text{ is strict and } \U  \not| \ \lambda_i \text{ for all
 $i \in \N$}.
\end{align}
\end{enumerate}
\end{remark}

\section{ The Andrews-Olsson identity } \label{Sec: Anderws-Olsson identity}

In this section, we prove the following theorem.
{\it Hereafter, we drop $(\Lambda)$ of notations given in Section
\ref{Sec: Young wall Partition}}.

\begin{theorem} \label{Thm: Andrews-Olsson identity}
For all $m \in \N$, we have
$$ |\mathcal{AO}_1[m]| = |\mathcal{AO}_2[m]|.$$
\end{theorem}

\noindent For $i=1,2$, we denote by $\mathcal{AO}^c_i$ the
complement of $\mathcal{AO}_i$ in $\mathcal{Z}$.

\begin{proposition}
For all $m \in \N$, there are bijections given as follows:
$$\overline{\Psi}[m]:\mathcal{AO}_1^{c}[m] \to \displaystyle
\bigsqcup_{k > 0, \ m-k  \U \ge 0} (\mathcal{AO}_1[m-k
 \U ]  \times \OP[k]).$$
\end{proposition}

\begin{proof}
In this proof, we give an explicit bijection between the above two sets.
Let $\Psi[m]$ be a map from $\mathcal{AO}_1^{c}[m] $ to
$\displaystyle \bigsqcup_{k > 0, \ m-k \U \ge 0} \mathcal{AO}_1[m-k
\U]$ given by the following algorithm $({\mathbf A})$:
\begin{itemize}
\item[(${\mathbf A}$1)] Let $Y =(y_1,y_2,\ldots) \in \mathcal{AO}_1^{c}[m]$ be given. Set $Y^{(0)}=Y$ and $l=0$.
\item[(${\mathbf A}$2)] Find the maximal $i$ such that
\begin{align} \label{condition A2}
y^{(l)}_{i-1}-y^{(l)}_i \succeq t \cdot \U  \text{ for some } \ t \in \Z_{> 0} \ \text{ and } \
(y^{(l)}_{i-1}-t \cdot \U ,y^{(l)}_i) \in \mathcal{Z}.
\end{align}
\item[(${\mathbf A}$3)] Among $t$'s satisfying the condition in $\eqref{condition A2}$, choose the maximal one and say $t_{l+1}$.
Set
$$Y^{(l+1)}\seteq(y^{(l)}_1-t_{l+1} \cdot \U ,y^{(l)}_2-t_{l+1} \cdot \U ,...,y^{(l)}_{i-1}-t_{l+1}\cdot \U ,y^{(l)}_i,y^{(l)}_{i+1},...).$$
\item[(${\mathbf A}$4)]
If $Y^{(l+1)} \in \mathcal{AO}_1$, then
define $Y'=Y^{(l+1)}$ and terminate this algorithm. Otherwise, set $l=l+1$ and go to $({\mathbf A}2)$.
\end{itemize}
This algorithm terminates in finitely many steps and we have
\begin{align*}
 k\seteq\dfrac{\Sigma_Y-\Sigma_{Y'}}{ \U } \in \N, \  Y' \in \mathcal{AO}_1[m-k  \U], \
 \lambda_i\seteq\dfrac{y_i-y'_i}{\U } \in \Z_{\ge 0}, \ \lambda\seteq(\lambda_1,\lambda_2,\ldots)
 \in \OP[k].
\end{align*}
Thus we obtain the map
$$\overline{\Psi}[m]: \ \mathcal{AO}_1^{c}[m] \to \displaystyle \bigsqcup_{k > 0, \ m-k  \U \ge 0}
(\mathcal{AO}_1[m-k \U]  \times  \OP[k] )$$ given by $Y \mapsto
(Y',\lambda)$. Then $\overline{\Psi}[m]$ becomes a bijection.
Moreover, the preimage $Y$ of $(Y',\lambda') \in
\mathcal{AO}_1[m-k\U]  \times \OP[k]$ under the map
$\overline{\Psi}[m]$ is given as follows:
$$Y\seteq(y'_1+ \lambda'_1 \cdot \U ,y'_2+ \lambda'_2 \cdot \U ,...,y'_k+ \lambda'_k \cdot \U ,..). $$
\end{proof}

\begin{definition}
For a given two-colored partition $Y=(y_1,y_2,\ldots) \in \mathcal{Z}(\Lambda)$
and $e,j \in \N$, we define the {\it left insertion} of
$(\underbrace{e,\ldots,e}_j)$ into $Y $ to be a two-colored partition in
$\mathcal{Z}(\Lambda)$,
denoted by $(e)^j \hookrightarrow Y$, as follows:
$$
(e)^j \hookrightarrow Y \seteq \begin{cases} (y_1,y_2,...,y_i,
\underbrace{\overline{e},...,\overline{e}}_{j}, y_{i+1} ,..) &
\text{ if } y_i = \overline{e}, \\
(y_1,y_2,...,y_i,\underbrace{e,...,e}_{j},y_{i+1} ,..) & \text{ otherwise },
\end{cases} \text{ for } y_i \succeq \overline{e}  \succ y_{i+1}. $$
\end{definition}

Recall that $\epsilon=2$ if $\g_n$ is of type $D^{(2)}_{n+1}$, and $\epsilon=1$ otherwise.

\begin{proposition} \label{Prop: Algorithm B}
For all $m \in \N$, there are bijections given as follows:
$$\overline{\Phi}[m]:\mathcal{AO}_2^{c}[m] \to
\displaystyle \bigsqcup_{k > 0, \ m-k  \U \ge 0}
(\mathcal{AO}_2[m-k  \U] \times \OP[k] ).$$
\end{proposition}

\begin{proof}
Let $\Phi[m]$ be a map from $\mathcal{AO}_2^{c}[m] $ to
$\displaystyle \bigsqcup_{k > 0, \ m-k \U \ge 0} \mathcal{AO}_2[m-k
\U]$ given by the following algorithm $({\mathbf B})$:
\begin{itemize}
\item[(${\mathbf B}$1)] Let $Y \in \mathcal{AO}_2^{c}[m]$ be given. Set $Y^{(0)}=Y$, $\lambda^{(0)}=(0)$,
$\widehat{\lambda}_0=0$ and $l=0$.
\item[(${\mathbf B}$2)] Find the maximal $i$ such that
\begin{align*}
y^{(l)}_i =\widehat{\lambda}_{l+1} \cdot \epsilon^{-1}\U \ \text{ or }
\ \widehat{\lambda}_{l+1} \cdot \overline{\epsilon^{-1}\U} \ \text{ for
some } \ \widehat{\lambda}_{l+1} \in \N \text{ with } \widehat{\lambda}_{l+1}> \lambda^{(l)}_{1}
\end{align*}
and set $t_{l+1}$ to be the number of parts in $Y^{(l)}$ which are equal to $y^{(l)}_i$.
\item[(${\mathbf B}$3)] Set $a = \lfloor \epsilon^{-1} t_{l+1} \rfloor- c(y^{(l)}_i)$
and define
$$Y^{(l+1)}\seteq(y^{(l)}_1,y^{(l)}_2,...,y^{(l)}_{i-a},y^{(l)}_{i+1},...)
\ \text{ and } \ \lambda^{(l+1)}=
(\underbrace{\widehat{\lambda}_{l+1},...,\widehat{\lambda}_{l+1}}_a)*\lambda^{(l)}.
$$
\item[(${\mathbf B}$4)] If $Y^{(l+1)} \in \mathcal{AO}_2$, then
define $Y'=Y^{(l+1)}$ and terminate this algorithm. Otherwise, set $l=l+1$ and go to $({\mathbf B}2)$.
\end{itemize}

This algorithm terminates in finitely many steps and we have
$$k \seteq \dfrac{\Sigma_Y-\Sigma_{Y'}}{\U} \in \N, \quad
Y' \in \mathcal{AO}_2[m-k  \U], \quad \lambda^{(l)} \in \OP[k].$$
Thus we obtain the map
$$\overline{\Phi}[m]: \ \mathcal{AO}_2^{c}[m] \to \displaystyle
\bigsqcup_{k > 0, \ m-k  \U \ge 0} (\mathcal{AO}_2[m-k \U] \times
\OP[k] )$$ given by $Y \mapsto (Y',\lambda^{(l)})$. Then
$\overline{\Phi}[m]$ is a bijection. Moreover, the preimage $Y$ of
$(Y',\lambda') \in \mathcal{AO}_2[m-k  \U]  \times  \OP[k]$ under
the map $\overline{\Phi}[m]$ is given as follows:
$$Y\seteq (\lambda'_r \cdot\epsilon^{-1}\U)^{1+\lfloor\epsilon/2 \rfloor} \hookrightarrow (\lambda'_{r-1} \cdot\epsilon^{-1}\U)^{1+\lfloor\epsilon/2 \rfloor} \hookrightarrow \cdots (\lambda'_1 \cdot\epsilon^{-1}\U)^{1+\lfloor\epsilon/2 \rfloor} \hookrightarrow Y',$$
where $\lambda'=(\lambda'_1,\ldots,\lambda'_{r-1},\lambda'_{r}).$
\end{proof}

\begin{example} In this example, we assume that $\g_{n}=A^{(2)}_{7}$ and apply Algorithm $({\mathbf B})$ for
$$  Y=(20,14,14,13,11,\overline{7},\overline{7},\overline{7},5,3) \in \mathcal{AO}_2^c[102].$$
Note that $\U=7$ and $\epsilon=1$. In the first round, the $i$ given in $({\mathbf B}2)$ is $8$. Thus we have
$$ \widehat{\lambda}_1=1 \quad \text{ and } \quad t_1= 3.$$
Since $y_8=y^{(0)}_8=\overline{7}$, the $a$ given in $({\mathbf B}3)$ is $t_1-1=2$. Hence the results in the first round are
$$ Y^{(1)}=(20,14,14,13,11,\overline{7},5,3) \quad \text{ and } \quad \lambda^{(1)}=(1,1).$$
As $y^{(1)}_2=y^{(1)}_3= 14 =2\U$, we have to run the second round. In this round,  the $i$ given in $({\mathbf B}2)$ is $3$ and
$$ \widehat{\lambda}_2=2 \quad \text{ and } \quad t_2= 2.$$
Because $c(14)=0$, the $a$ given in $({\mathbf B}3)$ is equal to $t_2=2$. Hence
$$ Y^{(2)}=(20,13,11,\overline{7},5,3) \quad \text{ and } \quad \lambda^{(2)}=(2,2,1,1).$$
Finally, we have $ Y^{(2)} \in \mathcal{AO}_2[59]$. Thus ($Y^{(2)}$,$\lambda^{(2)}$) is the result obtained by Algorithm $({\mathbf B})$.
\end{example}

\noindent \textbf{Proof of Theorem \ref{Thm: Andrews-Olsson
identity} } By definitions, $\mathcal{AO}_1[t]=\mathcal{AO}_2[t]$
for $0 \le t \le \U-\lfloor \epsilon/2 \rfloor$. For the case of
$D^{(2)}_{n+1}$, $\mathcal{AO}_1[\U] \setminus \{
(\V,\V,0,...) \}
=\mathcal{AO}_2[\U] \setminus \{ (\U,0,0,...) \}$. Thus
$$|\mathcal{AO}_1[t]|=|\mathcal{AO}_2[t]| \quad \text{for} \quad 0 \le t \le \U.$$
Proposition \ref{Prop: Algorithm B} tells us that for all $m >
\U$, $|\mathcal{AO}_i^c[m]|$ ($i=1,2$) depends on the sets $\mathcal{AO}_i[l]$ and $\OP[k]$ satisfying $k =
\dfrac{m-l}{\U} \in \N$. Using an induction on $m$, we can
conclude that  $|\mathcal{AO}_1^c[m]|=|\mathcal{AO}_2^c[m]|$. Hence
$$|\mathcal{AO}_1[m]|=|\mathcal{AO}_2[m]|.$$

Recall the definition of $\OS_{(\U)}$ given in \eqref{Def: S_{U}}.

\begin{corollary} \label{Cor: generating function}
For the types of $D^{(2)}_{n+1}$, $A^{(2)}_{2n}$ and $A^{(2)}_{2n-1}$,
\begin{enumerate}
\item the generating functions of $\mathcal{AO}_i$ $(i=1,2)$ are
\begin{align*}
    \chi_{\g_n}^{\Lambda}(t)=\prod_{i=1}^{\infty} (1+t^i)^{\kappa_i}, \quad \text{ where }
 \begin{cases} \kappa_i=0 &\text{ if } \g_n= A^{(2)}_{2n} \text{ and } i \equiv 0 \ {\rm mod} \ \U, \\
               \kappa_i=1 &\text{ otherwise},\end{cases}
\end{align*}
\item $ |\mathcal{Z}[m]| = \displaystyle \sum_{k \ge 0,\ m- k \U \ge 0}
(|\OS'[m-k \U] | \times |\OP[k]|)$, where $\begin{cases} \OS'=\OS_{(\U)} &\text{ if } \g_n= A^{(2)}_{2n}, \\
              \OS'=\OS &\text{ otherwise}.\end{cases}$
\end{enumerate}
\end{corollary}

\begin{proof}
Note that $\Delta$ in $\eqref{eqa:Delta}$ coincides with $\U$.
By Proposition \ref{Prop:Fock decomposition} and Remark \ref{Rmk: AO2}, our assertions follow.
\end{proof}

\begin{remark} \label{rmk: injective p}
Note that we can define the virtual character on $\mathcal{AO}_i$
by regarding the partitions in $\mathcal{AO}_i$ as Young walls in $\mathtt{Z}(\Lambda)$ $(i=1,2)$. If
we have $\lambda=(3,2,1) \in \mathcal{AO}_2$ when $\g=D^{(2)}_3$ and $\Lambda=\Lambda_0$, $\lambda$ corresponds to the following Young wall:
$$\DthreetwoEx $$
Then one can assign a weight of $\lambda$ as $\Lambda_0-(3\alpha_0+2\alpha_1+2\alpha_2)$.
\end{remark}

The following theorem tells that Theorem \ref{Thm: Andrews-Olsson identity} for types of $D^{(2)}_{n+1}$ and
$A^{(2)}_{2n}$ can be interpreted in stronger sense.

\begin{theorem} \label{Thm: virtual cahracter}
For the types of $D^{(2)}_{n+1}$ and $A^{(2)}_{2n}$, we have
$$\chi_{\g_n}^{\Lambda}=\vch_{\g_n}(\mathcal{AO}_2). $$
\end{theorem}

\begin{proof}
By definition, Young walls of types  $D^{(2)}_{n+1}$ and $A^{(2)}_{2n}$ do not contain
half-thickness blocks. Then we have $\vch_{\g_n}(\mathtt{Y}[t]) =
\vch_{\g_n}(\mathcal{AO}_2[t])$ for $0 \le t \le \U$. Using an
induction, the assertion holds in the similar way of Theorem \ref{Thm:
Andrews-Olsson identity}.
\end{proof}

\begin{proposition} \label{Prop: C}
For the types of $B^{(1)}_{n}$, $D^{(1)}_{n}$ and $m \in \Z_{\ge 0}$, we have bijections
$$ \overline{\Xi}[m] : \mathcal{AO}_2[m]  \to \bigsqcup_{k=0}^{ \lfloor \frac{m}{ \z_1} \rfloor} (\OS'[m - k \z_1]
\times \OS[k]), \text{ where }
\begin{cases} \OS'=\OS_{(\U)} &\text{ if } \g_n= B^{(1)}_{n} \text{ and } \Lambda=\Lambda_n, \\
              \OS'=\OS &\text{ otherwise}.\end{cases}$$
\end{proposition}

\begin{proof} Set
$$\tilde{\epsilon}= \begin{cases}2 & \text{ if } \g_n = B^{(1)}_{n} \text{ and } \Lambda \neq \Lambda_n, \\
1 & \text{ otherwise}. \end{cases}$$

Let $\Xi[m]$ be a map from $\mathcal{AO}_2[m] $ to $\displaystyle
\bigsqcup_{k \ge 0, \ m-k \z_1 \ge 0}  \OS'[m-k \z_1]$ given by the
following algorithm $({\mathbf C})$:
\begin{itemize}
\item[(${\mathbf C}$1)] Let $Y \in \mathcal{AO}_2[m]$ be given. Set $Y^{(0)}=Y$,
$\lambda^{(0)}$=(0) $\widehat{\lambda}_0=0$ and $l=0$. If $|Y| \in \OS'[m]$, define $Y'=Y$ and terminate this algorithm.
\item[(${\mathbf C}$2)] Find the maximal $i$ such that
\begin{align*}
y^{(l)}_i =\widehat{\lambda}_{l+1} \cdot \tilde{\epsilon}^{-1}\z_1
\quad \text{ or } \quad \widehat{\lambda}_{l+1} \cdot
\overline{\tilde{\epsilon}^{-1}\z_1} \text{ for some
$\widehat{\lambda}_{l+1} \in 2\Z_{\ge 0}+1$} \text{ with } \widehat{\lambda}_{l+1} > \lambda^{(l)}_1
\end{align*}
and set $t_{l+1}$ to be the number of parts in $Y^{(l)}$ which are equal to $y^{(l)}_i$.
\item[(${\mathbf C}$3)] Set $a = \lfloor \frac{t_{l+1}}{\tilde{\epsilon}} \rfloor-c(y^{(l)}_i)$
and define
$$Y^{(l+1)}\seteq(y^{(l)}_1,y^{(l)}_2,...,y^{(l)}_{i-a},y^{(l)}_{i+1},...)
\text{ and } \lambda^{(l+1)}=
(\underbrace{\widehat{\lambda}_{l+1},...,\widehat{\lambda}_{l+1}}_a)*\lambda^{(l)}.
$$
\item[(${\mathbf C}$4)] If $Y^{(l+1)} \in \OS'$, then
define $Y'=Y^{(l+1)}$ and terminate this algorithm. Otherwise, set $l=l+1$ and go to $({\mathbf C}2)$.
\end{itemize}
This algorithm terminates in finitely many steps and we have
\begin{align*}
k\seteq\dfrac{\Sigma_Y-\Sigma_{\overline{Y}}}{\z_1} \in \Z_{\ge 0},
\ Y' \in \OS'[m-k \z_1], \ \lambda\seteq\lambda^{(l)} \in
\mathscr{O}[k].
\end{align*}
Thus we obtain the desired map
$$\overline{\Xi}[m]: \ \mathcal{AO}_2[m] \to \displaystyle
\bigsqcup_{k \ge 0, \ m-k \z_1 \ge 0} (\OS'[m-k \z_1] \times
\mathscr{O}[k] )$$ given by $Y \mapsto (Y',\lambda)$. Then
$\overline{\Xi}[m]$ is a bijection. Moreover the preimage $Y$ of
$(Y',\lambda') \in \OS'[m-k \z_1]  \times \OS[k] $ under the map
$\overline{\Xi}[m]$ is given as follows:
$$Y \seteq (\lambda'_r \cdot \tilde{\epsilon}^{-1}\z_1)^{1+\lfloor\tilde{\epsilon}/2 \rfloor} \hookrightarrow
(\lambda'_{r-1} \cdot
\tilde{\epsilon}^{-1}\z_1)^{1+\lfloor\tilde{\epsilon}/2 \rfloor}
\hookrightarrow \cdots (\lambda'_1 \cdot
\tilde{\epsilon}^{-1}\z_1)^{1+\lfloor\tilde{\epsilon}/2 \rfloor}
\hookrightarrow Y'.$$ Thus, by the Sylvester's bijection, we have a
bijection
$$\mathcal{AO}_2[m]  \to \bigsqcup_{k=0}^{ \lfloor \frac{m}{\z_1} \rfloor}(\OS'[m - k \z_1] \times \OS[k]).$$
\end{proof}

\begin{example} Assume that $\g_{n}=B^{(1)}_{3}$ and $\Lambda=\Lambda_n$. Now, we apply Algorithm $({\mathbf C})$ for
$$  Y=(31,17,\overline{15},\overline{15},13,7,5,3,3) \in \mathcal{AO}_2[109].$$
Note that $\U=6$, $\z_1=3$ and $\tilde{\epsilon}=1$. In the first round, the $i$ given in $({\mathbf C}2)$ is $9$. Thus we have
$$ \widehat{\lambda}_1=1 \quad \text{ and } \quad t_1= 2.$$
Since $y_8=y^{(0)}_8=3$, the $a$ given in $({\mathbf C}3)$ is $t_1=2$. Hence the results in the first round are
$$ Y^{(1)}=(31,17,\overline{15},\overline{15},13,7,5) \quad \text{ and } \quad \lambda^{(1)}=(1,1).$$
As $y^{(1)}_3=y^{(1)}_4= \overline{15} =3 \overline{\z_1}$, we have to run the second round. In this round,  the $i$ given in $({\mathbf C}2)$ is $4$ and
$$ \widehat{\lambda}_2=5 \quad \text{ and } \quad t_2= 2.$$
Because $c(\overline{15})=1$, the $a$ given in $({\mathbf C}3)$ is equal to $t_2-1=1$. Hence
$$ Y^{(2)}=(31,17,\overline{15},13,7,5) \quad \text{ and } \quad \lambda^{(2)}=(5,1,1).$$
Finally, we have $ Y^{(2)} \in \OS'[88]$. Thus  $\big( (31,17,15,13,7,5) ,(5,1,1) \big) $ is the result obtained by Algorithm $({\mathbf C})$.
\end{example}

By the similar argument in Corollary \ref{Cor: generating function}, we have the following corollary
from Proposition \ref{Prop: C}:

\begin{corollary} For the types of $B^{(1)}_{n}$ and $D^{(1)}_{n}$,
\begin{enumerate}
\item the generating functions of $\mathcal{AO}_i$ $(i=1,2)$ are
\\
$\chi_{\g_n}^{\Lambda}(t)=\prod_{i=1}^{\infty} (1+t^i)^{\kappa_i}$, where
 $\begin{cases}  &\text{ if } \g_n= B^{(1)}_{n}, \ \Lambda \neq \Lambda_n
              \text{ and } i \equiv 0 \ {\rm mod} \ \U, \\
              \kappa_i=2 & \text{ or if } \g_n= B^{(1)}_{n}, \ \Lambda = \Lambda_n
                \text{ and } i \equiv \z_1 \ {\rm mod} \ \U, \\
              & \text{ or if } \g_n= D^{(1)}_{n} \text{ and } i \equiv 0 \ {\rm mod} \ \z_1, \\
               \kappa_i=1 &\text{ otherwise},\end{cases}$

\item $|\mathcal{Z}[m]|=\displaystyle \sum_{ \substack{l \ge 0, \\ m-l  \U \ge 0}} \left\{
\left( \sum_{ \substack{k \ge 0, \\ m-l  \U -k \z_1 \ge 0}}
\big( |\OS'[m-l \U-k \z_1]|
 \times |\OS[k]| \big) \right) \times |\OP[l]| \right\}$.
\end{enumerate}
\end{corollary}

Since we consider only the parts which can be repeated in $\mathcal{Z}$, we
have an identity which is more general than the one in Theorem \ref{Thm:
Andrews-Olsson identity} as follows:

\medskip

\begin{corollary}\label{Cor: more general AO id}
Let $X_{\z_3}=\{ x_1,\ldots,x_r \}$ be an arbitrary subset of $\{ 1, 2, \ldots, \z_3-1 \} $ such that
$$ 1 \le x_1 <  \cdots < x_r < \z_3.$$
For a given $X_{ \z_3}$, we denote by $\mathcal{AO}^{X_{ \z_3}}_i$ the subset of $\mathcal{AO}_i$
$(i=1,2)$ satisfying the following condition:
\begin{center}
$\lambda \in \mathcal{AO}^{X_{ \z_3}}_i$ if and only if each part of $\lambda$ is congruent to $x_j$ modulo $\z_3$ for some $1 \le j \le r$.
\end{center}
Then we have
$$|\mathcal{AO}^{X_{ \z_3}}_1[m]|=|\mathcal{AO}^{X_{ \z_3}}_2[m]| \quad \text{ for all } m \in \Z_{\ge 0}.$$
\end{corollary}

\vskip 1em

\section{ Generalization of Bessenrodt's algorithm} \label{Sec: Bessenrodt's refinement}

In the previous section, we proved that $|\mathcal{AO}^{X_{ \z_3}}_1[m]|$
coincides with $|\mathcal{AO}^{X_{ \z_3}}_2[m]|$. In this section, we will
construct an explicit bijection between these sets by generalizing Bessenrodt's insertion algorithm
\cite{B91,B94,WHKM10}. Recall the fact that the map $\mathsf{P}$ is
injective in case of types $D^{(2)}_{n+1}$ and $A^{(2)}_{2n}$. Thus
we can use the results in \cite{B91,B94,WHKM10} efficiently for
types $D^{(2)}_{n+1}$ and $A^{(2)}_{2n}$.

\vskip 1em

This section is devoted to prove the following theorem:

\begin{theorem} \label{Thm: Bessenrodt's refinement}
For all $m \in \N$, there is a bijection
$$\Theta: \mathcal{AO}_1[m] \to \mathcal{AO}_2[m].$$
\end{theorem}

\subsection{Types of $D^{(2)}_{n+1}$ and $A^{(2)}_{2n}$.}

\subsubsection{$A^{(2)}_{2n}$.}

For $N \in \N$, let $X_{N}=\{ x_1,x_2,...,x_r \}$ be an arbitrary subset of $\{ 1,2,\cdots,N-1 \}$ such that
$$1 \le x_1 < x_2 < \cdots < x_r < N.$$

\begin{definition} \label{def: AOXN1}
Let $\mathcal{AO}_1^{X_{N}}[m]$ denote the subset of partitions of $m$ satisfying the following
conditions:
\begin{enumerate}
\item Each part is congruent to $0$ or some $x_i$ modulo $N$.
\item Only the multiples of $N$ can be repeated and the smallest part is strictly less than $N$.
\item The difference between two successive parts is at most $N$ and strictly less than $N$ if either part is
divisible by $N$.
\end{enumerate}
\end{definition}

\begin{definition} \label{def: AOXN2}
Let $\mathcal{AO}_2^{X_{N}}[m]$ denote the subset of partitions of $m$ satisfying the following
conditions:
\begin{enumerate}
\item Each part is congruent to some $x_i$ modulo $N$.
\item No part can be repeated.
\end{enumerate}
\end{definition}

\begin{theorem} \cite{AO91,B91} \label{Thm: Bess origin}
Let $m \in \Z_{\ge 0}$.
\begin{enumerate}
\item $|\mathcal{AO}_1^{X_{N}}[m]|=|\mathcal{AO}_2^{X_{N}}[m]|.$
\item There is an insertion algorithm which establishes an explicit bijection
\begin{equation} \label{eq: B insert}
\Theta: \mathcal{AO}_1^{X_{N}}[m] \to \mathcal{AO}_2^{X_{N}}[m].
\end{equation}
\end{enumerate}
\end{theorem}

In later subsections, we will generalize the insertion
algorithm \eqref{eq: B insert} given by Bessenrodt in \cite{B91}. Thus we do not give an
overview of \eqref{eq: B insert}. \vskip 1em \noindent \textbf{Proof of
Theorem \ref{Thm: Bessenrodt's refinement} } ($A^{(2)}_{2n}$) Note
that $\U=2n+1 \in \N$. If we choose $X_{\U}=\{ 1,2,3,\ldots, \U-1 \}$, one
can easily check that
$$  \text{ $\mathcal{AO}_i[m]$ coincides with $\mathcal{AO}_i^{X_{\U}}[m]$} \quad (i=1,2). $$
Thus we have a bijection between $\mathcal{AO}_1$ and $\mathcal{AO}_2$. Moreover, the bijection in
Theorem \ref{Thm: Bess origin} is a weight-preserving map \cite{B91}.

\subsubsection{$D^{(2)}_{n+1}$.} For $N \in \N$, let $X_{2N}=\{ x_1,x_2,...,x_r \}$ be an arbitrary subset of $\{ 1,2,\cdots,2N-1 \}$ such that
$$1 \le x_1 < x_2 < \cdots < x_r < 2N.$$

\begin{definition} Let $\mathcal{AO}_3^{X_{2N}}[m]$ denote the subset of partitions of $m$ satisfying the following
conditions:
\begin{enumerate}
\item Each part is congruent to $0$ or some $x_i$ modulo $2N$.
\item Only the multiples of $N$ can be repeated and the smallest part is less than $2N$.
\item The difference between two successive parts is at most $2N$ and strictly less than $2N$ if either part is
divisible by $N$.
\end{enumerate}
\end{definition}

\begin{definition} Let $\mathcal{AO}_4^{X_{2N}}[m]$ denote the subset of partitions of $m$ satisfying the following
conditions:
\begin{enumerate}
\item Each part is congruent to some $x_i$ modulo $2N$.
\item Only the multiples of $N$ can be repeated.
\end{enumerate}
\end{definition}

\begin{theorem} \label{Thm: AO3 to AO4} \cite{WHKM10}
For any $m \in \Z_{\ge 0}$, there is an insertion algorithm which establishes an explicit bijection
$$\Theta':\mathcal{AO}_3^{X_{2N}}[m] \leftrightarrow \mathcal{AO}_4^{X_{2N}}[m].$$
\end{theorem}

\noindent
\textbf{Proof of Theorem \ref{Thm: Bessenrodt's refinement} } ($D^{(2)}_{n+1}$)
 Note that $\U=2n+2=2(n+1)=2\V$. If we choose $X_{\U}=\{ 1,2,\ldots,\U-1 \} $, one can
easily check that
$$ \mathcal{AO}_1[m] =  \mathcal{AO}_3^{X_{\U}}[m], \quad \text{ for all } m \in \Z_{\ge 0}.$$
Thus it suffices to show that there is a bijection between
$\mathcal{AO}_4^{X_{\U}}[m]$ and $\mathcal{AO}_2[m]$. This
bijection can be constructed by using Sylvester's bijection
\cite{M84} between $\mathscr{O}$ and $\OS$.

We can extract an odd partition $\lambda$ from $\mu \in
\mathcal{AO}_4^{X_{\U}}[m]$ by the following algorithm $({\mathbf
D})$:
\begin{itemize}
\item[(${\mathbf D}$1)] Let $\mu \in \mathcal{AO}_4^{X_{\U}}[m]$. Set $\mu^{(0)}=\mu$ and $l=0$.
\item[(${\mathbf D}$2)] Find the maximal $i$ such that
$ \mu^{(l)}_i = \lambda_{l+1} \V \text{ for some } \lambda_{l+1} \in 2\Z_{\ge 0}+1$, and
set $$\mu^{(l+1)}=(\mu^{(l)}_1,\mu^{(l)}_2,\ldots,\mu^{(l)}_{i-1},\mu^{(l)}_{i+1},\ldots).$$
\item[(${\mathbf D}$3)] If there is no $j$ such that
$$ \mu^{(l+1)}_{j}= k \V \text{ for some } k \in 2\Z_{\ge 0}+1,$$
define $\underline{\mu}=\mu^{(l+1)}$ and terminate this algorithm. Otherwise, set $l=l+1$ and go to (${\mathbf D}$2).
\end{itemize}
This algorithm terminates in finitely many steps and
$\lambda:=(\lambda_l,\lambda_{l-1},\ldots,\lambda_1)$ is an odd
partition.

Then by the Sylvester's bijection, we have a strict partition
$\lambda'=(\lambda'_1,\ldots,\lambda'_r)$ and
$$ \mu':=(\lambda'_r \cdot \V) \hookrightarrow (\lambda'_{r-1} \cdot \V) \hookrightarrow
\cdots (\lambda'_1 \cdot \V) \hookrightarrow \underline{\mu} \in \mathcal{AO}_2[m].$$
In conclusion, we can construct a bijection $\mathcal{AO}_1 \to \mathcal{AO}_2$ which preserves weight.

\begin{definition} Let $\mathcal{AO}_5^{X_{2N}}[m]$ denote the set of partitions of $m$ satisfying the following
conditions:
\begin{enumerate}
\item Each part is congruent to some $x_i$ modulo $2N$ and to $0$ modulo $2N$ if $ N  \in X_{2N}$.
\item No part can be repeated.
\end{enumerate}
\end{definition}

\begin{corollary} \label{Thm: AO3,AO4 to 5}
For any $m \in \Z_{\ge 0}$ and  $X_{2N}$, there is an insertion algorithm which establishes an explicit bijection
$$\Theta'':\mathcal{AO}_4^{X_{2N}}[m] \to \mathcal{AO}_5^{X_{2N}}[m],$$
and hence
\begin{enumerate}
\item there exists a bijection $\Theta'' \circ \Theta':\mathcal{AO}_3^{X_{2N}}[m] \to \mathcal{AO}_5^{X_{2N}}[m],$
\item $|\mathcal{AO}_3^{X_{2N}}[m]|=|\mathcal{AO}_4^{X_{2N}}[m]|=|\mathcal{AO}_5^{X_{2N}}[m]|.$
\end{enumerate}
\end{corollary}

\subsection{For the other types.} From now on, we define the sets $\mathcal{AO}_{1^o}$ and $\mathcal{AO}_{2^s}$ which correspond
to $\mathcal{AO}_{1}$ and $\mathcal{AO}_{2}$, respectively. In particular,
the case when $\g_n=B^{(1)}_n$ and $\Lambda=\Lambda_{n}$, we just set $\mathcal{AO}_{1^o} \seteq \mathcal{AO}_{1}$
and $\mathcal{AO}_{2^s} \seteq \mathcal{AO}_{2}$.

Let $\lambda$ be an element of $\mathcal{AO}_i$ ($i=1,2$).
For a subsequence
$\underline{\lambda}=(\lambda_i,\ldots,\lambda_{i+t})$, we say that
$\underline{\lambda}$ is a {\it $\U$-subsequence} of $\lambda$ if
it satisfies
\begin{itemize}
\item for all $i \le j \le i+t$, $\lambda_j=l \cdot \U$ or $l \cdot \overline{\U}$ ($l \in \N$),
\item for all $i < j \le i+t$, $\lambda_j-\lambda_{j-1}=0, \U$ or $\overline{\U}$.
\end{itemize}

For $k \in \N$, we say that $\underline{\lambda}$ is the {\it $k
\cdot \U $-component} of $\lambda$ if it is maximal among
$\U$-subsequence of $\lambda$ containing $k\cdot \U$ or $k
\cdot \overline{\U}$ as a part.

For $\lambda \in \mathcal{AO}_1$, the $k
\cdot \U$-component $\underline{\lambda}$ of $\lambda$ can be written
by one of the following:
\begin{itemize}
\item[(i)]
$\underline{\lambda}=( (k_l \cdot \overline{\U})^{t_{l}},
(k_{l-1} \cdot \U)^{t_{l-1}},\ldots,
(k_2 \cdot \U)^{t_{2}},
(k_1 \cdot \overline{\U})^{t_{1}})$,
\item[(ii)]
$\underline{\lambda}=( (k_l \cdot \overline{\U})^{t_{l}},
(k_{l-1} \cdot \U)^{t_{l-1}},\ldots,
(k_2 \cdot \overline{\U})^{t_{2}},
(k_1 \cdot \U)^{t_{1}})$,
\item[(iii)]
$\underline{\lambda}=( (k_l \cdot \U)^{t_{l}},
(k_{l-1} \cdot \overline{\U})^{t_{l-1}},\ldots,
(k_2 \cdot \U)^{t_{2}},
(k_1 \cdot \overline{\U})^{t_{1}})$,
\item[(iv)]
$\underline{\lambda}=( (k_l \cdot \U)^{t_{l}},(k_{l-1} \cdot \overline{\U})^{t_{l-1}},\ldots, (k_2 \cdot \overline{\U})^{t_{2}},
(k_1 \cdot \U)^{t_{1}})$,
\end{itemize}
where $k_l>0$ and $k_{i+1}=k_i-1$.

Let $\widetilde{\lambda}$ be the
sequence in $\OP^o$ associated with $\underline{\lambda}$ given by
\begin{itemize}
\item[(i)]
$\widetilde{\lambda}=( k_l \cdot \overline{\U},(k_l \cdot \U)^{t_{l}-1},k_{l-1}\cdot \overline{\U},(k_{l-1} \cdot \U)^{t_{l-1}-1},\ldots,k_2\cdot \overline{\U},
(k_2 \cdot \U)^{t_{2}-1},k_1\cdot \overline{\U},(k_1 \cdot \U)^{t_{1}-1})$,
\item[(ii)]
$\widetilde{\lambda}=( k_l\cdot \overline{\U},(k_l \cdot \U)^{t_{l}-1},k_{l-1}\cdot \overline{\U},(k_{l-1} \cdot \U)^{t_{l-1}-1},\ldots, k_2\cdot \overline{\U},
(k_2 \cdot \U)^{t_{2}-1},k_1\cdot \overline{\U}, (k_1 \cdot \U)^{t_{1}-1})$,
\item[(iii)]
$\widetilde{\lambda}=( (k_l \cdot \U)^{t_{l}},k_{l-1}\cdot \overline{\U},(k_{l-1} \cdot \U)^{t_{l-1}-1},\ldots, k_2 \cdot \overline{\U},
(k_{2}\cdot \U)^{t_{2}-1}, k_1\cdot \overline{\U},(k_1 \cdot \U)^{t_{1}})$,
\item[(iv)]
$\widetilde{\lambda}=( (k_l \cdot \U)^{t_{l}},k_{l-1}\cdot \overline{\U},(k_{l-1} \cdot \U)^{t_{l-1}-1},\ldots,
k_2\cdot \overline{\U},(k_2 \cdot \U)^{t_{2}-1},k_1\cdot \overline{\U},(k_1 \cdot \U)^{t_{1}})$.
\end{itemize}

For $\lambda \in \mathcal{AO}_1$, we define
$\mathsf{P}^o(\lambda)$ to be the element of $\OP^o$ given by replacing all $k
\cdot \U$-components $\underline{\lambda}$ of $\lambda$ with their
associated sequences $\tilde{\lambda}$ in $\OP^o$. Then we have an
injective map $\mathsf{P}^o: \mathcal{AO}_1 \to \OP^o$ and can
recover $\lambda$ from $\mathsf{P}^o(\lambda)$.

\begin{definition} If $\g_n$ is of type $A^{(2)}_{2n-1}$, $B^{(1)}_{n}$ ($\Lambda \neq \Lambda_n$) and $D^{(1)}_n$,
we denote by
$\mathcal{AO}_{1^o}$ the set of all
$\mathsf{P}^o(\lambda)$ for $\lambda \in \mathcal{AO}_1$.
\end{definition}

For $\lambda \in \mathcal{AO}_2$, the $k \cdot \U$-component $\breve{\lambda}$
of $\lambda$ can be written by one of the
following:
\begin{itemize}
\item[(a)]
$\breve{\lambda}=( (k_{2l+2} \cdot \overline{\U})^{t_{2l+2}},
(k_{2l+1} \cdot \overline{\U})^{t_{2l+1}},\ldots,
(k_2 \cdot \overline{\U})^{t_{2}},
(k_1 \cdot \overline{\U})^{t_{1}})$,
\item[(b)]
$\breve{\lambda}=( (k_{2l+1} \cdot \overline{\U})^{t_{2l+1}},
(k_{2n} \cdot \overline{\U})^{t_{2n}},\ldots,
(k_2 \cdot \overline{\U})^{t_{2}},
(k_1 \cdot \overline{\U})^{t_{1}})$,
\end{itemize}
for some $l \in \Z_{\ge 0}$.

Let $\ddot{\lambda}$ be the sequence in $\OP^2$ associated to the $k \cdot \U$-component
$\breve{\lambda}$ which is defined as follows:
\begin{itemize}
\item[(a)]
$\ddot{\lambda}=( (k_{2l+2} \cdot \U)^{t_{2l+2}},
(k_{2l+1} \cdot \overline{\U})^{t_{2l+1}},\ldots,
(k_2 \cdot \U)^{t_{2}},
(k_1 \cdot \overline{\U})^{t_{1}})$,
\item[(b)]
$\ddot{\lambda}=( (k_{2l+1} \cdot \overline{\U})^{t_{2l+1}},
(k_{2n} \cdot \U)^{t_{2l}},\ldots,
(k_2 \cdot \U)^{t_{2}},
(k_1 \cdot \overline{\U})^{t_{1}})$.
\end{itemize}

For a given partition $\lambda \in \mathcal{AO}_2$, we
define $\mathsf{P}^s(\lambda)$ to be the element of $\OP^2$ given by replacing all
$k \cdot \U$-components $\breve{\lambda}$ of $\lambda$ with their
associated sequences $\ddot{\lambda}$ in $\OP^o$. Then we have an injective map $\mathsf{P}^s:
\mathcal{AO}_2 \to \OP^2$ and can recover
$\lambda$ from $\mathsf{P}^s(\lambda)$.

\begin{definition} If $\g_n$ is of type $A^{(2)}_{2n-1}$, $B^{(1)}_{n}$ ($\Lambda \neq \Lambda_n$) and $D^{(1)}_n$,
we denote by $\mathcal{AO}_{2^s}(\Lambda)$ the set of all
$\mathsf{P}^s(\lambda)$ for $\lambda \in \mathcal{AO}_2(\Lambda)$.
\end{definition}

\begin{proposition} Let $\mathcal{AO}_{1 \cap 2^s}\seteq \mathcal{AO}_{1} \cap \mathcal{AO}_{2^s}$.
For all $m \in \N$, there are injections
$$\overline{\Theta}[m]:\mathcal{AO}_{2^s}[m] \to \displaystyle
\bigsqcup_{k \ge 0, \ m-k \U \ge 0} (\mathcal{AO}_{1 \cap
2^s}[m-k \U ] \times \OP[k])$$ given by
$$ Y \longmapsto (Y',\lambda) \quad \text{ such that } \quad \ell(\lambda) \le \ell(Y')=\ell(Y).$$
\end{proposition}

\begin{proof}
Let $\Theta[m]$ be a map from $\mathcal{AO}_{2^s}[m]$ to
$\displaystyle \bigsqcup_{ k \ge 0, \ m-k  \U \ge 0} \mathcal{AO}_{1
\cap 2^s}[m-k\U]$ given by the following algorithm $({\mathbf D}')$:
\begin{itemize}
\item[(${\mathbf D}'$1)] Let $Y \in \mathcal{AO}_{2^s}[m]$ be given. Set $Y^{(0)}=Y$ and $l=0$. If
$Y \in \mathcal{AO}_{1\cap 2^s}[m]$, then terminate this algorithm and set $Y'=Y$.
\item[(${\mathbf D}'$2)] Find the maximal $i$ such that
$$\begin{cases}
y^{(l)}_{i-1}-y^{(l)}_{i} \succ \U & \text{ if } \g_n= A^{(2)}_{2n-1},\\
y^{(l)}_{i-1}-y^{(l)}_{i} \succ \U, \text{ or }
y^{(l)}_{i-1}-y^{(l)}_{i}=\U \text{ and } y^{(l)}_{i-1} \equiv \z_3 \ {\rm mod} \ \U
& \text{ if } \g_n= B^{(1)}_{n}  \text{ and } \Lambda \neq \Lambda_n , \\
y^{(l)}_{i-1}-y^{(l)}_{i} \succ \U, \text{ or }
y^{(l)}_{i-1}-y^{(l)}_{i}=\U \text{ and } y^{(l)}_{i-1} \equiv \z_3 \text{ or } \overline{\z_3} \ {\rm mod}
\ \U &
\text{ otherwise.}
\end{cases}$$
Then set $$Y^{(l+1)}\seteq(y^{(l)}_1-\U,y^{(l)}_2-\U, \ldots,
y^{(l)}_{i-1}-\U,y^{(l)}_{i},
  y^{(l)}_{i+1}, \ldots).$$
\item[(${\mathbf D}'$3)] If $Y^{(l+1)} \in \mathcal{AO}_{1 \cap 2^s}$, then $Y'=Y^{(l+1)}$ and terminate this
algorithm. Otherwise, set $l=l+1$ and go to (${\mathbf D}'$2).
\end{itemize}

This algorithm terminates in finitely many steps and we have
\begin{itemize}
\item $k\seteq\dfrac{\Sigma_Y-\Sigma_{Y'}}{\U} \in \Z_{\ge 0}$, $Y' \in \mathcal{AO}_{1\cap 2^s}[m-k \U]$,
\item $\lambda_i \seteq \dfrac{y_i-y'_i}{\U} \in \Z_{\ge 0}$,
      $\lambda \seteq (\lambda_1,\lambda_2,\ldots) \vdash k$,
      $\ell(\lambda) \le t = \ell(Y)=\ell(Y')$ and $Y=Y'+ \U \cdot \lambda$.
\end{itemize}
Thus we get a map
$$ \overline{\Theta}[m]:\mathcal{AO}_{2^s}[m] \to
\bigsqcup_{ k \ge 0, \ m-k \U  \ge 0} \mathcal{AO}_{1\cap
2^s}[m-k \U] \times \OP[k]$$
 given by $$ Y \mapsto (Y',\lambda) \quad \text{ such that } \quad
\ell(\lambda) \le \ell(Y')=\ell(Y).$$ Conversely, for given $Y'=(y'_1,\ldots,y'_t) \in \mathcal{AO}_{1\cap 2^s}[m']$ and $\lambda'$ satisfying
\begin{itemize}
\item $m' < m$ and $\dfrac{m-m'}{\U} \in \N$,
\item $\lambda'=(\lambda'_1,\ldots,\lambda'_{t'}) \vdash \dfrac{m-m'}{\U}$ and $t' \le t$,
\end{itemize}
there is a unique $Y\seteq Y'+\U \cdot \lambda' \in \mathcal{AO}_{2^s}[m]$. For a fixed $Y'=(y'_1,\ldots,y'_t)
\in \mathcal{AO}_{1 \cap 2^s}[m']$, let $\mathcal{AO}_{2^s}[m';Y']$
denote the set of corresponding sequences in
$\mathcal{AO}_{2^s}[m]$. Then $|\mathcal{AO}_{2^s}[m';Y']|$
coincides with the number of partitions of $\dfrac{m-m'}{\U}$
whose length are less than or equal to $t$.
\end{proof}

\begin{proposition} \label{Prop: Algorithm E,F}
Let $\mathcal{AO}_{1^o \cap 2}\seteq \mathcal{AO}_{1^o} \cap \mathcal{AO}_{2}$.
For all $m \in \N$, there are injective maps given as follows:
$$\overline{\Theta}[m]:\mathcal{AO}_{1^o}[m] \to \displaystyle
\bigsqcup_{k \ge 0, \ m-k \U \ge 0} (\mathcal{AO}_{1^o \cap
2}[m-k \U ] \times \OP[k])$$ given by
$$ Y \longmapsto (Y',\lambda) \quad \text{ such that } \quad  \lambda_1 \le \ell(Y').$$
\end{proposition}

\begin{proof}
For a given $Y \in \mathcal{AO}_{1^o}[m]$, we define $Y'$ by the following algorithm $({\mathbf E'})$:
\begin{itemize}
\item[(${\mathbf E'}$1)] First, we remove all parts which are $\N$-multiples of $\U$ and denoted by $Y^{(0)}$.
If $Y^{(0)} \in \mathcal{AO}_{1^o \cap 2}$, define $Y'=Y^{(0)}$ and terminate this algorithm. Otherwise, set $l=0$.
\item[(${\mathbf E'}$2)] Find the maximal $i$ such that
\begin{align*}
&y^{(l)}_{i-1}-y^{(l)}_i \succ \U,  \text{ or }
y^{(l)}_{i-1}-y^{(l)}_i = \U  \text{ and } y^{(l)}_{i-1} = a \cdot \U \text{ for some } a \in \Z_{\ge 1}.
\end{align*}
Set
$$Y^{(l+1)}\seteq(y^{(l)}_1- \U ,y^{(l)}_2- \U ,...,y^{(l)}_{i-1}-\U ,y^{(l)}_i,y^{(l)}_{i+1},...).$$
\item[(${\mathbf E'}$3)] If $Y^{(l+1)} \in \mathcal{AO}_{1^o \cap 2}$, then
define $Y'=Y^{(l+1)}$ and terminate this algorithm. Otherwise, set $l=l+1$ and go to $({\mathbf E'}2)$.
\end{itemize}

Then this algorithm terminates in finitely many steps and $Y' \in
\mathcal{AO}_{1^o \cap 2}[m-k \U]$ (for some $k \in \Z_{\ge 0}$) is
a two-colored partition which satisfies the following formula:
\begin{align*}
&\re Y= (y'_1,\ldots,y'_{r_0}, (\U)^{l_1},y'_{r_0+1}+\ve_1 \cdot \U, \ldots, y'_{r_1}+\ve_1 \cdot \U,
(2 \cdot \U)^{l_2}, \ldots, \\
& \qquad\qquad\qquad\qquad\qquad\qquad\qquad\qquad\qquad
\ldots,y'_{r_{e-2}+1}+\ve_{e-1} \cdot \U,\ldots,
y'_{r_{e-1}}+\ve_{e-1} \cdot \U,(e \cdot \U)^{l_e}),
\end{align*}
where
\begin{itemize}
\item $\re Y'=(y'_1,\ldots,y'_t)$ such that, for all $i,a \in \Z_{\ge 1}$, $y'_i \neq a \cdot \U$,
\item $r_{e-1}=t$,
\item $0 \le \varepsilon_i \le \varepsilon_{i+1} \le \varepsilon_i+1 \le i+1$ and $l_i \ge 0$
for all $i$.
\end{itemize}
We set $\varepsilon_0=0$ and $r_i=0=\varepsilon_i$ for
all $i \ge e$.

We define $\re\lambda \seteq(1^{\mu_1},2^{\mu_2},\ldots,t^{\mu_t})$  by the following formula:
\begin{align} \label{eq: getting lambda}
 \mu_i=l_i-(\varepsilon_i-\varepsilon_{i-1})+
\sum^{i}_{ j=1, \ j+t-r_{j-1}=i  }
(\varepsilon_j-\varepsilon_{j-1}), \quad \text{for $1 \le i \le t$.}
\end{align}
We can interpret this process by removing some columns which are multiples of $\U$, and removing some $\mathsf{L}$-hooks of
$\U$'s. For this purpose, we introduce a $\U$-modular Young diagram of $Y \in \mathcal{AO}_{1^o}$
by showing the following example.

\begin{example} \label{Ex: example1}
For $\g_n=A^{(2)}_{7}$ and
$Y=(33,31,\overline{28},28,\overline{21},21,15,9,7,1) \in
\mathcal{AO}_{1^o}$, the $7$-modular diagram is
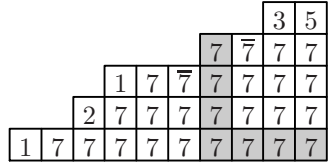
\begin{center}
\modularone
\end{center}
Here the shaded part denotes an $\mathsf{L}$-hook.
\end{example}
Then the way of obtaining ($Y',\lambda$) from $Y$ can be also interpreted by the following algorithm
$({\mathbf E})$:
\begin{itemize}
\item[(${\mathbf E}$i)] $(1 \le i \le t)$ First, remove columns of $i \cdot \U$ for $l_i-(\varepsilon_i-
\varepsilon_{i-1})$ times. Then also remove $\mathsf{L}$-hook of $\U$ whose length is $i$ if there is any.
This case occurs only if
$j+(t-r_j-1)=i$ and $\varepsilon_j-\varepsilon_{j-1}=1$ for some $j < i$. Set the obtained Young wall as $Y^{(i)}$.
\end{itemize}
This algorithm terminates at $t$th step and $Y':=Y^{(t)} \in
\mathcal{AO}_{1^o \cap 2}[\Sigma_{Y}-\Sigma_\lambda \U].$

For $k \in \Z_{\ge 0}$, define the {\it right insertion} of $(k\cdot \U)^{j}$
into $\lambda \in \mathcal{AO}_{1^o \cap 2}$, denoted by $\lambda \hookleftarrow (k\cdot \U)^j$, as follows:
\begin{equation} \label{eq: right inserting}
\lambda \hookleftarrow (k\cdot \U)^j \seteq(\lambda_1,\lambda_2,...,\lambda_i,
\underbrace{k \cdot \U ,...,k\cdot \U}_{j},\lambda_{i+1},..)
\qquad \text{ for } \lambda_i \succeq k \cdot \overline{\U} \succ \lambda_{i+1}.
\end{equation}

Now, we give a reverse algorithm which is a kind of insertion algorithm. Assume we
have $\re\lambda=(1^{\mu_1},2^{\mu_2},\ldots,t^{\mu_t})$ and $Y^{(t)} \in \mathcal{AO}_{1^o \cap 2}$.

The way of obtaining $Y$ from $Y'=Y^{(t)}$ and $\lambda$ can be explained by the following algorithm
($\mathbf F$): \\
We start from $i=t$ to $i=1$.
\begin{itemize}
\item[($\mathbf F$i)] If $Y^{(i)} \hookleftarrow (i \cdot \U)^{\mu_i} \in \mathcal{AO}_{1^o}$, set
$Y^{(i-1)}\seteq Y^{(i)} \hookleftarrow (i \cdot \U)^{\mu_i}.$
Otherwise, insert as many $\mathsf{L}$-hooks of length $i$ as necessary and then
insert columns of $i \cdot \U$ until the number of columns and $\mathsf{L}$-hooks together is $\mu_i$.
Set the resulting Young wall as $Y^{(i-1)} \in \mathcal{AO}_{1^o}$.
\end{itemize}
This insertion algorithm always works for $Y^{(i)} \in \mathcal{AO}_{1^o}$ ($0 \le i <t$)
and $Y:=Y^{(1)}$.
\end{proof}

\noindent
\textbf{Proof of Theorem \ref{Thm: Bessenrodt's refinement} }
For $t \in \Z_{\ge 0}$, there is a bijection
$$\{ \lambda \in \OP \ | \ \ell(\lambda) \le t \} \to \{ \lambda \in \OP \ | \ \lambda_1 \le t \} \quad \text{ given by }
\quad \lambda \mapsto \lambda^{{\rm tr}}.$$ Thus, for $m \in \Z_{\ge 0}$, we have a bijection
$$ \mathcal{AO}_1[m] \longrightarrow \mathcal{AO}_2[m] $$
given by
$$ Y_1 \overset{\mathsf{P}^o}{\longmapsto} Y_2 \overset{(\mathbf{E})}{\longmapsto} (Y_2',\lambda)
\overset{(\mathsf{P}^s,{ \ }^{{\rm tr}})}{\longmapsto}
(Y_3'\seteq \mathsf{P}^s(Y_2'),\lambda^{{\rm tr}})
\overset{Y_3'+\U \cdot \lambda^{{\rm tr}}}{\longmapsto} Y_3 \overset{{\mathsf{P}^{s}}^{-1}}{\longmapsto}  Y_4.$$

To express the algorithm given in Proposition \ref{Prop: Algorithm E,F} more concretely,
we give the following example.

\begin{example} In this example, we keep the notation given above and assume that $\g_{n}=A^{(2)}_{7}$.
\begin{enumerate}
\item Let $Y_1=(33,31,\overline{28}^2,21^2,15,9,7,1)$ be a two-colored partition in $\mathcal{AO}_{1}[194]$. Then we have
$$Y_2 \seteq \mathsf{P}^o(Y_1)=(33,31,\overline{28},28,\overline{21},21,15,9,7,1) \quad \text{ and } \quad Y'_2=
(26,24,\overline{21},\overline{14},8,2,1).$$
To get the parameters $l_i$, $\ve_i$ and $r_i$ appearing in $\eqref{eq: getting lambda}$, we write
$ \re Y_2$ as follows:
$$ \re Y_2=(1, 7,(2+7),(8+7),21,(\overline{14}+7), 28, (\overline{21}+7), (24+7), (26+7)).$$
Hence we obtain
$$
\begin{tabular}{c|c c c c c c c c}
$i$ & 0 & 1 & 2 & 3 & 4 & 5 & 6 & 7   \\ \hline
$l_i$ & - & 1 & 0 & 1 & 1 & 0 & 0 & 0   \\
$\varepsilon_i$ & 0 & 1 & 1 & 1 & 1 & 0 & 0 & 0   \\
$r_i$ & 1 & 2 & 3 & 4 & 7 & 0 & 0 & 0   \\  \end{tabular}
$$
Thus, by $\eqref{eq: getting lambda}$, ${}^{\triangleright}\lambda=(3,4,7)$.
Now we run through the
combinatorial algorithm starting with $Y$. The $7$-modular Young diagram for $Y_2=Y^{(0)}$
is depicted by
\begin{center}
\modulartwo
\end{center}
Here $\mu_1=0$, since we can not remove the column $\Sevenone$ without violating the $\mathcal{AO}_{1^o}$-
condition.
As there is no column $\Seventwo^{{\rm tr}}$ and also no L-hook of length $2$, we also have
$\mu_2=0$ and $Y^{(0)}=Y^{(1)}=Y^{(2)}$.
Then we can remove one column $\Seventhree^{{\rm tr}}$, but there is no $\mathsf{L}$-hook of length $3$ which can be removed. So
$\mu_3=1$ and $Y^{(3)}$ is
\begin{center}
\modularthree
\end{center}
At this step, we can remove the column $\Sevenfour^{{\rm tr}}$, and again there is no extra $\mathsf{L}$-hook of length $4$.
Thus $\mu_4=1$ and $Y^{(4)}$ is
\begin{center}
\modularfour
\end{center}
There is neither a column nor an admissible $\mathsf{L}$-hook of length 5 and
6, so $\mu_5=\mu_6=0$ and $Y^{(4)}=Y^{(5)}=Y^{(6)}$. At the seventh
step, there is no column of length 6 but there is an admissible
$\mathsf{L}$-hook of length $7$ which is shaded. Thus $\mu_7=1$ and
$Y_2'=Y^{(7)}=(26,24,\overline{21},\overline{14},8,2,1)$ is the desired one.
\item Let $Y_1=(33,31,28^2,\overline{21}^2,15,9,7,1)$ be a two-colored partition in $\mathcal{AO}_{1}[194]$. Then we have
$$Y_2 \seteq \mathsf{P}^o(Y_1)=(33,31,28,28,\overline{21},21,15,9,7,1) \quad \text{ and } \quad Y'_2=
(19,17,\overline{14},9,2,1).$$
To get the parameters $l_i$, $\ve_i$ and $r_i$ appearing in $\eqref{eq: getting lambda}$, we write
$ \re Y_2$ as follows:
$$ \re Y_2=(1,7,(2+7),(8+7),21,(\overline{14}+7),28^2, (17+2 \cdot 7), (19+2 \cdot 7)).$$
Hence we obtain
$$
\begin{tabular}{c|c c c c c c c }
$i$ & 0 & 1 & 2 & 3 & 4 & 5 & 6    \\ \hline
$l_i$ & - & 1 & 0 & 1 & 2 & 0 & 0    \\
$\varepsilon_i$ & 0 & 1 & 1 & 1 & 2 & 0 & 0    \\
$r_i$ & 1 & 2 & 3 & 4 & 6 & 0 & 0    \\  \end{tabular}
$$
Thus, by the $\eqref{eq: getting lambda}$,
${}^{\triangleright}\lambda=(3,4,6^2)$. Now we run through the
combinatorial algorithm starting with $Y$. The $7$-modular Young
diagram for $Y_2=Y^{(1)}$ is depicted as
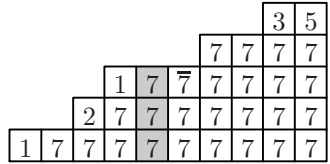
\begin{center}
\modularsix
\end{center}
By the same reason given in $(1)$, we have $\mu_1=\mu_2=0$,
$Y^{(1)}=Y^{(2)}$ and $\mu_3=1$. Then $Y^{(3)}$ is depicted by
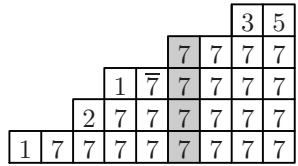
\begin{center}
\modularseven
\end{center}
At this step, we can remove the column $\Sevenfour^{{\rm tr}}$, and again there is no extra $\mathsf{L}$-hook of length $4$.
Thus $\mu_4=1$ and $Y^{(4)}$ is
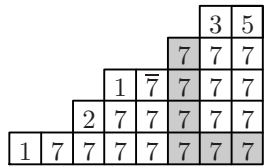
\begin{center}
\modulareight
\end{center}
There is neither a column nor an admissible $\mathsf{L}$-hook of length 5
$\mu_5=0$ and $Y^{(4)}=Y^{(5)}$.
However, there is one $\mathsf{L}$-hook of length 6 which is shaded.
After we remove the $\mathsf{L}$-hook, we have another $\mathsf{L}$-hook of length 6 by the following:
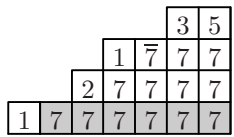
\begin{center}
\modularnine
\end{center}
Thus we have $\mu_6=2$ and $Y_2'=Y^{(6)}=(19,17,\overline{14},8,2,1)$ is the desired one.
\item For given $Y_2'=(19,17,\overline{14},9,2,1) \in \mathcal{AO}_{1^o \cap 2}[61]$ and
$\re \lambda=(3,4,6^2) \vdash 19$, we will get a $Y_2 \in
\mathcal{AO}_{1^o}[194]$ by the following way. $Y^{(6)}=Y_2'$ can be
depicted by follows:
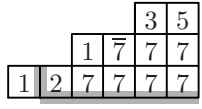
\begin{center}
\modularNine
\end{center}
Then there is only one $\mathsf{L}$-hook of length $6$. Thus we have
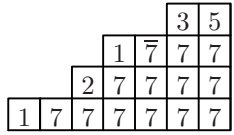
\begin{center}
\modularEight
\end{center}
At this step, there seem to be two choices of $\mathsf{L}$-hook of length $6$:
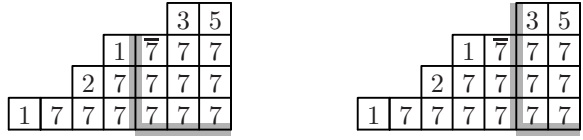
\begin{center}
\modularSeven
\end{center}
After inserting these two $\mathsf{L}$-hooks of length 6, we have
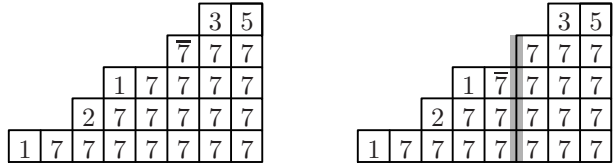
\begin{center}
\modularSix
\end{center}
But the first one is not contained in $\mathcal{AO}_{1^o}$. Thus there is only one $\mathsf{L}$-hook actually.
Hence $Y^{(5)}$ is the second one. Since the column $\Sevenfour^{{\rm tr}}$
can be inserted, so $Y^{(3)}$ is
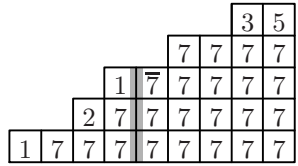
\begin{center}
\modularFive
\end{center}
Finally, we have $Y_2=(33,31,28^2,\overline{21},21,15,9,7,1)$, which
is the inverse of (2) in this example.
\item Considering (1) in this example, for $Y_1=(33,31,28^2,\overline{21}^2,15,9,7,1)$ in $\mathcal{AO}_{1}[194]$,
we have $Y_2'=(26,24,\overline{21},\overline{14},8,2,1)$ and $\lambda=(7,4,3)$. Then
$\lambda^{{\rm tr}}=(3,3,3,2,1,1,1)$ and hence we obtain $Y_4=(47,45,\overline{42},\overline{28},16,9,8)$
in $\mathcal{AO}_{2}[194]$.
\end{enumerate}
\end{example}

\begin{corollary} For the types of $B^{(1)}_n$, $D^{(1)}_n$  and $m \in  \Z_{\ge 0}$, there is a bijection
$$ \overline{\Theta}[m]: \mathcal{AO}_1[m] \leftrightarrow
\bigsqcup_{  k \ge 0 , \ m-k \z_1 \ge 0}  \OS'[m-k \z_1] \times \OS[k] \ \text{where} \
\begin{cases} \OS'=\OS_{(\U)} &\text{ if } \g_n= B^{(1)}_{n} \text{ and } \Lambda=\Lambda_n, \\
              \OS'=\OS &\text{ otherwise}.\end{cases}$$
\end{corollary}

\begin{proof}
By composing Algorithm $({\mathbf C})$ with $\Theta$, we have
the desired bijection between two sets.
\end{proof}

By the same argument in Corollary \ref{Cor: more general AO id}, we
have the following corollary:
\begin{corollary}
Let $X_{ \z_3}=\{ x_1,\ldots,x_r \}$ be an arbitrary subset of $\{ 1,2,\ldots,\z_3-1 \}$ such that $1 \le x_1 <  \cdots <
x_r < \z_3$. We denote by $\mathcal{AO}^{X_{ \z_3}}_i$ the subset of $\mathcal{AO}_i$
$(i=1,2)$ whose elements contain only parts congruent to $x_j$ for some $1 \le j \le r$. Then there is a bijection
$$\mathcal{AO}^{X_{ \z_3}}_1[m] \to \mathcal{AO}^{X_{ \z_3}}_2[m] \quad \text{ for all } m \in \Z_{\ge 0}.$$
\end{corollary}

\begin{remark} From Theorem \ref{Thm: Bessenrodt's refinement}, we can deduce the following facts:
\begin{enumerate}
\item For the types of $D^{(2)}_{n+1}$ and $A^{(2)}_{2n}$,
we can give crystal structures on $\mathcal{AO}_2(\Lambda)$ which are isomorphic to
 $B(\Lambda)$ by using the bijection $\Theta$. Here, the weight of $\lambda \in \mathcal{AO}_2(\Lambda)$ is defined by the way given in Remark \ref{rmk: injective p}
and Kashiwara operators $\tilde{f}_i$ $(i \in I)$ are defined as follows:
 $$  \tilde{f}_i(\lambda) =\lambda' \ \ \text{ if and only if } \ \
  \tilde{f}_i(Y) =  Y' \ \ \text{ for $\Theta(Y)=\lambda$, $\Theta(Y')=\lambda'$ and $Y,Y' \in \mathtt{Y}(\Lambda)$}.$$
In Appendix \ref{Appen: crystal}, we will give an example for the crystal structure on $\OS$ of type $D^{(2)}_3$.
\item Note that the virtual character of $\mathtt{Y}(\Lambda)$ and $\mathcal{AO}_2(\Lambda)$ are different for
the types $A^{(2)}_{2n-1}$, $B^{(1)}_{n}$ and $D^{(1)}_{n}$
whenever we assign a weight of $\lambda \in \mathcal{AO}_2(\Lambda)$ in the way of Remark \ref{rmk: injective p}. However, through the bijection $\Theta$, we can
assign the weight of $\lambda \in \mathcal{AO}_2(\Lambda)$ in an alternative way as follows:
$$\text{$ \wt(\lambda) \seteq \wt(Y)$  for $\Theta(Y)=\lambda$, $Y \in \mathtt{Y}(\Lambda)$}.$$
Then it gives a crystal structure on $\mathcal{AO}_2(\Lambda)$ and hence (pair of) the (sub)set(s) of strict partitions
$$
\begin{cases}
\mathscr{S} & \text{ if } \g=A^{(2)}_{2n-1}, \\
\mathscr{S}_{(\U)} \times \mathscr{S} & \text{ if } \g=B^{(1)}_{n} \text{ and } \Lambda=\Lambda_n, \\
\mathscr{S} \times \mathscr{S} & \text{ otherwise,}
\end{cases}
$$
as in $(1)$.
\end{enumerate}
\end{remark}

\appendix

\section{The crystal structure on $\OS$ of type $D^{(2)}_3$} \label{Appen: crystal}
\begin{center}
\GraphF
\end{center}


\bibliographystyle{amsplain}


\end{document}